\documentclass[a4paper,11pt]{article}
\usepackage{aligned-overset}
\usepackage{amsmath,cancel}
\usepackage{amsmath,amssymb,enumitem,multicol}
\usepackage[utf8]{inputenc}
\usepackage[T1]{fontenc}
\usepackage[hmargin={20mm,30mm},vmargin={30mm,35mm}]{geometry}
\usepackage{authblk}

\usepackage{caption}
\usepackage{subcaption}
\usepackage{amssymb}
\usepackage{amsthm}
\usepackage{parskip}
\usepackage[autostyle]{csquotes}
\usepackage{mathrsfs}
\usepackage{multirow}
\usepackage{mathtools}
\usepackage{csquotes}
\usepackage[english]{babel}
\usepackage{graphicx}
\graphicspath{{./Figures/}}
\usepackage{xcolor}
\usepackage{epstopdf}
\usepackage{esint}
\usepackage{tikz}
\usepackage{pgfplots}
\usepgfplotslibrary{groupplots}
\pgfplotsset{compat=newest}
\usepackage[colorlinks,allcolors={blue}]{hyperref}
\usepackage{placeins}

\newcounter{example}[section]

\newtheorem{theorem}{Theorem}[section]

\newtheorem{lemma}[theorem]{Lemma}
\newtheorem{prop}[theorem]{Proposition}
\theoremstyle{definition}
\newtheorem{defn}{Definition}[section]

\newtheorem{remark}{Remark}[section]

\newtheorem{exmp*}{Example}
\theoremstyle{remark}

\newcommand{\inv}{^{\raise box{.2ex}{$\scriptscriptstyle-1$}}}

\newcommand{\norm}[2]{\|#2\|_{#1}}

\newcommand{\vvvert}{\vert\kern-0.25ex\vert\kern-0.25ex\vert}

\newcommand{\uvec}[1]{\underline{\mathbf{#1}}}
\newcommand{\bvec}[1]{\boldsymbol{#1}}
\newcommand{\bo}[1]{\mathbf{#1}}
\newcommand{\bx}{\bvec{x}}
\newcommand{\bos}[1]{\boldsymbol{#1}}
\newcommand{\up}[1]{\Upsilon^2_{#1}}
\newcommand{\dgr}[1]{\mathbb G^1_{#1}}
\newcommand{\mI}[1]{\mathbb I_h\bo{#1}}
\newcommand{\aI}[1]{\mathbb I^{a}_h\bo{#1}}
\def\ext{\text{ext}}
\def\inte{\text{int}}
\def\R{\mathbb{R}}

\def\P{\mathbb{P}}
\def\G{\mathbb{G}}
\def\I{\mathbb{I}}
\def\bsig{\boldsymbol{\sigma}}
\def\be{\boldsymbol{\epsilon}}

\newcommand{\Id}{\mathrm{I}_3}

\def\cC{\mathcal{C}}
\def\cE{\mathcal{E}}
\def\cF{\mathcal{F}}

\def\cK{\mathcal{K}}

\def\cM{\mathcal{M}}
\def\cN{\mathcal{N}}

\def\cT{\mathcal{T}}
\def\cV{\mathcal{V}}

\def\cW{\mathcal{W}}
\def\fr{\mathfrak{r}}

\def\fC{\mathfrak{C}}

\def\blambda{\boldsymbol{\lambda}}
\def\bmu{\boldsymbol{\mu}}
\def\Up{\Upsilon}
\def\bv{{\bf{v}}}

\def\dls{\boldsymbol{X}_{{\bf grad},K}}
\def\dgs{\boldsymbol{X}_{{\bf grad},h}}
\def\ddgs{\boldsymbol{X}_{{\bf grad},h,0}}

\def\sjump#1{[\hskip -1.5pt[#1]\hskip -1.5pt]}

\def\[{\partial}
\def\O{\Omega}
\def\tr{{\rm {tr}}}
\def\bc{\bos{\chi}}

\newcommand{\bspace}{\O\setminus\overline{\Gamma}}
\newcommand{\Ks}{{\mathcal{K}s}}
\newcommand{\Ke}{{\mathcal{K}e}}
\newcommand{\Kz}{{\mathcal{K}z}}
\newcommand{\Lz}{{\mathcal{L}z}}
\def\Ksig{{K\!\sigma}}
\def\Lsig{{L\!\sigma}}

\def\Lsig{{L\!\sigma}}

\newcommand{\DIV}{\operatorname{\mathrm{div}}}

\newcommand{\email}[1]{\href{mailto:#1}{#1}}

\begin{document}

\allowdisplaybreaks[4]
\numberwithin{figure}{section}
\numberwithin{table}{section}
 \numberwithin{equation}{section}

\title{A higher order polyhedral method for contact mechanics with Tresca friction}

\author[1,2]{J\'er\^ome Droniou}
\author[1]{Raman Kumar}
\author[3]{Roland Masson}
\author[1]{Ritesh Singla}
\affil[1]{IMAG, Univ. Montpellier, CNRS, Montpellier, France, \email{jerome.droniou@cnrs.fr}, \email{raman.kumar@umontpellier.fr}, \email{iritesh281@gmail.com}}
\affil[2]{School of Mathematics, Monash University, Melbourne, Australia}
\affil[3]{Universit\'e C\^ote d’Azur, Inria, CNRS, Laboratoire J.A. Dieudonn\'e, Team Galets, Nice, France, \email{roland.masson@univ-cotedazur.fr}}

\maketitle

 \date{}
 \begin{abstract}
 In this work, we design and analyze a Discrete de Rham (DDR) scheme for a contact mechanics problem involving fractures along which a model of Tresca friction is considered. Our approach is based on a mixed formulation involving a displacement field and a Lagrange multiplier, enforcing the contact conditions, representing tractions at fractures. The approximation space for the displacement is made of vectors values attached to each vertex, edge, face, and element, while the Lagrange multiplier space is approximated by piecewise constant vectors on each fracture face. The displacement degrees of freedom allow reconstruct piecewise quadratic approximations of this field. We prove a discrete Korn inequality that account for the fractures, as well as an inf-sup condition (in a non-standard $H^{-1/2}$-norm) between the discrete Lagrange multiplier space and the discrete displacement space. We provide an in-depth error analysis of the scheme and show that, contrary to usual low-order nodal-based schemes, our method is robust in the quasi-incompressible limit for the primal variable~(displacement). An extensive set of numerical experiments confirms the theoretical analysis and demonstrate the practical accuracy and robustness of the scheme.
 \end{abstract}

\section{Introduction}
The study of Contact mechanics in fractured porous media plays a crucial role in subsurface applications where understanding the complex interplay between mechanical deformation, fracture behavior, and fluid flow is essential for accurate prediction and risk assessment. A prime example of this is $\text{CO}_2$ sequestration, where the action of injecting the gas can create a buildup of pressure that may eventually cause the reactivation of faults by changing pore pressure and normal stress on fracture surfaces, thus posing risks of deformation of rocks, $\text{CO}_2$ leakage \cite{abbassifracture} and seismic activity \cite{GH19}. Fractures are classically represented as a network of planar surfaces connected to the surrounding matrix domain, resulting in mixed-dimensional models that have been extensively studied in poromechanics \cite{abbassifracture, contact-norvegiens, GDM-poromeca-cont, GDM-poromeca-disc, BDMP:21, tchelepi-castelletto-2020, GH19, GKT16, NEJATI2016123}. The movement along fractures is governed by nonlinear frictional laws arising from inequality constraints imposed by non-penetration conditions and the non-smooth nature of friction laws. The Tresca friction model, which prescribes a yield threshold for tangential traction independent of normal traction, provides a mathematically tractable framework while capturing essential friction behavior \cite{chouly2023}. More generally, the Coulomb friction law, where the tangential traction threshold depends on the normal stress, better represents physical friction but introduces additional nonlinearity \cite{CHLR2020}. 

Capturing the geometric complexity of fracture networks with their corners and intersections typically requires a robust and accurate discretization method. Traditional finite element methods (FEM), while well-established for regular meshes, often struggle with non-simplicial element shapes that naturally arise in geological applications \cite{brennerscott, hho-book}. This challenge has motivated the development of polytopal discretization methods offering greater geometric flexibility while maintaining accuracy and stability. Several families of polytopal methods have emerged in computational mechanics over the past two decades, including Hybrid High Order (HHO) \cite{dipietro-ern,hho-book}, Discontinuous Galerkin \cite{hansbo-larson}, Virtual Element Methods (VEM) \cite{beirao-brezzi-marini}, and the Discrete De Rham (DDR) method \cite{DDR19,pjddr}. Some of these have been extended to account for contact mechanics, such as: VEM on non-conforming meshes \cite{Wriggers-2022}, HHO combined with Nitsche's contact formulation \cite{CEP20}, VEM for Nitsche frictionless contact \cite{rolandlaazirifrictionless} and Coulomb frictional contact mechanics \cite{LM25}, and VEM with bubble stabilisation \cite{jhr}. So far, DDR has not been considered for contact mechanics, but it is a natural candidate for mechanics problems since it provides a nodal-based discrete version of the $H^1$ space; it moreover already proved efficient in approximating solid and fluid mechanics models, including Navier-Stokes equations \cite{Di-Pietro.Droniou.ea:24}, Maxwell equations on manifolds \cite{DRONIOU2025113886}, Yang-Mills equations \cite{droniou-yang-mills}, as well as the Reissner-Mindlin and Kirchoff-Love plate problems \cite{pietro-reissner,ddr_kl}.

The discretization of contact mechanics has been an active area of research, driven by the challenges of variational inequalities, the nonlinearity of the frictional constraints, and the need to ensure stability under unilateral conditions. For mixed-dimensional models with fractures represented as co-dimension one manifolds (e.g., 2D surfaces in a 3D medium), additional complexities arise due to intersections, corners, and tips. The mixed-dimensional framework for flow in fractured porous media was pioneered in \cite{MJE05,MAE02}, in which discrete fracture-matrix (DFM) models are developed to represent fractures as lower-dimensional manifolds embedded in the porous matrix. In mixed formulation with Lagrange multipliers, surface tractions are introduced as independent unknowns, with the Lagrange multiplier representing the traction at the contact interface, and enforcing dual cone constraint \cite{brezzifortin,alart}. A fundamental requirement for well-posedness is the satisfaction of a discrete inf-sup condition between displacement and multiplier spaces.
\cite{BR2003} introduced the $\mathbb{P}_1-$bubble FEM discretization for the Signorini problem, enriching the displacement space with bubble functions to satisfy the inf-sup condition. In the polytopal framework, the recent work of \cite{jhr} extended the bubble-stabilized mixed formulation to the first-order VEM framework for contact mechanics with Tresca friction.
Furthermore, \cite{wohlmuthquadraticcontact} proposed an abstract framework for the $~ priori$ error analysis of 3D frictionless contact problems using quadratic finite elements. They established an $\mathcal{O}(h^{\frac12+\fr})$ convergence for $\frac12 <\fr<1$, when the solution belongs to $H^{\frac32+\fr}$, providing $H^1$ and $H^{-1/2}$-norm bounds for displacement and traction errors, respectively. The analysis assumes constant normal vectors and no friction, together with an additional regularity condition~\cite[Assumption~4]{wohlmuthquadraticcontact} on the contact zone (where the transition between contact and non-contact occurs) to ensure smooth behavior near the interface. It is noted that the estimates are not robust for nearly incompressible materials~(that is, when the Poisson ratio tends to $0.5$).

In this work, we design and analyze a second-order DDR scheme for contact mechanics with Tresca friction on fracture networks modeled with co-dimension 1 interfaces. In the discrete formulation, the displacement field is approximated using vector values attached to vertices, edges, faces, and elements, and the contact constraints are enforced through face-wise constant Lagrange multipliers on the fracture network. The chosen degrees of freedom for the displacement allow us to reconstruct piecewise quadratic displacements on the edges, faces, and elements.
A key analytical result of this study is the discrete Korn inequality adapted to the DDR-type spaces, where our proof draws inspiration from \cite{jhr} but differs substantially due to the presence of quadratic face displacement reconstruction. Furthermore, we establish a discrete inf-sup condition based on a non-standard $H^{-1/2}$-type norm on the fracture network, ensuring stability of the mixed formulation. We provide a comprehensive error analysis for solutions with $H^{\frac32+\fr}$ regularity, which yields an $\mathcal O(h^{\frac{1}{2}+\fr})$, where $h$ is the mesh size, order convergence in the energy norm (for $0< \fr<1$ when the Tresca threshold vanishes and $0< \fr\le 1/2$ otherwise -- although the numerical results suggests that $0< \fr<1$ remains valid even when the Tresca threshold does not vanish). Moreover, the error estimates are robust for the primal variable (displacement) in the incompressible limit, when the Poisson ratio tends to $\frac{1}{2}$. This is made possible by the use of a bespoke interpolator (designed by correcting the cell values of the standard DDR interpolator) that enjoys a commutation (Fortin) property with the continuous and discrete divergences. While the global $H^{\frac32+\fr}$ regularity assumption is standard for establishing optimal convergence rates in contact problems, realistic fracture network configurations may exhibit reduced regularity near fracture tips and intersections \cite{Dauge1988}. This could be dealt with by incorporating locally weighted regularity estimates that represent an interesting direction for future theoretical developments out of the scope of the present work.
We also note that, although the analysis is presented for a Tresca-type friction model, the extension of the proposed discretisation to the more involved case of Coulomb friction can be carried out without difficulty, following the same approach as for the lowest-order scheme in \cite{droniou2023bubble}. However, very few error estimates have been derived in the Coulomb case with non-penetration, all based on a mesh-dependent smallness assumption on the friction coefficient \cite{Hild2012}. A more tractable extension to be investigated would consist in considering a normal compliance model with Coulomb friction, following the approach proposed in \cite{Wang21}.

The rest of this paper is organized as follows. Section \ref{model} introduces the governing equations for contact with Tresca friction, presenting the strong and weak formulations. Section \ref{scheme} details the DDR discretization, including mesh notation (Section \ref{subsec mesh}), the construction of discrete spaces with degrees of freedom (Section \ref{subsec_space}), displacement and gradient reconstruction (Section \ref{subsec reconstruct}), and the discrete scheme (Section \ref{subsec discrete_formulation}) including its well posedness. Section \ref{main_results} states the main results, starting from an abstract error estimate (Theorem \ref{thm:abstract.error.estimate}) that does not assume any regularity on the solution, before presenting an error estimate (Theorem \ref{thm:error.estimate}) providing rates of convergence depending on the regularity of the solution; the proofs are presented in Section \ref{sec_proofs}. Finally, the theoretical results are validated by extensive numerical tests presented in Section \ref{sec numerical}.


\section{Model}\label{model}
We consider a discrete fracture-matrix model on the polyhedral domain $\O\subset\R^3$ including a fracture network $\Gamma$ with co-dimension $1$ defined as the union of flat fractures $\Gamma_i$:
\begin{equation*}
\overline{\Gamma}=\bigcup_{i\in\bo{I}}\overline{\Gamma}_i.
\end{equation*}
We assume that $\Gamma$ does not disconnect any part of $\Omega$ from $\partial\Omega$, that is, the boundary of each connected component of $\Omega\backslash \Gamma$ contains a portion of $\partial\Omega$ of non-zero $(d-1)$-dimensional measure. Each fracture $\Gamma_i$ (for $i\in\bo{I}$) is a polygonal simply connected open subdomain of a plane of $\R^3$. Without restriction of generality, it is assumed that fractures may only intersect at their boundaries. The two sides of a given fracture of $\Gamma$ are denoted by $\pm$ in the matrix domain $\O\setminus\overline{\Gamma}$, and respectively called positive and negative sides. The two unit normal vectors $\bo{n}^\pm$, oriented outward from the sides $\pm$, satisfy $\bo{n}^++\bo{n}^-=\bo{0}.$ For simplicity, given homogeneous Dirichlet boundary conditions, the space
\begin{equation*}
\bo{U}_0\coloneqq H^1_0(\O\setminus\overline{\Gamma})^3
\end{equation*}
of the displacements (which may be discontinuous across fractures) is endowed with the norm $\lVert\bo{v}\rVert_{H^1(\O\setminus\overline{\Gamma})}=\lVert \nabla\bo{v}\rVert_{L^2(\O\setminus\overline{\Gamma})}$. 
The oriented jump operator on $\Gamma$ for functions $\bo{u}\in\bo{U}_0$ is defined by
\begin{equation*}
\sjump{\bo{u}}\coloneqq\gamma^+\bo{u}-\gamma^-\bo{u},
\end{equation*}
where $\gamma^\pm$ are the trace operators on each side of $\Gamma$. Its normal and tangential components are denoted respectively by $\sjump{\bo{u}}_{\bo{n}}\coloneqq\sjump{\bo{u}}\cdot\bo{n}^+$ and $\sjump{\bo{u}}_\tau\coloneqq\sjump{\bo{u}}-\sjump{\bo{u}}_{\bo{n}}\bo{n}^+$. The normal trace operators on $H_{\DIV}(\bspace)$ are denoted by $\gamma_{\bo{n}}^\pm$. The symmetric gradient operator $\bos{\epsilon}$ is defined on $\bo{U}_0$ by $\bos{\epsilon}(\bo{v})\coloneqq\frac{1}{2}(\nabla \bo{v}+(\nabla \bo{v})^t)$.

The model we consider accounts for the mechanical equilibrium equation with a linear isotropic elastic constitutive law and a Tresca frictional contact model at matrix-fracture interfaces. In its strong form, it is written
\begin{equation}\label{eq:strong_form}
\begin{cases}
-\DIV\bsig(\bo{u})=\bo{f} & \text{on}~\bspace,\\
\bsig(\bo{u})=\frac{E}{1+\nu} \left(\be( \bo{u})+\frac{\nu}{1-2\nu}(\tr \be(\bo{u}))\Id\right)& \text{on}~\bspace,\\
\gamma_{\bo{n}}^+\bos{\sigma}(\bo{u}) + \gamma_{\bo{n}}^-\bos{\sigma}(\bo{u})=\bo{0}& \text{on}~\Gamma,\\
T_{\bo{n}}(\bo{u})\leq 0,~\sjump{\bo{u}}_{\bo{n}}\leq 0,~\sjump{\bo{u}}_{\bo{n}}T_{\bo{n}}(\bo{u})=0&\text{on}~\Gamma,\\
\lvert \bo{T}_{\tau}(\bo{u})\rvert\leq g&\text{on }\Gamma,\\
\bo{T}_{\tau}(\bo{u})\cdot\sjump{\bo{u}}_\tau+g\lvert \sjump{\bo{u}}_\tau\rvert=0&\text{on }\Gamma,
\end{cases}
\end{equation}
with a Tresca threshold $g\geq 0$, Poisson ratio $\nu$, and Young’s modulus $E$ satisfying $\frac{E}{1+\nu}\in [\mu_{1},\mu_{2}]$ with $0 < \mu_{1} \leq\mu_{2}$, the $3\times 3$ identity matrix $\Id$, and normal and tangential surface tractions $T_{\bo{n}}(\bo{u})=\gamma_{\bo{n}}^+\bos{\sigma}(\bo{u})\cdot\bo{n}^+$ and $\bo{T}_\tau(\bo{u})=\gamma_{\bo{n}}^+\bos{\sigma}(\bo{u})-T_{\bo{n}}(\bo{u})\bo{n}^+$. We assume henceforth that the external force term $\bo{f}$ belongs to $L^2(\O)^3$ and that $g\in L^2(\Gamma)$ is nonnegative.

The weak formulation of the mechanical model with Tresca frictional-contact \eqref{eq:strong_form} is written in mixed form using a vector-valued Lagrange multiplier $\blambda\colon\Gamma\rightarrow \R^3$ at matrix-fracture interfaces. Define the displacement jump space by
\begin{equation*}
\bo{W}_j(\Gamma)\coloneqq\{\sjump{\bo{v}}\colon \bo{v}\in\bo{U}_0\}
\end{equation*}
and denote by $\bo{W}'_j(\Gamma)$ its dual space; the duality pairing between these two spaces is written $\langle\cdot,\cdot\rangle_{\Gamma}$. 
We note that $L^2(\Gamma)^3\subset \bo{W}'_j(\Gamma)$ and that $\langle \bos{\mu},\bo{v}\rangle_{\Gamma}\coloneqq\int_{\Gamma}\bos{\mu}\cdot\bo{v}$ whenever $\bos{\mu}\in L^2(\O)^3$. The dual cone is then defined by 
\begin{equation*}
\cC_f\coloneqq \left\{\bos{\mu}\in \bo{W}'_j(\Gamma)\colon \langle \bos{\mu},\bo{v}\rangle_{\Gamma}\leq \int_\Gamma g|\bo{v}_\tau|\text{ for all }\bo{v}\in \bo{W}_j(\Gamma)\text{ with }\bo{v}\cdot\bo{n}^+\leq 0\right\}.
\end{equation*}
The weak mixed-variational formulation of \eqref{eq:strong_form} reads: find $\bo{u}\in\bo{U}_0$ and $\blambda\in\cC_f$ such that
\begin{subequations}\label{eq:model}
\begin{alignat}{2}
\int_{\O}\bos{\sigma}(\bo{u})\colon\bos{\epsilon}(\bo{v})+\langle\blambda,\sjump{\bo{v}}\rangle_\Gamma&=\int_\O \bo{f}\cdot\bo{v}&&\qquad\forall \bo{v}\in\bo{U}_0,
\nonumber\\
\langle \bos{\mu}-\blambda,\sjump{\bo{u}}\rangle_{\Gamma}&\leq 0&&\qquad\forall \bos{\mu}\in\cC_f.
\label{eq:model.contact}
\end{alignat}
\end{subequations}
We can see from this formulation that $\blambda=-\gamma_{\bo{n}}^+\bos{\sigma}(\bo{u})=\gamma_{\bo{n}}^-\bos{\sigma}(\bo{u})$.
\section{Scheme}\label{scheme}
\subsection{Mesh}\label{subsec mesh}

We take a polyhedral mesh of the domain $\O$ that is compliant with the fracture network $\Gamma$, and we assume that the Tresca threshold $g$ is piecewise constant on the trace $\cF_\Gamma$ of the mesh on $\Gamma$. By compliance,
$$
  \overline{\Gamma}=\bigcup\limits_{\sigma\in\cF_\Gamma}\overline{\sigma}.
$$
The set of cells $K$, the set of faces $\sigma$, the set of edges $e$ and the set of vertices $s$ are denoted respectively by $\cM$, $\cF$, $\cE$ and $\cV$. For $z\in\cM\cup\cF\cup\cE$, we denote by $h_z$ the diameter of $z$  and by $|z|$ its measure (in its own dimension). We also set $h=\max\limits_{K\in\mathcal M}h_K$. 
For $\mathcal X\in\{\cM,\cF,\cE,\cV\}$ and $z\in\cM\cup\cF\cup\cE\cup\cV$, $\mathcal X_z$ denotes the set of entities in $\mathcal X$ that have $z$ in their boundary or that belong to the boundary of $z$. So, for example, the set of faces in the boundary of some $K\in\cM$ is written $\cF_K$, and the set of cells having a vertex $s\in\cV$ in their boundary is $\cM_s$. If $\mathcal X\in\{\cF,\cE,\cV\}$ we denote by $\mathcal X^\inte$ the set of elements of $\mathcal X$ that lie in $\Omega$ and by $\mathcal X^\ext$ those that lie on $\partial\Omega$.

In the following, we write $a\lesssim b$ as shorthand for $a\le Cb$ for some constant $C$ that may depend on $\Omega$, the mesh regularity parameter and the Young modulus $E$, but does not depend on the Poisson ratio $\nu$ or the mesh size $h$. In particular, the constant $C$ does not blow up as $\nu\to \frac12$ (quasi-incompressible regime).

Next, we introduce the broken Sobolev space $H^1(\cM)$ as 
$$
  H^1(\cM)\coloneqq \{v\in L^2(\O):v|_K\in H^1(K)\quad\forall K\in\cM\}.
$$

For any $\sigma\in \cF^\inte$, we denote $\sigma=K|L$ where $K$ and $L$ are the two cells that share the face $\sigma$. For a given cell $K\in\cM$ and a face $\sigma\in\cF_K$, let $\mathbf{n}_{\Ksig}$ represent the unit outward normal to $K$ on $\sigma$. Additionally, for a face $\sigma\in\cF$ and an edge $e\in\cE_\sigma$, we denote by $\bo{n}_{\sigma e}$ the unit outward normal to $\sigma$, in the plane spanned by this face, along the edge $e$.

If $ X \in \mathcal{M} \cup \mathcal{F} \cup \mathcal{E}$ and $ l \in \mathbb{N} $, the space of polynomials of degree $\le l$ on $X$ is denoted by $\mathbb{P}^l(X)$. The $L^2(X)^3$-orthogonal projection on $\mathbb{P}^l(X)^3$ is written $\Pi^l_X$ and $L^2(X)^{3\times 3}$-orthogonal projection on $\mathbb P^l(X)^{3\times 3}$ is denoted as $\bos{\Pi}^l_X$. In case $\mathcal{X}=\mathcal{M}$ or  $\mathcal{X} = \mathcal{F}_\Gamma$, the notation $\mathbb{P}^l(\mathcal{X})$ denotes the space of piecewise polynomials of degree at most $l$ on $\mathcal{X}.$

\subsection{Spaces}\label{subsec_space}~

To account for the discontinuity of the discrete displacement field at matrix-fracture interfaces, the displacement degrees of freedom (DOFs) can be discontinuous across the fracture network; they are however continuous across mesh elements not on this network. To represent this, for each face, edge or vertex $z\in\cF\cup\cE\cup \cV$ and each cell $K\in\cM_z$ containing $z$, we denote by $\Kz$ the set of cells containing $z$ and on the same side of $K$ with respect to the fracture network. For two cells $K,L\in\cM_z$, we have $\Kz=\Lz$ if and only if $K,L$ are on the same side of $\Gamma$. We then denote by $\bo{v}_{\Kz}$ the displacement unknown at the face/edge/vertex $z$ corresponding to the side $K$ in the matrix; the discussion above shows that if $L$ is on the same side of $\Gamma$ as $K$ then $\bo{v}_{\Kz}=\bo{v}_{\Lz}$; however, if $K,L$ lie on different sides, then $\bo{v}_{\Kz}$ and $\bo{v}_{\Lz}$ are independent, which represents the possible discontinuity of the displacement across fractures. Note that if $z$ does not lie on $\Gamma$, then $\Kz=\cM_z$ and there is a unique displacement unknown at $z$ (still denoted by $\bo{v}_{\Kz}$), which represents the continuity of the displacement outside fractures.

Taking into account these multiple DOFs on fracture entities, the numerical scheme for \eqref{eq:model} is based on the vector-valued version of the $X_{\mathop{\mathbf{grad}},h}^1$ space of the Discrete De Rham complex \cite{pjddr}. This scheme is a higher-order extension of the one proposed in \cite{droniou2023bubble,jhr}.
The global discrete space is defined as: 
\begin{align*}
\begin{split}
\dgs\coloneqq\{\uvec{v}_h&=\big((\bo{v}_{\Ks})_{K\in\cM,~s\in \mathcal{V}_K},(\bo{v}_{\Ke})_{K\in\cM,~e\in \mathcal{E}_K},(\bo{v}_{\cK\sigma})_{K\in\cM,~\sigma\in\cF_K},(\bo{v}_K)_{K\in\cM}\big):\\&~\bo{v}_{\Ks}\in \mathbb R^3,~\bo{v}_{\Ke}\in\mathbb R^3,~\bo{v}_{\cK\sigma}\in\mathbb R^3,\text{ and }\bo{v}_K\in\mathbb R^3\},
\end{split}
\end{align*}
and its subspace incorporating homogeneous Dirichlet boundary conditions on $\partial\O$ is
\begin{align*}
\ddgs\coloneqq\{\uvec{v}_h\in\dgs\colon \bo{v}_{\Ks}=0 ~\text{if }s\in\cV^{\ext},~\bo{v}_{\Ke}=0\text{ if }e\in\cE^{\ext},\bo{v}_{\cK\sigma}=0 \text{ if }\sigma\in \cF^{\ext}\}.
\end{align*}
For each $K\in\cM$, the local discrete space gathers the components attached to a cell, its faces, edges and vertices.
\begin{align*}
\begin{split}
\dls\coloneqq\{\uvec{v}_K&=\big((\bo{v}_{\Ks})_{s\in \mathcal{V}_K},(\bo{v}_{\Ke})_{e\in \mathcal{E}_K},(\bo{v}_{\cK\sigma})_{\sigma\in\cF_K},\bo{v}_K\big):~\bo{v}_{\Ks}\in \mathbb R^3,~\bo{v}_{\Ke}\in\mathbb R^3,\\&\qquad\bo{v}_{\cK\sigma}\in\mathbb R^3,\text{ and }\bo{v}_K\in\mathbb R^3\},
\end{split}
\end{align*}
The Lagrange multiplier plays the role of approximations of $-\gamma^+_{\bf n}\sigma({\bf u})$. Its space is made of piecewise constant vectors:
\begin{equation*}
{\bf M}_{h}\coloneqq\{\blambda_h\in L^2(\Gamma)^3:~\blambda_\sigma\coloneqq(\blambda_h)|_\sigma \text{ is a constant }\forall\sigma\in\cF_{\Gamma}\}.
\end{equation*}
The normal and tangential components of $\blambda_h\in\mathbf{M}_h$ are
\begin{equation*}
\lambda_{h,{\bf n}}=\blambda_h\cdot{\bf n}^+, \quad\blambda_{h,\tau}=\blambda_h-\lambda_{h,{\bf n}}{\bf n}^+,
\end{equation*}
and the discrete dual cone is
\begin{equation*}
\cC_{f,h}\coloneqq\{\blambda_h\in{\bf M}_h\colon \lambda_{h,{\bf n}}\geq 0,~\lvert \blambda_{h,\tau}\rvert\leq g\}\subset\cC_f.
\end{equation*}


\subsection{Gradient and Reconstruction Operators}\label{subsec reconstruct}

For each edge $e\in\cE_K$, define the local edge reconstruction as $\Upsilon_{Ke}^2:\ddgs\rightarrow \mathbb P^2(e)^3$ such that
$$
\Upsilon_{Ke}^2\underline{\bv}_h(\bx_s)=\bo{v}_{\Ks}\quad \forall s\in\cV_e,\quad\text{and}\quad
\Pi^0_e\Upsilon^2_{Ke}\underline{\bv}_h=\bo{v}_{\Ke}.
$$

Next, for any face $\sigma\in\cF_K$ we introduce the local face gradient reconstruction $\mathbb G^1_{K\sigma}:\ddgs\rightarrow \mathbb P^1(\sigma)^{3\times 2}$ and face displacement reconstruction $\Upsilon^2_{K\sigma}:\ddgs\rightarrow \mathbb P^2(\sigma)^3$ in the following manner: for all $\bos{\xi}\in\mathbb P^1(\sigma)^{3\times 2}$ and $\bos{\eta}\in\mathbb Z(\sigma)^3\coloneqq\big[(\bx-\overline{\bx}_\sigma)\mathbb P^2(\sigma)\big]^3$,
\begin{equation*}
\int_\sigma \G^1_{K\sigma}\underline{\bv}_h\colon\bos{\xi}=-\int_\sigma \bo{v}_{\cK\sigma}\cdot\DIV(\bos{\xi})+\sum_{e\in\cE_\sigma}\int_e \Up^2_{Ke}\underline{\bv}_h\cdot\bo{\bos{\xi}}{{\bf n}_{\sigma e}},
\end{equation*}
and
\begin{equation}\label{eq:def.upsilon2}
\int_\sigma \Up^2_{K\sigma}\underline{\bv}_h\cdot\DIV(\bos{\eta})=-\int_\sigma \G^1_{K\sigma}\underline{\bv}_h\colon\bos{\eta}+\sum_{e\in\cE_\sigma}\int_e \Up^2_{Ke}\underline{\bv}_h\cdot\bos{\eta}{\bo{\bf n}_{\sigma e}},
\end{equation}
where $\mathbb Z(\sigma)^3$ is identified with a space of matrix-valued functions (with independent copies of $(\bx-\overline{\bx}_\sigma)\mathbb P^2(\sigma)$ on each row), $\overline{\bx}_\sigma$ is the barycenter of face $\sigma$, and the divergences of matrices are taken row-wise.
When writing $\G^1_{K\sigma}\underline{\bv}_h$ and $\bos{\xi}$ above as $3\times 2$-matrices, we actually consider that each row of this matrix is tangent to $\sigma$ (which gives meaning to $\DIV(\bos{\xi})$, the divergence being in the coordinate system on $\sigma$); likewise, each copy of $\mathbb Z(\sigma)$ in $\mathbb Z(\sigma)^3$ is interpreted as a space of tangent polynomials to $\sigma$, and $\mathbb Z(\sigma)^3$ as a space of $3\times 2$ polynomial-valued matrices. Note also that \eqref{eq:def.upsilon2} uniquely defines $\Up^2_{K\sigma}$ because $\DIV:\mathbb Z(\sigma)\to \mathbb{P}^2(\sigma)$ is an isomorphism \cite[Corollary 7.3]{Arnold:18}.

Let $\sigma\in\cF_\Gamma$ be a fracture face, and let $K,L$ be the two cells on each side of $\sigma$, with $K$ lying on the positive side of $\sigma$ (that is, $\bo{n}^+$ introduced in Section~\ref{model} is the outer normal to $K$ on $\sigma$). We define the displacement jump operator on $\sigma$ as $\sjump{\cdot}_{\sigma}:\ddgs\rightarrow \mathbb{P}^2(\sigma)^3$ such that, for all $\uvec{v}_h\in\ddgs,$
\begin{align}\label{eq:discrete.jump}
\sjump{\uvec{v}_h}_{\sigma}\coloneqq \up{K\sigma}\uvec{v}_h-\up{L\sigma}\uvec{v}_h.
\end{align}

Next, for any element $K\in\cM$, we introduce the local element gradient reconstruction $\mathbb G^1_{K}:\ddgs\rightarrow \mathbb P^1(K)^{3\times 3}$ and element displacement reconstruction $\Upsilon^2_{K}:\ddgs\rightarrow \mathbb P^2(K)^3$ in the following manner: for all $\bos{\xi}\in\mathbb P^1(K)^{3\times 3}$ and $\bos{\eta}\in\mathbb Z(K)^3\coloneqq\big[(\bx-\overline{\bx}_K)\mathbb P^2(K)\big]^3$,
\begin{equation}\label{eq:G1K}
\int_K \G^1_{K}\underline{\bv}_h\colon\bos{\xi}=-\int_K \bo{v}_K\cdot\DIV(\bos{\xi})+\sum_{\sigma\in\cF_K}\int_\sigma \Up^2_{K\sigma}\underline{\bv}_h\cdot\bos{\xi}{\bo{\bf n}_{K\sigma}},
\end{equation}
and
\begin{equation*} 
\int_K \Up^2_{K}\underline{\bv}_h\cdot\DIV(\bos{\eta})=-\int_K \G^1_{K}\underline{\bv}_h\colon\bos{\eta}+\sum_{\sigma\in\cF_K}\int_\sigma \Up^2_{K\sigma}\underline{\bv}_h\cdot\bos{\eta}{\bo{\bf n}_{K\sigma}}
\end{equation*}
(with the same considerations for $\mathbb Z(K)^3$ and divergence of matrices as above). Here, $\overline{\bx}_K$ is the barycenter for each cell~$K.$

These local jump, gradient and displacement reconstructions are patched together to create their global piecewise polynomial counterparts $\sjump{\cdot}_h\colon \ddgs\rightarrow \mathbb{P}^2(\cF_\Gamma)^3,~\dgr{h}\colon \ddgs\rightarrow \mathbb P^1(\cM)^{3\times 3}$, and $\up{h}\colon \ddgs\rightarrow \mathbb P^2(\cM)^3$: for all $\uvec{v}_h\in\ddgs$,
\begin{alignat*}{2}
(\sjump{\uvec{v}_h}_h)_{|\sigma}&=\sjump{\uvec{v}_h}_{\sigma}&&\qquad\forall \sigma\in \mathcal{F}_\Gamma,\\
(\dgr{h}\uvec{v}_h)_{|K}&=\dgr{K}\uvec{v}_h&&\qquad\forall K\in \mathcal{M},\\
(\up{h}\uvec{v}_h)_{|K}&=\up{K}\uvec{v}_h&&\qquad\forall K\in \mathcal{M}.
\end{alignat*}

\subsection{Discrete Formulation}\label{subsec discrete_formulation}

Define the discrete symmetric gradient $\epsilon_h$, divergence $\DIV_h$, and stress tensor $\sigma_h$ as: for all $\uvec{v}_h\in\dgs$,
\begin{alignat}{2}
\be_h(\uvec{v}_h)&\coloneqq\frac{1}{2}(\dgr{h}\uvec{v}_h+(\dgr{h}\uvec{v}_h)^T),\quad\DIV_h(\uvec{v}_h)\coloneqq \tr(\be_h(\uvec{v}_h)),\nonumber\\
 \bsig_h(\uvec{v}_h)&\coloneqq \frac{E}{1+\nu}\left(\be_h(\uvec{v}_h)+\frac{\nu}{1-2\nu} \DIV_h(\uvec{v}_h)\Id\right).
 \label{eq:def.sigma.h}
\end{alignat}

The scheme for the contact mechanics problem reads: Find $(\underline{\bf u}_h,\blambda_h)\in \ddgs\times\mathcal C_{f,h}$ such that, for all $(\underline{\bf v}_h,\bmu_h)\in \ddgs\times\mathcal C_{f,h}$,
\begin{subequations}\label{eq:scheme}
\begin{align}
\int_{\Omega}\bsig_h(\underline{\bf u}_h)\colon\be_h(\underline{\bf v}_h)+\mu_{1}\mathbb S_{h}(\underline{\bf u}_h,\underline{\bf v}_h)+\int_{\Gamma}\blambda_h\cdot\sjump{\underline{\bf v}_h}_h&=\sum_{K\in\cM}\int_K {\bf f}\cdot\bo{v}_K ,\label{eq:scheme.displacement}\\
\int_{\Gamma}(\bmu_h-\blambda_h)\cdot\sjump{\underline{\bf u}_h}_h&\leq 0\label{eq:scheme.lagrange}.
\end{align}
\end{subequations}
Here $\mathbb S_{h}$ is the stabilisation bilinear form defined by
\begin{equation*}
\mathbb S_{h}(\underline{\bf u}_h,\underline{\bf v}_h)\coloneqq \sum_{K\in\cM}\mathbb S_K(\underline{\bf u}_h,\underline{\bf v}_h)
\end{equation*}
with local stabilisation bilinear forms $\mathbb S_K:\mathbb X^1_{{\bf grad},h}\times \mathbb X^1_{{\bf grad},h}\rightarrow \mathbb R$ given by 
\begin{equation}\label{eq:stabK}
  \begin{aligned}
    \mathbb S_{K}(\underline{\bf u}_h,\underline{\bf v}_h)={}&h^{-1}_K\sum_{\sigma\in\cF_K}
      \int_\sigma (\Upsilon^2_K\underline{\bf u}_h-\Upsilon^2_{\Ksig}\underline{{\bf u}}_h)\cdot(\Upsilon^2_K\underline{\bf v}_h-\Upsilon^2_{\Ksig}\underline{{\bf v}}_h)\\
    &+\sum_{e\in\cE_K}
      \int_e(\Upsilon^2_K\underline{\bf u}_h-\Upsilon^2_{K e}\underline{{\bf u}}_h)\cdot(\Upsilon^2_K\underline{\bf v}_h-\Upsilon^2_{Ke}\underline{{\bf v}}_h)\\
    &+h_K\sum_{s\in\mathcal V_K}(\Upsilon^2_K\underline{\bf u}_h(\bx_s)-\bo{u}_\Ks)\cdot(\Upsilon^2_K\underline{\bf v}_h(\bx_s)-\bo{v}_\Ks).
  \end{aligned}
\end{equation}
Next we define the broken $H^1$-like semi-norm on $\dgs$ as: for all $\uvec{v}_h\in\dgs:$
\begin{align}\label{discrete_norm}
\lVert \uvec{v}_h\rVert_{1,h}\coloneqq \left(\sum_{K\in\cM}\lVert \uvec{v}_h\rVert^2_{1,K}\right)^{1/2} \text{ with }\lVert\uvec{v}_h\rVert^2_{1,K}\coloneq \lVert \mathbb G^1_K\uvec{v}_h\rVert^2_{L^2(K)}+\mathbb S_K(\uvec{v}_h,\uvec{v}_h).
\end{align}
Note that, through the terms $\mathbb S_K$, this norm bounds the jump between each element components and its faces/edges/vertices components. Due to the single-valuedness of components on faces/edges/vertices that are not on the fracture network, this norm therefore bounds inter-element jumps that do not occur across the fracture network.

\begin{theorem}[Discrete Korn Inequality]\label{thm:korn}
It holds
\begin{alignat}{2}\label{eq:dis_korn}
\lVert \uvec{v}_h\rVert_{1,h}^2&\lesssim \lVert \be_h(\uvec{v}_h)\rVert^2_{L^2(\O\setminus\overline{\Gamma})}+\mathbb S_h(\uvec{v}_h,\uvec{v}_h)&&\qquad\forall\uvec{v}_h\in\ddgs.
\end{alignat}
\end{theorem}
\begin{proof}
See Section \ref{sec:Korn}.
\end{proof}

For notational convenience, let us define the discrete energy inner product $\langle\cdot,\cdot\rangle_{\mathrm{en},h}$ such that, for all $\uvec{v}_h,\uvec{w}_h\in\ddgs$,
\begin{align}\label{eq:def.inner.Eh}
\langle\uvec{v}_h,\uvec{w}_h\rangle_{\mathrm{en},h}\coloneqq\int_{\Omega}\bsig_h(\underline{\bf v}_h)\colon\be_h(\underline{\bf w}_h)+\mu_{1}\mathbb S_h(\underline{\bf v}_h,\underline{\bf w}_h),
\end{align}
where $\mathbb S_h=\sum_{K\in\cM}\mathbb S_K(\uvec{u}_h,\uvec{v}_h)$ ~and $\mu_{1}$ is, as in Section~\ref{model}, a lower bound of $E/(1+\nu)$. This bilinear form is indeed an inner product by the discrete Korn inequality \eqref{eq:dis_korn}, $\mu_{1}\leq \frac{E}{1+\nu}$, and the definition \eqref{eq:def.sigma.h} of $\bsig_h$, which shows that
\begin{align}\label{eq:norm_energy}
\mu_{1}\lVert \uvec{v}_h\rVert_{1,h}^2\leq \widetilde{C} \lVert \uvec{v}_h\rVert_{\mathrm{en},h}^2,
\end{align}
where $\lVert\cdot\rVert_{\mathrm{en},h}$ is the norm induced by the inner product $\langle\cdot,\cdot\rangle_{\mathrm{en},h}$ and the constant $\widetilde{C}$ is independent of $E$ and $\nu$.

\begin{defn}[${H}^{-1/2}$-like norm on $L^2(\Gamma)^3$]\label{def:H^{-1/2}}
The norm $\lVert \cdot \rVert_{-1/2,\Gamma}$ on  $L^2(\Gamma)^3$ is defined by: for all $\blambda\in L^2(\Gamma)^3$,
\begin{align}\label{eq:def.dis{-1/2}norm}
\lVert \blambda\rVert_{-1/2,\Gamma}\coloneqq\sup\limits_{\bo{v}\in {H}^1_0(\O\setminus\Gamma)^{3}\setminus\{0\}}{}\frac{\int_{\Gamma}\blambda\cdot\sjump{\bo{v}}}{\lVert \bo{v}\rVert_{{H}^1(\O\setminus\Gamma)}}.
\end{align}
\end{defn}

\begin{remark}[Comparison with the norm in \cite{jhr}]
This discrete $H^{-1/2}$-like norm is simpler than that in \cite[Definition 4.3]{jhr}, which is made possible by our use of a higher-order scheme for the displacement that involves face degrees of freedom on both sides of the fracture.
\end{remark}

\begin{theorem}[Discrete inf-sup condition]\label{thm:inf-sup}
It holds
\begin{equation}\label{eq:def.inf-sup}
  \sup\limits_{\uvec{v}_h\in\ddgs\setminus\{0\}}\frac{\int_\Gamma \blambda_h\cdot\sjump{\uvec{v}_h}_h}{\lVert \uvec{v}_h\rVert_{1,h}}\gtrsim\lVert\blambda_h\rVert_{-1/2,\Gamma}\qquad\forall\blambda_h\in\bo{M}_h.
\end{equation}
\end{theorem}
\begin{proof}
See Section \ref{sec:Inf-sup}.
\end{proof}

\begin{prop}[Existence and uniqueness result]
There exists a unique solution $(\uvec{u}_h,\blambda_h)\in\ddgs\times\cC_{f,h}$ to \eqref{eq:scheme}.
\end{prop}
\begin{proof}
The abstract framework of \cite[Theorem 3.8]{HASLINGER1996313} on saddle point problems gives the existence of a solution. The uniqueness of $\uvec{u}_h$ derives from the discrete Korn inequality \eqref{eq:dis_korn} and the uniqueness of $\blambda_h$ from the discrete inf-sup property \eqref{eq:def.inf-sup}, as demonstrated in \cite[Proposition 4.1]{jhr}.
\end{proof}

\section{Main results}\label{main_results}
We denote by $C^0_0(\overline{\O}\setminus\Gamma)$ the set of continuous functions $\overline{\O}\setminus\Gamma\to\R$ that vanish on $\partial\O$.
The interpolator $\mathbb I_h:C^0_0(\overline{\O}\setminus\Gamma)^3\rightarrow \dgs$ is defined as follows: for all $\bo{v}\in C^0_0(\overline{\O}\setminus\Gamma)^3$,
\begin{align*}
\mathbb I_h\bo{v}\coloneqq \big((\I_{\Ks}\bo{v})_{K\in\cM,~s\in \mathcal{V}_K},(\I_{\Ke}\bo{v})_{K\in\cM,~e\in \mathcal{E}_K},(\I_{\cK\sigma}\bo{v})_{K\in\cM,~\sigma\in\cF_K},(\I_K\bo{v})_{K\in\cM}\big),
\end{align*}
where
\begin{subequations}\label{eq:DOF.Ih}
\begin{alignat}{2}
\I_{\Ks}\bo{v}&=\bo{v}|_K(\bx_s)&&\qquad\forall s\in\cV_K,~K\in\cM,\notag\\
\I_{\cK e}\bo{v}&=\Pi^0_e(\bo{v}|_K)&&\qquad\forall e\in\cE_K,~K\in\cM,\notag\\
\I_{\cK\sigma}\bo{v}&=\Pi^0_\sigma(\bo{v}|_K)&&\qquad\forall \sigma\in\cF_K,~K\in\cM,\label{eq:def.Ih.sigma}\\
\I_{K}\bo{v}&=\Pi^0_K\bo{v}+\widetilde{\bo{v}}_K&&\qquad\forall K\in\cM,\label{eq:def.Ih.K}
\end{alignat}
\end{subequations}
and $\widetilde{\bo{v}}_K\in\mathbb R^3$ is defined by
\begin{equation}\label{eq:tildevK}
\int_K \widetilde{\bo{v}}_K\cdot\nabla \phi=-\sum_{\sigma\in\cF_K}\int_\sigma (\bo{v}|_K-\up{\Ksig}\mathbb I_h\bo{v})\cdot \phi\bo{n}_{\Ksig}
\quad\forall \phi\in \P^1(K).
\end{equation}

\begin{remark}[Definition of the interpolator]\label{rem:def.interpolator}
We note that the definition \eqref{eq:tildevK} is not self-referencing (and therefore makes sense): even though the interpolate $\mathbb I_h\bo{v}$ appears in the right-hand side, this is through the reconstruction $\up{\Ksig}$ which only relies on the -- already defined -- values of the interpolate on vertices, edges, and faces, not on $K$ (for which $\widetilde{\bo{v}}_K$ must be defined). Moreover, \eqref{eq:tildevK} properly fixes a unique vector in $\mathbb R^3$ because $\nabla:P^1(K)/\mathbb R\to \mathbb R^3$ is an isomorphism and, if $\phi$ is constant, the right-hand side vanishes by \eqref{eq:Pisigma.Pot} below.

In the definition \eqref{eq:def.Ih.K} of the cell component of $\I_h\bo{v}$, the correction $\widetilde{\bo{v}}_K$ to the $L^2$-projection $\Pi^0_K(\bo{v}|_K)$ is required to ensure that $\mathbb I_h$ is a Fortin interpolator for the divergence, see Lemma \ref{lem:fortin}.
We note however that, if $\bo{v}\in\mathbb P^2(K)^3$, then $\up{\Ksig}\mathbb I_h\bo{v}=(\bo{v}|_K)|_\sigma$, so $\widetilde{\bo{v}}_K=0$. This correction is therefore a higher-than-quadratic correction to the interpolator.
\end{remark}

\begin{lemma}[Properties of the reconstruction operators] \label{lem:prop.reconstructions}
The edge, face, and element gradients, and reconstruction operators satisfy the following properties:
\begin{alignat}{2}
\Pi^0_\sigma\Upsilon_{\Ksig}^2(\uvec{v}_h)&=\bo{v}_{\Ksig}&&\qquad\forall \sigma\in\cF\,,\;\forall \uvec{v}_h\in \dgs,\label{eq:Pisigma.Pot}\\
\Pi^0_K\Upsilon_{K}^2(\uvec{v}_h)&=\bo{v}_K&&\qquad\forall K\in\cM\,,\;\forall \uvec{v}_h\in \dgs,\label{eq:PiK.Pot}\\
\mathbb G_{K}^1\mathbb I_h\bo{q}&=\nabla\bo{q}&&\qquad \forall K\in\cM\,,\;\forall \bo{q}\in \mathbb P^2(K)^3,\label{eq:Grad.interpolate}\\
\Upsilon_{K}^2\mathbb I_h\bo{q}&=\bo{q}&&\qquad\forall K\in\cM\,,\;\forall \bo{q}\in \mathbb P^2(K)^3,\label{eq:Upsilon.interpolate}\\
\up{Ke}\mathbb I_h\bo{q}&=\bo{q} &&\qquad \forall K\in\cM\,,\;\forall e\in\cE_K,\;\forall \bo{q}\in \mathbb P^2(e)^3,\label{eq:Upsilon.sigma.interpolate}\\
\up{K\sigma}\mathbb I_h \bo{q}&=\bo{q}&&\qquad \forall K\in\cM\,,\; \forall \sigma\in\cF_K,\;\forall \bo{q}\in \mathbb P^2(\sigma)^3,\notag
\end{alignat}
where, by abuse of notation, the interpolator $\mathbb I_h$ applied to functions only defined on a mesh entity (cell, face, or edge) corresponds to
the restriction of the degrees of freedom \eqref{eq:DOF.Ih} belonging to that mesh entity (notice that the gradient or displacement reconstructions on the corresponding mesh entity only depend on these degrees of freedom).
\end{lemma}
\begin{proof}
We refer to \cite[Lemma 3]{pjddr} for the proof of this lemma, using Remark \ref{rem:def.interpolator} for \eqref{eq:Grad.interpolate} and \eqref{eq:Upsilon.interpolate}.
\end{proof}

Following the approach of the Gradient Discretisation Method \cite{gdm}, the abstract error estimate will be stated in terms of two quantities measuring the (primal) consistency and the limit-conformity (also known as adjoint consistency) of the discrete space and operators.
\begin{itemize}
\item The primal consistency error is: for $\uvec{v}_h\in\ddgs,$
\begin{align}\label{eq:primal.consist}
\bos{\fC}_h(\bo{u},\uvec{v}_h)\coloneqq \left(\lVert \nabla \bo{u}-\dgr{h}\uvec{v}_h\rVert_{L^2(\O\setminus\overline{\Gamma})}^2+\mathbb S_h(\uvec{v}_h,\uvec{v}_h)\right)^{1/2}.
\end{align}
\item Letting
\begin{align*}
\bo{\Sigma}\coloneqq\{\bos{\chi}\in H_{\DIV}(\O\setminus\overline{\Gamma};\mathcal{S}^3(\mathbb R)),\colon \gamma^+_{\bo{n}}\bos{\chi}+\gamma^-_{\bo{n}}\bos{\chi}=\bo{0},~\gamma^+_{\bo{n}}\bos{\chi}\in L^2(\Gamma)\},
\end{align*}
(where $\mathcal S^3(\mathbb R)$ is the space of symmetric $3\times 3$ matrices with real coefficients), the limit-conformity measure is defined as, for $\bos{\chi}\in\bo{\Sigma}:$
\begin{equation}\label{eq:adjoint.consist}
  \begin{aligned}
&  \bos{\cW}_h(\bos{\chi})\coloneqq{} \sup\limits_{\uvec{v}_h\in\ddgs}\frac{w_h(\bos{\chi},\uvec{v}_h)}{\lVert\uvec{v}_h\rVert_{1,h}}, \\
\text{ where }  w_h(\bos{\chi},\uvec{v}_h)\coloneqq{}& -\int_\O \bos{\chi}\colon \be_h(\uvec{v}_h)+\int_\Gamma \Pi^0_{\cF_\Gamma}(\gamma^+_{\bo{n}}\bos{\chi})\cdot\sjump{\uvec{v}_h}_h-\sum_{K\in\cM}\int_{K}\bo{v}_K\cdot\DIV\bos{\chi}.
  \end{aligned}
\end{equation}
\end{itemize}
The term ``adjoint consistency'' refers to the fact that $\bos{\cW}_h(\bos{\chi})$ measures a defect of integration-by-parts formula between the  discrete strain and the continuous divergence; put another way, it measures how well some formal adjoint of the discrete strain approximates the divergence.

\begin{theorem}[Abstract error estimate]\label{thm:abstract.error.estimate}
Let $(\bo{u},\blambda)$ be the solution to \eqref{eq:model} and assume that $\bo{u}\in H^1_0(\O\setminus\overline{\Gamma})^3\cap H^{\frac{3}{2}+\fr}(\cM)^3$ for some $0<\fr\leq 1$. Then, the solution $(\uvec{u}_h,\blambda_h)$ of \eqref{eq:scheme} satisfies the following error estimate:
\begin{align}
  \lVert \mI{u}-\uvec{u}_h\rVert_{\mathrm{en},h}\lesssim{}& \frac{1}{\sqrt{\mu_{1}}} \bos{\cW}_h(\bsig(\bo{u}))+ \sqrt{\mu_{2}}\bos{\fC}_h(\bo{u},\mI{u})\notag\\
  &+\max\Bigg[0;\int_{\Gamma}(\blambda-\Pi^0_{{\cF}_\Gamma}\blambda)\cdot\big( \sjump{\bo{u}}-\Pi^0_{\cF_\Gamma}\sjump{\bo{u}}\big)\Bigg]^{1/2},\label{eq:abstract.error.estimate}\\
 \lVert \blambda-\blambda_h\rVert_{-1/2,\Gamma} \lesssim{}& \bos{\cW}_h(\bsig(\bo{u}))+\mu_{2}\bos{\fC}_h(\bo{u},\mI{u})+\lVert\blambda-\Pi^0_{\cF_\Gamma}\blambda\rVert_{-1/2,\Gamma}\notag{}\\&\qquad +\sqrt{\mu_{2}}\left(1+\left(\frac{\nu}{1-2\nu}\right)^{\frac{1}{2}}\right)\lVert\uvec{u}_h-\mI{u}\rVert_{\mathrm{en},h}.
 \label{eq:abstract.error.estimate_lam}
 \end{align}
\end{theorem}

\begin{remark}[Locking-free estimate]
Recalling that the hidden multiplicative constant in $\lesssim$ does not depend on the Poisson ratio, we notice that the error bound \eqref{eq:abstract.error.estimate} on the matrix displacement does not contain any constant that blows up in the incompressible limit $\nu\to \frac12$, whereas the term $\frac{1}{\sqrt{1-2\nu}}$ in the right-hand side of the estimate \eqref{eq:abstract.error.estimate_lam} on the Lagrange multiplier does blow up in this limit; this is not unexpected since this Lagrange multiplier plays the role of a traction, whose definition explicitly involve such a term.
\end{remark}

\begin{theorem}[Error estimate]\label{thm:error.estimate}
Let $(\bo{u},\blambda)$ be the solution to \eqref{eq:model}. Consider real parameters $\fr$ and $s_{\fr}$ satisfying
\begin{itemize}
\item If $g=0$ (no friction): $0\le \fr<1$, $s_{\fr}=\fr$, and $t_{\fr} =1+\fr$. 

\item If $g\not=0$ (friction): $0\le\fr\le 1/2$, $s_{\fr}=2\fr$, and $t_{\fr}=1$.
\end{itemize}
Assume that $\bo{u}\in H^1_0(\O\setminus\overline{\Gamma})^3\cap H^{\frac{3}{2}+\fr}(\cM)^3$ and $\blambda\in H^{s_r}(\Gamma)^3$. Then, the discrete solution $(\uvec{u}_h,\blambda_h)$ of \eqref{eq:scheme} satisfies the following error estimates:
\begin{align}
\lVert\uvec{u}_h-\mI{u}{}\rVert_{\mathrm{en},h}\lesssim{}& h^{\frac{1}{2}+\fr}\bigg[{}\frac{1}{\sqrt{\mu_{1}}}\bigg(\lvert\bos{\sigma}(\bo{u})\rvert_{H^{\frac{1}{2}+\fr}(\cM)}+\lvert \blambda\rvert_{H^{\fr}(\Gamma)}\bigg)+\sqrt{\mu_{2}}\lvert\bo{u}\rvert_{H^{\frac{3}{2}+\fr}(\cM)}\notag\\
&\qquad\qquad+\lvert \blambda\rvert_{H^{s_\fr}(\Gamma)}+\lvert \sjump{\bo{u}}\rvert_{H^{t_{\fr}}(\Gamma)}\bigg],\label{eq:total.error}\\
\lVert \blambda-\blambda_h\rVert_{-1/2,\Gamma}\lesssim{}&h^{\frac{1}{2}+\fr} \left(\lvert\bos{\sigma}(\bo{u})\rvert_{H^{\frac{1}{2}+\fr}(\cM)}+\mu_2\lvert\bo{u}\rvert_{H^{\frac{3}{2}+\fr}(\cM)}+\lvert \blambda\rvert_{H^{\fr}(\Gamma)}\right){} \notag\\
&+\sqrt{\mu_{2}}\left(1+\left(\frac{\nu}{1-2\nu}\right)^{\frac{1}{2}}\right)\lVert\uvec{u}_h-\mI{u}\rVert_{\mathrm{en},h},\label{eq:total.error_lm}
\end{align}
where we recall that the hidden constants in $\lesssim$ do not depend on the Young modulus $E$ or the Poisson ratio $\nu$.
\end{theorem}

\begin{remark}[Error estimates]\label{rem:error.estimates}
The error estimate \eqref{eq:total.error} is (potentially, if the solution is smooth enough) better in the case of no-friction, since $\fr$ is then allowed to go almost to $1$, instead of being limited to $1/2$ in case of friction. We note that the contact mechanics problem is presented for quadratic finite elements in \cite{wohlmuthquadraticcontact} but that the analysis in this reference, which allows for $0\le \fr<1$, is only performed in the case of zero Tresca threshold $g$. It is unclear at the moment if the analysis allowing for larger $\fr$ can be adapted (in the case of finite elements or our scheme) to the case $g\not=0$. Our numerical simulations however indicate that the higher accuracy is also achieved in the latter case.
\end{remark}

\section{Proof of error estimates}\label{sec_proofs}
We start this section with bounds involving face operators.

\begin{lemma}
For all $\uvec{v}_h\in\dgs$ and $K\in\cM$, the following relations hold:
\begin{alignat}{2}
\lVert \uvec{v}_h\rVert^2_{1,K}\approx{}&\lVert \nabla\up{K}\uvec{v}_h\rVert^2_{L^2(K)}+\mathbb S_K(\uvec{v}_h,\uvec{v}_h)\,,\label{eq:norm_equiv}\\
h_K\lVert \nabla_\sigma \up{\Ksig}\uvec{v}_h\rVert^2_{L^2(\sigma)}\lesssim{}&\lVert \uvec{v}_h\rVert^2_{1,K}\quad\forall \sigma\in\cF_K\,,\label{eq:gradPot.bound}\\
h^{-{1}}_K\lVert \bo{v}_K-\up{\Ksig}\uvec{v}_h\rVert^2_{L^2(\sigma)}\lesssim{}&\lVert\uvec{v}_h\rVert_{1,K}^2\quad\forall \sigma\in\cF_K,\label{eq:vk.bound}
\end{alignat}
where $\nabla_\sigma$ is the tangential gradient on $\sigma$.
\end{lemma}

\begin{proof}
Let $\uvec{v}_h\in\dgs$ and $K\in\cM$. For every $\bos{\xi}\in\P^1(K)^{3\times 3}$, we have
\begin{alignat*}{1}
\int_K \nabla\up{K}\underline{\bv}_h\colon\bos{\xi}\overset{\text{IBP}, \eqref{eq:PiK.Pot}} &=-\int_K \bo{v}_K\cdot\DIV(\bos{\xi})+\sum_{\sigma\in\cF_K}\int_\sigma \Up^2_{K}\underline{\bv}_h\cdot\bos{\xi}{\bo{\bf n}_{K\sigma}}\nonumber\\
\overset{\eqref{eq:G1K}}&=\int_K \G^1_{K}\underline{\bv}_h\colon\bos{\xi}+\sum_{\sigma\in\cF_K}\int_\sigma \big(\up{K}\uvec{v}_h-\Up^2_{K\sigma}\underline{\bv}_h\big)\cdot\bos{\xi}{\bo{\bf n}_{K\sigma}}\nonumber\\
&\leq \int_K \G^1_{K}\underline{\bv}_h\colon\bos{\xi}+\sum_{\sigma\in\cF_K}\lVert \up{K}\uvec{v}_h-\Up^2_{K\sigma}\underline{\bv}_h\rVert_{L^2(\sigma)}\lVert \bos{\xi}\rVert_{L^2(\sigma)}.
\end{alignat*}
Applying the discrete trace inequality \cite[Lemma 1.32]{hho-book} to $\bos{\xi}$, we infer
\begin{align}\label{eq:gradPotGrad.int}
\int_K \big(\nabla\up{K}\underline{\bv}_h-\dgr{K}\uvec{v}_h\big)\colon\bos{\xi}\lesssim \left(\sum_{\sigma\in\cF_K}h_K^{-1/2}\lVert \up{K}\uvec{v}_h-\Up^2_{K\sigma}\underline{\bv}_h\rVert_{L^2(\sigma)}\right)\lVert\bos{\xi}\rVert_{L^2(K)}.
\end{align}
Taking $\bos{\xi}= \nabla\up{K}\underline{\bv}_h-\dgr{K}\uvec{v}_h\in\P^1(K)^{3\times 3}$ and simplifying, we obtain
\begin{align}\label{eq:gradPotGrad.bound}
\lVert \nabla\up{K}\underline{\bv}_h-\dgr{K}\uvec{v}_h\rVert_{L^2(K)}\lesssim \sum_{\sigma\in\cF_K}h_K^{-1/2}\lVert \up{K}\uvec{v}_h-\Up^2_{K\sigma}\underline{\bv}_h\rVert_{L^2(\sigma)}\overset{\eqref{eq:stabK}}\leq\mathbb S_K(\uvec{v}_h,\uvec{v}_h)^{1/2}.
\end{align}
The norm equivalence \eqref{eq:norm_equiv} follows by combining triangles inequalities, the definition \eqref{discrete_norm} of $\lVert{\cdot}\rVert_{1,K}$, and \eqref{eq:gradPotGrad.bound}.

Next, assume that $\sigma\in\cF_K$ is a generic face of $K$. The relation \eqref{eq:gradPot.bound} follows by using the triangle inequality, the bound $|\nabla_\sigma\up{K}\uvec{v}_h|\le |\nabla\up{K}\uvec{v}_h|$, the discrete inverse and trace inequalities \cite[Lemmas 1.28 and 1.32]{hho-book}, the relation $h_K\approx h_\sigma$ (resulting from the mesh regularity assumption) and \eqref{eq:stabK} to write
\begin{align*}
\lVert \nabla_\sigma\up{\Ksig}\uvec{v}_h\rVert^2_{L^2(\sigma)}&\lesssim \lVert \nabla_\sigma \big(\Up^2_{K\sigma}\underline{\bv}_h-\up{K}\uvec{v}_h\big)\rVert^2_{L^2(\sigma)}+\lVert \nabla\up{K}\uvec{v}_h\rVert^2_{L^2(\sigma)}\\
&\lesssim h_K^{-2}\lVert \up{K}\uvec{v}_h-\Up^2_{K\sigma}\underline{\bv}_h\rVert^2_{L^2(\sigma)}+h^{-1}_K\lVert \nabla\up{K}\uvec{v}_h\rVert^2_{L^2(K)}\\
\overset{\eqref{eq:stabK}}&\lesssim h^{-1}_K\left(\mathbb S_K(\uvec{v}_h,\uvec{v}_h)+\lVert \nabla\up{K}\uvec{v}_h\rVert^2_{L^2(K)}\right)\\
\overset{\eqref{eq:norm_equiv}}&\approx h^{-1}_K\lVert\uvec{v}_h\rVert^2_{1,K}.
\end{align*}
Furthermore, using \eqref{eq:PiK.Pot}, the triangle inequality, the discrete trace inequality \cite[Lemma 1.32]{hho-book} and the approximation properties of $\Pi^0_K$ \cite[Theorem 1.45]{hho-book}, we obtain
\begin{align*}
\lVert \bo{v}_K-\up{\Ksig}\uvec{v}_h\rVert_{L^2(\sigma)}&\leq \lVert \Pi^0_K\up{K}\uvec{v}_h-\up{K}\uvec{v}_h\rVert_{L^2(\sigma)}+\lVert \up{K}\uvec{v}_h-\up{\Ksig}\uvec{v}_h\rVert_{L^2(\sigma)}\\
\overset{\eqref{eq:stabK}}&\lesssim h^{-1/2}_K\lVert \Pi^0_K\up{K}\uvec{v}_h-\up{K}\uvec{v}_h\rVert_{L^2(K)}+h^{1/2}_K\mathbb S_K(\uvec{v}_h,\uvec{v}_h)^{1/2}\\
&\lesssim h^{1/2}_K\left(\lVert\nabla\up{K}\uvec{v}_h\rVert_K+\mathbb S_K(\uvec{v}_h,\uvec{v}_h)^{1/2}\right)\\
\overset{\eqref{eq:norm_equiv}}&\approx h^{1/2}_K\lVert \uvec{v}_h\rVert_{1,K}.
\qedhere
\end{align*}
\end{proof}
The following lemma provides a DOF-based bound on the discrete norm, which plays a crucial role in proving the bound on the consistency of the gradient reconstruction (Lemma \ref{l4.4}).

\begin{lemma}\label{l3.4}
Let $K\in \cM$ and $\uvec{v}_h\in\dgs$. Then, it holds:
\begin{align}\label{eq:dofnorm}
\lVert \uvec{v}_h\rVert_{1,K}\lesssim h^{1/2}_K\left(\lvert \bo{v}_K\rvert+\max_{\sigma\in\cF_K}\lvert \bo{v}_{\cK\sigma}\rvert+\max_{e\in\cE_K}\lvert \bo{v}_\Ke\rvert+\max_{s\in\mathcal V_K}\lvert \bo{v}_\Ks\rvert\right).
\end{align}
\end{lemma}

\begin{proof}
Recalling the definition \eqref{discrete_norm} of $\lVert\cdot\rVert_{1,K}$, we will begin by finding a bound on $\mathbb S_K(\uvec{v}_h,\uvec{v}_h)$. To this purpose, we use triangle and discrete trace inequalities \cite[Lemma 1.32]{hho-book} to write
\begin{align*}
\mathbb S_K(\uvec{v}_h,{}&\uvec{v}_h)=h^{-1}_K\sum_{\sigma\in\cF_K}\lVert \Upsilon^2_K\uvec{v}_h-\Upsilon^2_{K\sigma}\uvec{v}_h\rVert_{L^2(\sigma)}^2+\sum_{e\in\cE_K}\lVert \Upsilon^2_K\uvec{v}_h-\Upsilon^2_{Ke}\uvec{v}_h\rVert_{L^2(e)}^2\\
&\qquad\qquad+h_K\sum_{s\in\mathcal V_K}\lvert \Upsilon^2_K\uvec{v}_h(\bx_s)-\bo{v}_{\cK s }\rvert^2\\
\lesssim{}& h^{-2}_K\lVert \Upsilon^2_K\uvec{v}_h\rVert^2_{L^2(K)}+\max_{\sigma\in\cF_K}h^{-1}_K\lVert \Upsilon^2_{K\sigma}\uvec{v}_{h}\rVert^2_{L^2(\sigma)}
+\max_{e\in\cE_K}\lVert \Upsilon^2_{Ke}\uvec{v}_h\rVert^2_{L^2(e)}+h_K\max_{s\in\mathcal V_K}\lvert \bo{v}_{\cK s}\rvert^2.
\end{align*}
Now, the estimates of \cite[Lemma 8]{Di-Pietro.Droniou.ea:24} give
\begin{align}
\lVert \up{Ke}\uvec{v}_h\rVert_{L^2(e)}&\lesssim \norm{L^2(e)}{\bo{v}_\Ke} + h^{1/2}_e\sum_{s\in\mathcal V_e}\lvert \bo{v}_\Ks\rvert,\nonumber\\
\lVert \up{K\sigma}\uvec{v}_h\rVert_{L^2(\sigma)}&\lesssim \norm{L^2(\sigma)}{\bo{v}_{\cK\sigma}} + h_\sigma^{1/2}\sum_{e\in\cE_\sigma}\norm{L^2(e)}{\up{Ke}\uvec{v}_h}
+h_\sigma \sum_{s\in\mathcal V_\sigma}\lvert \bo{v}_\Ks\rvert, \nonumber\\ 
\lVert \up{K}\uvec{v}_h\rVert_{L^2(K)}&\lesssim \norm{L^2(K)}{\bo{v}_K} + h_K^{1/2}\sum_{\sigma\in\cF_K} \norm{L^2(\sigma)}{\bo{v}_{\cK\sigma}} + h_K\sum_{e\in\cE_K}\norm{L^2(e)}{\up{\Ke}\uvec{v}_h}+h_K^{3/2}\sum_{s\in\mathcal V_K}\lvert \bo{v}_\Ks\rvert.\label{eq:cellPotdof.bound}
\end{align}
Using triangle inequality, equation \eqref{eq:norm_equiv}, and mesh regularity (which gives $h_e\approx h_{\sigma}\approx h_K$), we obtain
\begin{align*}
\mathbb S_K(\uvec{v}_h,\uvec{v}_h)\lesssim h^{-2}_K\lVert \bo{v}_K\rVert_{L^2(K)}^2+\max_{\sigma\in\cF_K}h^{-1}_K\lVert \bo{v}_{\Ksig}\rVert^2_{L^2(\sigma)}+\max_{e\in\cE_K}\lVert \bo{v}_{\cK e}\rVert^2_{L^2(e)}+h_K\max_{s\in\mathcal V_K}\lvert \bo{v}_{\cK s}\rvert^2.
\end{align*}
Since, all of $\bo{v}_K, \bo{v}_{\Ksig},\bo{v}_{\cK e}, \text{ and }\bo{v}_{\cK s}$ are constant, a use of mesh regularity (which gives $|e|\approx h_K$, $|\sigma|\approx h_K^2$ and $|K|\approx h^3_K$) shows that $\mathbb S_K(\uvec{v}_h,\uvec{v}_h)^{1/2}$ is bounded above by the right-hand side of \eqref{eq:dofnorm}. The bound on $\lVert \uvec{v}_h\rVert_{1,K}$ then follows from \eqref{eq:norm_equiv}, the bound \eqref{eq:cellPotdof.bound}, and an inverse inequality.
\end{proof}

\subsection{Proof of the discrete Korn Inequality (Theorem~\ref{thm:korn})}\label{sec:Korn}

Recalling \eqref{eq:norm_equiv}, to prove the inequality \eqref{eq:dis_korn}, we just need to bound $\lVert\nabla_{\cM}\up{h}\uvec{v}_h\rVert_{L^2(\O\setminus\overline{\Gamma})}$ (where $\nabla_\cM$ is the broken gradient on $\cM$).
Applying the same techniques as in \cite[Theorem 5.7]{jhr} (which consist in adapting the node-averaging approach of \cite[Section 7.3.2]{hho-book}
to the case of a fractured medium), we obtain
\begin{equation}\label{eq:Korn.start}
\lVert \nabla_\cM\up{h}\uvec{v}_h\rVert^2_{L^2(\O)}\lesssim \lVert \be_\cM( \up{h}\uvec{v}_h)\rVert_{L^2(\O\setminus\overline{\Gamma})}^2+\sum_{\sigma\in\cF\setminus\cF_\Gamma}h^{-1}_\sigma\lVert \sjump{\up{h}\uvec{v}_h}\rVert_{L^2(\sigma)}^2, 
\end{equation}
where $\be_\cM$ is the broken strain on $\cM$ and, for all $\sigma=K|L\in\cF^\inte$, $|\sjump{\up{h}\uvec{v}_h}|_\sigma|=|\up{K}\uvec{v}_h-\up{L}\uvec{v}_h|$ while, if $\sigma\in\cF^\ext\cap\cF_K$, $|\sjump{\up{h}\uvec{v}_h}|_\sigma|=|\up{K}\uvec{v}_h|$. Note that, in \eqref{eq:Korn.start}, thanks to the modified node-averaging operator of \cite[Theorem 5.7]{jhr} only jumps across non-fracture faces appear; this will be essential below as only these jumps can be bounded above using the stabilisation $\mathbb S_h$.

 Recalling \eqref{eq:gradPotGrad.int}, taking $\bos{\xi}=\be_\cM(\up{h}\uvec{v}_h)|_K-\be_h(\uvec{v}_h)|_K\in \mathbb P^1(K)^{3\times 3}$ and simplifying, we have
$$
\norm{L^2(K)}{\be_\cM(\up{h}\uvec{v}_h)-\be_h(\uvec{v}_h)}\lesssim \mathbb S_K(\uvec{v}_h,\uvec{v}_h)^{1/2}.
$$

Using a triangle inequality, squaring and summing over $K\in\cM$ yields
\begin{equation}\label{eq:symPot.bound}
\norm{L^2(\Omega)}{\be_\cM(\up{h}\uvec{v}_h)}^2\lesssim \norm{L^2(\Omega)}{\be_h(\uvec{v}_h)}^2+
\sum_{K\in\cM}\mathbb S_K(\uvec{v}_h,\uvec{v}_h).
\end{equation}
Let $\sigma=K|L\in\cF^\inte\setminus\cF_\Gamma$. Since $\sigma$ is not a fracture face, the cells $K$ and $L$ are on the same side of the fracture and hence all the degrees of freedom on $\sigma$ are common to $K$ and $L$, and so $\up{\Ksig}\uvec{v}_h=\up{\Lsig}\uvec{v}_h$. Hence, 
\begin{equation}\label{eq:jump.sigma.int}
|\sjump{\up{h}\uvec{v}_h}_{|\sigma}|=|\up{K}\uvec{v}_h-\up{L}\uvec{v}_h|
=|\left(\up{K}\uvec{v}_h-\up{\Ksig}\uvec{v}_h\right)-\left(\up{L}\uvec{v}_h-\up{\Lsig}\uvec{v}_h\right)|.
\end{equation}
For $\sigma\in\cF^\ext\cap\cF_K$, we have  $\bo{v}_{\Ks}=\bo{v}_{\Ke}=\bo{0}$ for all $s\in\cV_\sigma$ and $e\in\cE_\sigma$. Additionally since $\bo{v}_{\Ksig}=\bo{0}$, it follows that $\up{\Ksig}\uvec{v}_h=\bo{0}.$ Thus
\begin{equation}\label{eq:jump.sigma.ext}
|\sjump{\up{h}\uvec{v}_h}_{|\sigma}|=|\up{K}\uvec{v}_h-\up{\Ksig}\uvec{v}_h|.
\end{equation}
Taking the $L^2(\sigma)$-norms of \eqref{eq:jump.sigma.int} and \eqref{eq:jump.sigma.ext}, using triangle inequalities, summing over $\sigma\in\cF\setminus\cF_\Gamma$, invoking the definition \eqref{eq:stabK} of $\mathbb S_K(\cdot,\cdot)$ and using the mesh
regularity to write $h_\sigma\simeq h_K$ whenever $\sigma\in\cF_K$, we conclude
\begin{equation}\label{eq:Korn.bound.S}
\sum_{\sigma\in\cF\setminus\cF_\Gamma}h^{-1}_\sigma\lVert \sjump{\up{h}\uvec{v}_h}\rVert_{L^2(\sigma)}^2\lesssim \sum_{K\in\cM}\sum_{\sigma\in\cF_K}h^{-1}_{\sigma}\lVert \up{K}\uvec{v}_h-\up{\Ksig}\uvec{v}_h\rVert^2_{L^2(\sigma)}
\lesssim\sum_{K\in\cM}\mathbb S_K(\uvec{v}_h,\uvec{v}_h).
\end{equation}
The inequality \eqref{eq:dis_korn} follows from \eqref{eq:Korn.start}, \eqref{eq:symPot.bound}, and \eqref{eq:Korn.bound.S}. \hfill $\Box$

\subsection{Proof of the discrete inf-sup condition}\label{sec:Inf-sup}

The proof of the inf-sup condition requires to interpolate displacements with minimal ${H}^1(\Omega\setminus{\Gamma})$ regularity. To do so, we therefore first describe a Clément-like (averaged) interpolator for such displacements.

For each $K\in\cM$ and $s\in\cV_K$, we take an open set $U_{\cK s}\subset\bigcup_{K\in\cK s}K$ and a function $\overline{\omega}_{\cK s}\in L^\infty(U_{\cK s})$ such that 
\begin{align}
&\lvert U_{\cK s}\rvert \gtrsim \max\limits_{K\in\cK s}\lvert K\rvert,\qquad\lvert\overline{\omega}_{\cK s}\rvert\lesssim 1,\notag\\
&\frac{1}{\lvert U_{\cK s}\rvert}\int_{U_{\cK s}}\overline{\omega}_{\cK s}=1,\qquad\frac{1}{\lvert U_{\cK s}\rvert}\int_{U_{\cK s}}\bx\overline{\omega}_{\cK s}=\bx_s.\label{5.15b}
\end{align}
We refer to \cite[Appendix A]{jhr} for an explicit construction of $(U_{\cK s},\overline{\omega}_{\cK s})$ in the generic case.
The averaged interpolate $\mathbb I^{a}_h\bo{v}\in\dgs$ of $\bo{v}\in H^1_0(\O\setminus{\Gamma})^3$ is then defined by:
\begin{subequations}\label{3.44}
\begin{alignat}{2}
\left(\mathbb I^{a}_h\bo{v}\right)_{\Ks}&=\left(\mathbb I^{a}_h\bo{v}\right)_{\Ke}=\left(\mathbb I^{a}_h\bo{v}\right)_{\Ksig}=\bo{0}&&\quad\forall s\in \mathcal{V}^{\text{ext}}\,,\,e\in\cE^{\ext}\,,\,\sigma\in\cF^{\ext,} \nonumber\\
\left(\mathbb I^{a}_h\bo{v}\right)_{\Ks}&=\frac{1}{\lvert U_{\mathcal K s}\rvert}\int_{U_{\mathcal K s}}\overline{\omega}_{\mathcal Ks}\bo{v}&&\quad\forall s\in \cV^\inte,\nonumber\\
\left(\mathbb I^{a}_h\bo{v}\right)_{\Ke}&=\frac{1}{2\lvert U_{\mathcal K s_1}\rvert}\int_{U_{\mathcal K s_1}}\overline{\omega}_{\mathcal Ks_1}\bo{v}+\frac{1}{2\lvert U_{\mathcal K s_2}\rvert}\int_{U_{\mathcal K s_2}}\overline{\omega}_{\mathcal Ks_2}\bo{v}&&\quad\forall e\in\cE^\inte,~e=[s_1,s_2],\label{5.16c}\\
\left(\mathbb I^{a}_h\bo{v}\right)_{\cK\sigma}  &=\frac{1}{\lvert\sigma\rvert}\int_{\sigma}\bo{v}|_K &&\quad\forall \sigma\in\cF^\inte,\label{5.16d}\\
\left(\mathbb I^{a}_h\bo{v}\right)_K&=\Pi^0_K \bo{v}&&\quad\forall K\in\cM \nonumber.
\end{alignat}
\end{subequations}
The averaged interpolator $\mathbb I^{a} : H^1_0(\Omega \setminus \Gamma)^3 
\to \dgs$ is needed since, unlike $\mathbb{I}_{h}$, it is well-defined for  functions of minimal $H^1(\Omega \setminus \Gamma)$ regularity 
(for which pointwise evaluation is not available). Its construction using local weighted averages over patches $U_{\Ks}$ lying only on one side of $\Gamma$ ensures that these averages avoid the jumps of the interpolated displacement. Its two key properties used in the proof of Theorem~\ref{thm:inf-sup} 
are: (i) exact recovery of the duality pairing, i.e., $\int_\Gamma \blambda_h\cdot\sjump{\mathbb I^{a}_h\bo{v}}_h= \blambda_h \cdot \sjump{\bo{v}}$ for all $\blambda_h \in {\bf M}_{h}$;  and (ii) uniform continuity for the $H^1(\Omega\setminus\Gamma)$ norm, proved in the next theorem.

\begin{prop}[Stability of the averaged interpolator]\label{prop:stab.avginter}
For $\bo{v}\in H^1_0(\O\setminus{\Gamma})^3$, it holds
\begin{equation}\label{eq:stab_avgint.est1}
  \lVert \mathbb I^{a}_h\bo{v}\rVert_{1,h}\lesssim \lVert\nabla\bo{v}\rVert_{L^2(\O)}.
\end{equation}
\end{prop}
\begin{proof}
Let $\bo{v}\in H^1_0(\O\setminus{\Gamma})^3$ and $K\in\cM$. Let 
$$
  \cN(K)\coloneqq \bigcup\limits_{s\in\cV_K}\bigcup\limits_{L\in\cK s}L
$$
be the patch around $K$ made of all the cells in $\cK s$ for each $s$ vertex of $K$. By definition of $\cK s$, we have $\bo{v}\in H^1(\cN(K))^3$. Denote by $\bo{q}$ the $L^2(\cN(K))^3$-projection of $\bo{v}$ on $\mathbb P^1(\cN(K))^3$. This projection enjoys the following approximation and stability properties \cite[Theorem 1.45]{hho-book}:
\begin{align}
\lVert \bo{q}-\bo{v}\rVert_{L^2(\cN(K))}&\lesssim h_K\lVert \nabla \bo{v}\rVert_{L^2(\cN(K))},\label{eq:stab_avgint.est2}\\
\lVert (\bo{q}-\bo{v})|_K\rVert_{L^2(\sigma)}&\lesssim h^{1/2}_K\lVert \nabla \bo{v}\rVert_{L^2(\cN(K))}\quad\forall\sigma\in\cF_K,\label{eq:stab_avgint.est3}\\
\lVert \nabla \bo{q}\rVert_{L^2(\cN(K))}&\lesssim \lVert \nabla \bo{v}\rVert_{L^2(\cN(K))}.\label{eq:stab_avgint.est4}
\end{align}

Recalling \cite[Lemma 3]{pjddr} and from Lemma \ref{lem:prop.reconstructions} we have $\up{Ke}\mathbb I_h\bo{q}=\bo{q}$, $\up{K\sigma}\mathbb I_h\bo{q}=\bo{q}$,  $\dgr{K}\mathbb I_h\bo{q}=\nabla \bo{q}$, and $\up{K}\mathbb I_h\bo{q}=\bo{q}$.
Thus, the definition \eqref{eq:stabK} of $\mathbb S_K(\cdot,\cdot)$ gives
\begin{equation}\label{eq:stab.consistency}
\mathbb S_K(\mathbb I_h\bo{q},\uvec{w}_h)=0\qquad\forall\uvec{w}_h\in\dgs.
\end{equation}
The definition \eqref{discrete_norm} of $\lVert \cdot\rVert_{1,K}$, then gives
\begin{equation}\label{eq:stab_avgint.est5}
\lVert \mI{q}\rVert_{1,K}=\lVert \nabla\bo{q}\rVert_{L^2(K)}\lesssim \lVert\nabla \bo{v}\rVert_{L^2(\cN(K))}.
\end{equation}
To prove \eqref{eq:stab_avgint.est1}, we bound each type of degree of freedom (vertex, edge, face, element), then use \eqref{eq:dofnorm} to conclude.

From \cite[Proposition 5.5]{pjddr}, for any $s\in\cV_K$ the following inequality holds:
\begin{align}\label{eq:stab_avgint.est6}
\lvert (\mI{q})_{\cK s}-(\aI{v})_{\cK s}\rvert\lesssim h^{-1/2}_K\lVert\nabla\bo{v}\rVert_{L^2(\cN(K))}.
\end{align}
If $e\in\cE^{\ext}_K$, then $(\aI{v})_{\cK e}=\bo{0}$ and $(\mI{q})_{\cK e}=\Pi^0_e(\bo{q})$. Take $\sigma\in\cF^\ext$ that contains $e$ and write
\begin{align*}
\lvert (\mI{q})_{\cK e}-(\aI{v})_{\cK e}\rvert&=\lvert \Pi^0_e(\bo{q})\rvert\leq \lVert \bo{q}\rVert_{L^\infty(e)}\leq \lVert \bo{q}\rVert_{L^\infty(\sigma)}\lesssim h^{-1}_\sigma\lVert \bo{q}\rVert_{L^2(\sigma)}\notag\\
&\lesssim h^{-1}_\sigma h^{1/2}_K\lVert \nabla \bo{q}\rVert_{L^2(\cN(K))}\lesssim h^{-1/2}_K\lVert \nabla \bo{v}\rVert_{L^2(\cN(K))},
\end{align*}
where we have used the definition of interpolation in the first inequality, the inverse Lebesgue inequality \cite[Lemma 1.25]{hho-book} for the third inequality, the discrete trace inequality \cite[Lemma 1.32]{hho-book} in the fourth inequality, and the mesh regularity assumption (to write $h_K\approx h_\sigma$) along with \eqref{eq:stab_avgint.est4} in the last inequality.

If $e=[s_1,s_2]\in \cE^\inte_K$, then \eqref{5.15b}, \eqref{5.16c}, and the fact that $\bo{q}$ is linear ensure that
\begin{equation*}
(\mI{q})_{\cK e}=\Pi^0_e(\bo{q})=\frac{1}{2\lvert U_{\mathcal K s_1}\rvert}\int_{U_{\mathcal K s_1}}\overline{\omega}_{\mathcal Ks_1}\bo{q}+\frac{1}{2\lvert U_{\mathcal K s_2}\rvert}\int_{U_{\mathcal K s_2}}\overline{\omega}_{\mathcal Ks_2}\bo{q}.
\end{equation*}
Hence
\begin{align*}
\lvert (\mI{q})_{\cK e}-(\aI{v})_{\cK e}\rvert&=\left\vert\sum_{j=1}^2 \frac{1}{2\lvert U_{\cK s_j}\rvert}\int_{U_{\cK s_j}}\overline{\omega}_{\cK s_j}(\bo{q}-\bo{v})\right\vert\notag\\
&\lesssim\sum_{j=1}^2\lvert U_{\cK s_j}\rvert^{-1/2}\lVert \bo{q}-\bo{v}\rVert_{L^2(U_{\cK s_j})}\lesssim h^{-1/2}_K\lVert\nabla \bo{v}\rVert_{L^2(\cN (K))},
\end{align*}
where the first inequality follows from the bound $\lvert \overline{\omega}_{\cK s_j}\rvert\lesssim 1$ and the Cauchy–Schwarz inequality, while the conclusion is obtained using $\lvert U_{\cK s_j}\rvert \gtrsim \lvert K\rvert$, $U_{\cK s_j}\subset \cN(K)$, and \eqref{eq:stab_avgint.est2}.

If $\sigma\in\cF^{\ext}_K$, then $(\aI{v})_{\cK \sigma}=\bo{0}$ and $(\mI{q})_{\cK \sigma}=\Pi^0_\sigma(\bo{q})$, so
\begin{equation*}
\lvert (\mI{q})_{\cK \sigma}-(\aI{v})_{\cK \sigma}\rvert=\lvert \Pi^0_\sigma(\bo{q})\rvert\leq \lVert \bo{q}\rVert_{L^\infty(\sigma)}
\lesssim h^{-1/2}_K\lVert \nabla \bo{v}\rVert_{L^2(\cN(K))}.
\end{equation*}

If $\sigma\in\cF^{\inte}_K$, then
\begin{equation*}
\lvert (\mI{q})_{\cK \sigma}-(\aI{v})_{\cK \sigma}\rvert=\left\vert \frac{1}{\lvert \sigma\rvert}\int_\sigma (\bo{q}-\bo{v})|_K \right\vert\lesssim \lvert h_\sigma\rvert^{-1}\lVert \bo{q}-\bo{v}\rVert_{L^2(\sigma)}\overset{\eqref{eq:stab_avgint.est3}}\lesssim h^{-1/2}_K\lVert \nabla \bo{v}\rVert_{L^2(\cN(K))}.
\end{equation*}

Finally, for the cell degree of freedom, we write
\begin{align}\label{eq:stab_avgint.est7}
\lvert (\mI{q})_{K}-(\aI{v})_{K}\rvert&=\left\vert \frac{1}{\lvert K\rvert}\int_K\bo{q}-\bo{v} \right\vert\lesssim \lvert h_K\rvert^{-3/2}\lVert \bo{q}-\bo{v}\rVert_{L^2(K)}\overset{\eqref{eq:stab_avgint.est2}}\lesssim h^{-1/2}_K\lVert \nabla \bo{v}\rVert_{L^2(\cN(K))}.
\end{align}
Gathering \eqref{eq:stab_avgint.est6}-\eqref{eq:stab_avgint.est7} and invoking the bound \eqref{eq:dofnorm} with $\uvec{w}_h=\mI{q}-\aI{v}$ yields
\begin{align}\label{eq:stab_avgint.est8}
\lVert \mI{q}-\aI{v}\rVert_{1,K}\lesssim \lVert\nabla \bo{v}\rVert_{L^2(\cN(K))}.
\end{align}
Combining \eqref{eq:stab_avgint.est5} and \eqref{eq:stab_avgint.est8}, squaring, summing over $K\in\cM$ and gathering the integrals by cells and using the mesh regularity to see that $\#\{K\in\cM\,:\,L\in\cN_K\}\lesssim 1$ for all $L\in\cM$, we infer
\begin{align*}
\lVert \aI{v}\rVert^2_{1,h}&\lesssim \sum_{K\in\cM}\lVert\nabla\bo{v}\rVert^2_{L^2(\cN(K))}\lesssim \lVert \nabla \bo{v}\rVert^2_{L^2(\O)}.
\qedhere
\end{align*}
\end{proof}
With the establishment of Proposition \ref{prop:stab.avginter}, we are now ready to prove the discrete inf-sup condition (Theorem \ref{thm:inf-sup}).

\begin{proof}[\textit{\textbf {Proof of Theorem~\ref{thm:inf-sup}:}}]
Let $\blambda_h\in\bo{M}_h$. Recalling the definition \eqref{eq:def.dis{-1/2}norm} of the ${H}^{-1/2}$-like norm, it suffices to prove that for each $\bo{v}\in {H}^1_0(\O\setminus\Gamma)^{3},$ there exists some $\uvec{v}_h\in\ddgs$ such that 
\begin{align}\label{eq:inf-sup}
\frac{\int_\Gamma \blambda_h\cdot\sjump{\uvec{v}_h}_h}{\lVert \uvec{v}_h\rVert_{1,h}}\gtrsim \frac{\int_{\Gamma} \blambda_h\cdot{\sjump{\bo{v}}}}{\lVert \bo{v}\rVert_{{H}^1(\O\setminus \Gamma)}}.
\end{align}
Let us consider $\uvec{v}_h=\aI{\bo{v}}$ with the interpolator defined by \eqref{3.44}.  
Since $\blambda_h$ is piecewise constant on $\cF_\Gamma$, we have
$$
\int_\Gamma \blambda_h\cdot\sjump{\uvec{v}_h}_h=\sum_{\sigma\in\cF_\Gamma}\int_\sigma \blambda_h\cdot\sjump{\aI{\bo{v}}}_\sigma\overset{\eqref{eq:discrete.jump},\eqref{eq:Pisigma.Pot},\eqref{5.16d}}=\sum_{\sigma\in\cF_{\Gamma}}\int_{\sigma}\blambda_h\cdot\sjump{\bo{v}}
=\int_{\Gamma}\blambda_h\cdot\sjump{\bo{v}}.
$$
Dividing throughout by $\lVert \uvec{v}_h\rVert_{1,h}=\lVert \aI{\bo{v}}\rVert_{1,h}$ and using \eqref{eq:stab_avgint.est1} to infer that $\lVert \aI{\bo{v}}\rVert_{1,h}\lesssim \lVert \bo{v}\rVert_{{H}^1(\O\setminus\Gamma)}$ proves \eqref{eq:inf-sup}.
\end{proof}

\subsection{Proof of abstract error bounds}

To derive the abstract error bound in Theorem \ref{thm:abstract.error.estimate}, we need to establish the following lemmas. 
\begin{lemma}[Bound on the discrete stress]
For all $\uvec{v}_{h}\in\ddgs$, it holds: 
\begin{equation}\label{eq:bound.sigma.h}
\norm{L^2(\O)}{\bsig_h(\uvec{v}_h)}\lesssim \sqrt{\mu_{2}} \left(1+\left(\frac{\nu}{1-2\nu}\right)^{\frac{1}{2}} \right) \norm{\mathrm{en},h}{\uvec{v}_h},
\end{equation}
where the hidden constant in $\lesssim$ is independent of $E$ and $\nu.$
\end{lemma}

\begin{proof}
First, recall the definition of $\bsig_h(\cdot)$ from \eqref{eq:def.sigma.h} to get
\begin{equation*}
    \bsig_h(\uvec{v}_h)= \frac{E}{1+\nu}\left(\be_h(\uvec{v}_h)+\frac{\nu}{1-2\nu} \DIV_h(\uvec{v}_h)\Id\right).
\end{equation*}
An application of the triangle inequality in the above equation leads to
\begin{equation}
\norm{L^2(\O)}{\bsig_h(\uvec{v}_h)} \lesssim \frac{E}{1+\nu}\left(\norm{L^2(\O)}{\be_h(\uvec{v}_h)}+\frac{\nu}{1-2\nu}\norm{L^2(\O)}{ \DIV_h(\uvec{v}_h)}\right). \label{l2_bd1}
\end{equation}
The definition of $\norm{\mathrm{en},h}{{\cdot}}$ from \eqref{eq:def.inner.Eh} and using $(a+b)^{1/2}\ge \frac12 (a^{1/2}+b^{1/2})$, we have
$$
\norm{\mathrm{en},h}{\uvec{v}_h}\ge \frac12 \left(\frac{E}{1+\nu}\right)^{\frac{1}{2}}
\left(\norm{L^2(\O)}{\be_h(\uvec{v}_h)}+\left(\frac{\nu}{1-2\nu}\right)^{\frac{1}{2}}\norm{L^2(\O)}{\DIV_h(\uvec{v}_h)}\right),
$$
which gives
$$
\norm{L^2(\O)}{\be_h(\uvec{v}_h)}+\left(\frac{\nu}{1-2\nu}\right)^{\frac{1}{2}}\norm{L^2(\O)}{\DIV_h(\uvec{v}_h)}\lesssim
\left(\frac{E}{1+\nu}\right)^{-\frac{1}{2}}\norm{\mathrm{en},h}{\uvec{v}_h}.
$$
Plugging this relation into \eqref{l2_bd1} leads to
$$
\norm{L^2(\O)}{\bsig_h(\uvec{v}_h)}
\lesssim \left(\frac{E}{1+\nu}\right)^{\frac{1}{2}}\left(1+\left(\frac{\nu}{1-2\nu}\right)^{\frac{1}{2}}\right)\norm{\mathrm{en},h}{\uvec{v}_h},
$$
The above inequality with $\frac{E}{1+\nu} \leq \mu_{2}$ implies \eqref{eq:bound.sigma.h}. 
\end{proof}

\begin{lemma}[Fortin interpolator]\label{lem:fortin}
Let $K\in\cM$ and $\bos{\xi}\in \P^1(K)^{3\times 3}$. Then, we have 
\begin{align*}
\int_K \big(\DIV(\bo{u})-\DIV_h(\mI{u})\big)\Id\colon\bos{\xi}=0.
\end{align*}
Consequently,\
\begin{equation}\label{eq:Pi1.proj.sigma}
\int_\Omega (\bsig(\bo{u})-\bsig_h(\mI{u}))\colon \bos{\xi}=\frac{E}{1+\nu}\int_\Omega \big(\be(\bo{u})-\be_h(\mI{u})\big)\colon\bos{\xi}.
\end{equation}
\end{lemma}

\begin{proof}
Let $\bos{\xi}\in \P^1(K)^{3\times 3}$. We have $\tr\bos{\xi}\in \P^1(K)$. Thus,
\begin{align*}
\int_K (\tr \bos{\epsilon}{}&(\bo{u})-{\rm{div}}_h(\mI{u}))\Id\colon\bos{\xi}=\int_K (\tr \bos{\epsilon}(\bo{u})-\tr\bos{\epsilon}_h(\mI{u}))\tr\bos{\xi}\\
    &=\int_K (\bos{\epsilon}(\bo{u})-\bos{\epsilon}_h(\mI{u}))\colon(\tr\bos{\xi})\Id=\int_K (\nabla \bo{u}-\dgr{h}(\mI{u}))\colon(\tr\bos{\xi})\Id\\
\overset{\text{IBP}, \eqref{eq:G1K},\eqref{eq:def.Ih.K}}&=
\sum_{\sigma\in\cF_K}\int_\sigma \bo{u}|_K\cdot (\tr\bos{\xi})\bo{n}_{\Ksig}-\int_K \bo{u}\cdot\DIV((\tr\bos{\xi})\Id)+\int_K \left(\Pi^0_K\bo{u} +\widetilde{\bo{u}}_K\right)\cdot\DIV((\tr\bos{\xi})\Id)\\
&
\qquad\qquad-\sum_{\sigma\in\cF_K}\int_\sigma \up{\Ksig}\mI{u}\cdot (\tr\bos{\xi})\bo{n}_{\Ksig}
\\
&
=\cancel{\int_K \left(\Pi^0_K\bo{u} -\bo{u}\right)\cdot\DIV((\tr\bos{\xi})\Id)}+\int_K \widetilde{\bo{u}}_K\cdot\DIV((\tr\bos{\xi})\Id)
\\
&\qquad+\sum_{\sigma\in\cF_K}\int_\sigma (\bo{u}|_K-\up{\Ksig}\mI{u})\cdot (\tr\bos{\xi})\bo{n}_{\Ksig}=0,
\end{align*}
where the cancellation in the fifth equation is justified by the definition of $\Pi^0_K$ and $\DIV((\tr \bos{\xi})\Id)\in \P^0(K)^3$ while, in the conclusion, we have used the definition \eqref{eq:tildevK} of $\widetilde{\bo{u}}_K$ with $\phi=\tr\bos{\xi}$, after noticing that $\DIV((\tr\bos{\xi})\Id)=\nabla \tr\bos{\xi}$.  
\end{proof}

\begin{lemma}[$L^2$-to-$H^{-1/2}$ projection stability]\label{lma_Pi_lambda_est}
Let $\bos{\chi}\in L^2(\Gamma)^{3}.$ Then, we have
\begin{equation*}
\lVert \bos{\chi}-\Pi^0_{\cF_\Gamma}\bos{\chi}\rVert_{-1/2,\Gamma}\lesssim h^{1/2}\lVert \bos{\chi}-\Pi^0_{\cF_\Gamma}\bos{\chi}\rVert_{L^2(\Gamma)}.
\end{equation*}
\end{lemma}

\begin{proof}
Let $\bo{v}\in {H}^1_0(\O\setminus\Gamma)^{3}\setminus\{0\}$. Then,
\begin{align*}
\sum_{\sigma\in\cF_\Gamma}\left\Vert \sjump{\bo{v}}-\Pi^0_{\sigma}\sjump{\bo{v}}\right\Vert_{L^2(\sigma)}
&\le\sum_{\sigma\in\cF_\Gamma}\sum_{K\in\cM_\sigma}\left\Vert\bo{v}|_K-\Pi^0_{\sigma}(\bo{v}|_K)\right\Vert_{L^2(\sigma)}\notag\\
&=\sum_{\sigma\in\cF_\Gamma}\sum_{K\in\cM_\sigma}\left\Vert\bo{v}|_K-\Pi^0_K\bo{v} - \Pi^0_{\sigma}(\bo{v}|_K-\Pi^0_K\bo{v})\right\Vert_{L^2(\sigma)}\notag\\  
&\le \sum_{\sigma\in\cF_\Gamma}\sum_{K\in\cM_\sigma}2\left\Vert\bo{v}|_K-\Pi^0_{K}\bo{v}\right\Vert_{L^2(\sigma)}\notag\\
    &\lesssim \sum_{\sigma\in\cF_\Gamma}h_\sigma^{1/2}\sum_{K\in\cM_\sigma}\lvert \bo{v}\rvert_{H^1(K)}\lesssim h^{1/2} \lVert \bo{v}\rVert_{{H}^1(\O\setminus\Gamma)}.
\end{align*}
In deriving these estimates, we invoke the definition of $\sjump{\bo{v}}$ and triangle inequalities (first line), use $\Pi^0_{\sigma}(\Pi^0_K\bo{v})=\Pi^0_K\bo{v}$ since $(\Pi^0_K\bo{v})|_\sigma\in\mathbb P^0(\sigma)^3$ (second line), apply the triangle inequality together with the $L^2(\sigma)$-boundedness of $\Pi^0_\sigma$ (third line), and the approximation properties of $\Pi^0_K$ given in \cite[Theorem 1.45]{hho-book} (fourth line). Now, using orthogonality of the projection operator $\Pi^0_{\cF_\Gamma}$, the Cauchy-Schwarz inequality, and the previous bound, we conclude
\begin{align*}
\int_\Gamma \big(\bos{\chi}-\Pi^0_{\cF_\Gamma}\bos{\chi}\big)\cdot\sjump{\bo{v}}&=\int_\Gamma \big(\bos{\chi}-\Pi^0_{\cF_\Gamma}\bos{\chi}\big)\cdot\big(\sjump{\bo{v}}-\Pi^0_{\cF_\Gamma}\sjump{\bo{v}}\big)\\
&\lesssim \lVert \bos{\chi}-\Pi^0_{\cF_\Gamma}\bos{\chi}\rVert_{L^2(\Gamma)}\bigg(\sum_{\sigma\in\cF_\Gamma}\left\Vert \sjump{\bo{v}}-\Pi^0_{\sigma}\sjump{\bo{v}}\right\Vert_{L^2(\sigma)}\bigg)\\
&\lesssim h^{1/2}\lVert \bos{\chi}-\Pi^0_{\cF_\Gamma}\bos{\chi}\rVert_{L^2(\Gamma)}\lVert \bo{v}\rVert_{{H}^1(\O\setminus\Gamma)}.
\end{align*}
The lemma follows from Definition \ref{def:H^{-1/2}}.
\end{proof}

\begin{proof}[\textit{\textbf {Proof of Theorem~\ref{thm:abstract.error.estimate}:}}]
We first notice that the regularity assumption ensures that $\bo{u}\in C^0_0(\overline{\O}\setminus\Gamma)^3$. Let us start by estimating the norm of $\uvec{u}_h-\mI{u}$. By definition \eqref{eq:def.inner.Eh} of the inner product corresponding to $\norm{\mathrm{en},h}{{\cdot}}$, we have
\begin{align}
\norm{\mathrm{en},h}{\uvec{u}_h-\mI{u}}^2&=\int_{\Omega}\bsig_h(\uvec{u}_h-\mI{u})\colon\be_h(\uvec{u}_h-\mI{u})+ \mu_{1}\mathbb S_h(\uvec{u}_h-\mI{u},\uvec{u}_h-\mI{u})\nonumber\\
  \overset{\eqref{eq:scheme.displacement}}&=\sum_{K\in\cM}\int_K\bo{f}\cdot(\bo{u}_K-\I_K{\bo{u}})-\int_\Gamma \blambda_h\cdot\sjump{\uvec{u}_h-\mI{u}}_h\nonumber\\
  &\quad-\int_{\Omega}\bsig_h(\mI{u})\colon\be_h(\uvec{u}_h-\mI{u})-\mu_{1}\mathbb S_h(\mI{u},\uvec{u}_h-\mI{u}).
  \label{eq:abstract.est1}
\end{align}
  The definition \eqref{eq:adjoint.consist} of $w_h$ together with $\bo{f}=-\DIV(\bsig(\bo{u}))$ and $\blambda=-\gamma^+_{\bo{n}}\bsig(\bo{u})$ yields
  \begin{align*}
  0 ={}&w_h(\bsig(\bo{u}),\uvec{u}_h-\mI{u})-\sum_{K\in\cM}\int_K \bo{f}\cdot (\bo{u}_K-\I_K\bo{u})+\int_\Gamma \Pi^0_{\cF_\Gamma}\blambda\cdot\sjump{\uvec{u}_h-\mI{u}}_h\\
  &  +\int_\Omega\bsig(\bo{u}):\be_h(\uvec{u}_h-\mI{u}).
  \end{align*}
  Adding this relation to \eqref{eq:abstract.est1} leads to
\begin{align}
  \norm{\mathrm{en},h}{\uvec{u}_h-\mI{u}}^2&=w_h(\bsig(\bo{u}),\uvec{u}_h-\mI{u})+\int_\Gamma (\Pi^0_{\cF_\Gamma}\blambda-\blambda_h)\cdot\sjump{\uvec{u}_h-\mI{u}}_h\notag\\
  &\quad+\int_{\Omega}\big(\bsig(\bo{u})-\bsig_h(\mI{u})\big)\colon\be_h(\uvec{u}_h-\mI{u})-\mu_{1}\mathbb S_h(\mI{u},\uvec{u}_h-\mI{u})\notag\\
  \overset{\eqref{eq:adjoint.consist},\eqref{eq:norm_energy},\eqref{eq:Pi1.proj.sigma}}&\leq \left(\frac{{\widetilde{C}}}{\mu_{1}}\right)^{\frac{1}{2}}\lVert\uvec{u}_h-\mI{u}\rVert_{\mathrm{en},h}\bos{\cW}_h(\bsig(\bo{u}))-\int_\Gamma (\Pi^0_{\cF_\Gamma}\blambda-\blambda_h)\cdot \sjump{\mI{u}}_h\notag\\
  &\quad + \frac{E}{1+\nu}\int_{\Omega}\big(\be(\bo{u})-\be_h(\mI{u})\big)\colon\be_h(\uvec{u}_h-\mI{u}) 
  -\mu_{1}\mathbb S_h(\mI{u},\uvec{u}_h-\mI{u}),
  \label{eq:abstract.est.u.x}
\end{align}
where in the last inequality we have used that $\bos{\epsilon}_K(\uvec{u}_h-\mI{u})\in\mathbb P^1(K)^{3\times 3}$ for all $K\in\cM$ and, since $\Pi^0_{\cF_\Gamma}\blambda\in\cC_{f,h}$, the relation \eqref{eq:scheme.lagrange} to write $\int_\Gamma (\Pi^0_{\cF_\Gamma}\blambda-\blambda_h)\cdot\sjump{\uvec{u}_h}_h\leq 0$.
Using the Cauchy--Schwarz inequality on the last two terms, and Young's inequality $\left(ab\leq a^2+b^2/4\right)$, we get 
\begin{align*}
\left(\frac{{\widetilde{C}}}{\mu_{1}}\right)^{1/2}\lVert\uvec{u}_h-\mI{u}\rVert_{\mathrm{en},h}\bos{\cW}_h(\bsig(\bo{u})) &\leq \frac{{\widetilde{C}}}{\mu_{1}}\bos{\cW}_h(\bsig(\bo{u}))^2+\frac{1}{4}\lVert\uvec{u}_h-\mI{u}\rVert_{\mathrm{en},h}^2
\end{align*}
while, for the stabilisation term,
\begin{align*}
\left|\mu_1\mathbb S_h(\mI{u},\uvec{u}_h-\mI{u})\right| &\leq \mu_1\mathbb S_h(\mI{u},\mI{u}) + \frac{\mu_1}{4}\mathbb S_h(\uvec{u}_h-\mI{u},\uvec{u}_h-\mI{u})\\
\overset{\mu_1\leq \mu_2}&\leq \mu_2\mathbb S_h(\mI{u},\mI{u}) + \frac{\mu_1}{4}\mathbb S_h(\uvec{u}_h-\mI{u},\uvec{u}_h-\mI{u}),
\end{align*}
and, for the symmetric gradient term,
\begin{align*}
\Bigg| \frac{E}{1+\nu}\int_{\Omega}{}&\big(\be(\bo{u})-\be_h(\mI{u})\big)\colon\be_h(\uvec{u}_h-\mI{u})\Bigg|\\
&\leq \frac{E}{1+\nu}\bigg(\lVert\be(\bo{u})-\be_h(\mI{u})\rVert_{L^2(\O)}^{2}+\frac{1}{4}\lVert\be_h(\uvec{u}_h-\mI{u})\rVert_{L^2(\O)}^{2}\bigg)\\
\overset{\frac{E}{1+\nu}\leq \mu_2}&\leq \mu_2\lVert\be(\bo{u})-\be_h(\mI{u})\rVert_{L^2(\O)}^{2}+\frac{E}{4(1+\nu)}\lVert\be_h(\uvec{u}_h-\mI{u})\rVert_{L^2(\O)}^{2}.
\end{align*} 
Plugging these three inequalities into \eqref{eq:abstract.est.u.x}, then using the definition \eqref{eq:primal.consist} of $\bos{\fC}_h$ and 
$$
\frac{E}{4(1+\nu)}\lVert\be_h(\uvec{u}_h-\mI{u})\rVert_{L^2(\O)}^{2}+\frac{\mu_1}{4}\mathbb S_h(\uvec{u}_h-\mI{u},\uvec{u}_h-\mI{u})
\overset{\eqref{eq:def.inner.Eh}}\leq \frac{1}{4}\norm{\mathrm{en},h}{ \uvec{u}_h-\mI{u}}^2,
$$ 
we obtain
\begin{equation}\label{eq:abstract.est3}
\norm{\mathrm{en},h}{ \uvec{u}_h-\mI{u}}^2\leq \frac{2\widetilde{C}}{\mu_{1}}\bos{\cW}_h(\bsig(\bo{u}))^2+2\mu_{2}\bos{\fC}_h(\bo{u},\mI{u})^2-2\int_\Gamma (\Pi^0_{\cF_\Gamma}\blambda-\blambda_h)\cdot\sjump{\mI{u}}_h.
\end{equation}
Let us now consider the last addend. For all $\sigma\in\cF_\Gamma$, the definition \eqref{eq:discrete.jump} of $\sjump{\cdot}_{\sigma}$ gives
\begin{align*}
\int_\sigma (\Pi^0_{\sigma}\blambda-\blambda_\sigma)\cdot\sjump{\mI{u}}_\sigma
&=
\int_\sigma (\Pi^0_{\sigma}\blambda-\blambda_\sigma)\cdot(\up{\Ksig}\mI{u}-\up{\Lsig}\mI{u}\big)\\
\overset{\eqref{eq:Pisigma.Pot},\eqref{eq:def.Ih.sigma}}&=
\int_\sigma (\Pi^0_{\sigma}\blambda-\blambda_\sigma)\cdot(\Pi^0_\sigma(\bo{u}|_K)-\Pi^0_\sigma(\bo{u}|_L)\big)\\
&=\int_\sigma (\Pi^0_{\sigma}\blambda-\blambda_\sigma)\cdot \sjump{\bo{u}},
\end{align*}
where the second and third equalities are further justified by $\Pi^0_{\sigma}\blambda-\blambda_\sigma\in\P^0(\sigma)^3$ together with the definition of $\Pi^0_{\sigma}$. Summing these equations over $\sigma\in\cF_\Gamma$, we obtain
\begin{align*}
-\int_\Gamma (\Pi^0_{\cF_\Gamma}\blambda-\blambda_h)\cdot\sjump{\mI{u}}_h
&=\int_\Gamma (\blambda_h-\Pi^0_{\cF_\Gamma}\blambda)\cdot \sjump{\bo{u}}\\
&=\underbrace{\int_\Gamma (\blambda_h-\blambda)\cdot\sjump{\bo{u}}}_{\le 0\text{ by \eqref{eq:model.contact} and $\blambda_h\in\cC_{f,h}\subset\cC_f$}} + \int_\Gamma(\blambda-\Pi^0_{\cF_\Gamma}\blambda)\cdot \sjump{\bo{u}}\\
&\le \int_\Gamma(\blambda-\Pi^0_{\cF_\Gamma}\blambda)\cdot \big(\sjump{\bo{u}} - \Pi^0_{\cF_\Gamma}\sjump{\bo{u}}\big),
\end{align*}
where, in the conclusion, we have introduced $\Pi^0_{\cF_\Gamma}\sjump{\bo{u}}$ using the definition of $\Pi^0_{\cF_\Gamma}\blambda$.
Plugging this upper bound into \eqref{eq:abstract.est3} leads to
\begin{equation*}
\norm{\mathrm{en},h}{ \uvec{u}_h-\mI{u}}^2
\lesssim \frac{1}{\mu_{1}}\bos{\cW}_h(\bsig(\bo{u}))^2+\mu_{2}\bos{\fC}_h(\bo{u},\mI{u})^2+\int_\Gamma (\blambda-\Pi^0_{\cF_\Gamma}\blambda)\cdot\big(\sjump{\bo{u}}-\Pi^0_{\cF_\Gamma}\sjump{\bo{u}}\big),
\end{equation*}
where $\lesssim$ absorbs the quantity $\max\{2\widetilde{C},2\}$ which is independent from $E$ and $\nu$. This proves the estimate \eqref{eq:abstract.error.estimate}.

We now turn to the error on the Lagrange multiplier. Using the definition of $w_h$ in \eqref{eq:adjoint.consist} (together with $\gamma_n^+\bsig(\bo{u})=-\blambda$) and the scheme \eqref{eq:scheme.displacement}, we have for all $\uvec{w}_h\in \ddgs$:
\begin{align*}
  \int_{\Gamma}{}&(\Pi^0_{\cF_\Gamma}\blambda-\blambda_h)\cdot\sjump{\uvec{w}_h}_h=-w_h(\bsig(\bo{u}),\uvec{w}_h)-\int_\O\big(\bsig(\bo{u})-\bsig_h(\uvec{u}_h)\big)\colon\be_h(\uvec{w}_h)+\mu_{1} \mathbb S_h(\uvec{u}_h,\uvec{w}_h)\\
  &=-w_h(\bsig(\bo{u}),\uvec{w}_h)-\int_\O\big(\bsig(\bo{u})-\bsig_h(\mI{u})\big)\colon\be_h(\uvec{w}_h)+\mu_{1}\mathbb S_h(\uvec{u}_h,\uvec{w}_h)\\
  &\qquad+\int_\O\big(\bsig_h(\uvec{u}_h)-\bsig_h(\mI{u})\big)\colon\be_h(\uvec{w}_h)\\
  \overset{\eqref{eq:Pi1.proj.sigma}}&=-w_h(\bsig(\bo{u}),\uvec{w}_h)-\frac{E}{1+\nu}\int_{\Omega}\big(\be(\bo{u})-\be_h(\mI{u})\big)\colon\be_h(\uvec{w}_h)+\mu_{1}\mathbb S_h(\mI{u},\uvec{w}_h)\\
  &\qquad+\int_\O\big(\bsig_h(\uvec{u}_h)-\bsig_h(\mI{u})\big)\colon\be_h(\uvec{w}_h)+\mu_{1}\mathbb S_h(\uvec{u}_h-\mI{u},\uvec{w}_h)\\
  \overset{\eqref{discrete_norm},\eqref{eq:bound.sigma.h}}&\lesssim -w_h(\bsig(\bo{u}),\uvec{w}_h)+\mu_{2}\lVert\nabla \bo{u}-\dgr{h}\mI{u}\rVert_{L^2(\O)}\lVert\uvec{w}_h\rVert_{1,h}
  \\
  &\quad +\mu_{2} \mathbb S_h(\mI{u},\mI{u})^{\frac12}\lVert\uvec{w}_h\rVert_{1,h}+ \sqrt{\mu_{2}}\left(1+\left(\frac{\nu}{1-2\nu}\right)^{\frac{1}{2}}\right) \lVert \uvec{u}_h-\mI{u}\rVert_{\mathrm{en},h}\lVert \uvec{w}_h\rVert_{1,h}
  \\&\quad +\mu_{1} \mathbb S_h(\uvec{u}_h-\mI{u},\uvec{u}_h-\mI{u})^{\frac12}\lVert\uvec{w}_h\rVert_{1,h}\\
  \overset{\eqref{eq:def.inner.Eh}} &\lesssim -w_h(\bsig(\bo{u}),\uvec{w}_h)+\lVert\uvec{w}_h\rVert_{1,h}\bigg[\mu_{2} \lVert\nabla \bo{u}-\dgr{h}\mI{u}\rVert_{L^2(\O)}\\
  &\quad\quad+\sqrt{\mu_{2}}\left(1+\left(\frac{\nu}{1-2\nu}\right)^{\frac{1}{2}}\right) \lVert\uvec{u}_h-\mI{u}\rVert_{\mathrm{en},h}
  + \sqrt{\mu_{2}}\lVert\uvec{u}_h-\mI{u}\rVert_{\mathrm{en},h}+\mu_{2} \mathbb S_h(\mI{u},\mI{u})^{\frac12}\bigg],
\end{align*}
where we have used $\mu_{1}\leq \mu_2$, $\frac{E}{1+\nu}\leq \mu_{2}$ and, to pass to the second line, we have introduced $\pm \bsig_h(\mI{u})\colon\be_h(\uvec{w}_h)$ in the second term of the first right-hand side. 
Using the discrete inf-sup condition \eqref{eq:def.inf-sup} and the definitions \eqref{eq:primal.consist} of $\bos{\fC}_h$ and  \eqref{eq:adjoint.consist} of $\bos{\cW}_h$, we infer
\begin{equation}\label{4.11_lm}
\lVert \Pi^0_{\cF_\Gamma}\blambda-\blambda_h\rVert_{-1/2,\Gamma}\lesssim \bos{\cW}_h(\bsig(\bo{u}))+\mu_{2}\bos{\fC}_h(\bo{u},\mI{u})+\sqrt{\mu_{2}}\left(1+\left(\frac{\nu}{1-2\nu}\right)^{\frac{1}{2}}\right)\lVert\uvec{u}_h-\mI{u}\rVert_{\mathrm{en},h}.
\end{equation}
 The bound \eqref{eq:abstract.error.estimate_lam} follows from the triangle inequality and  \eqref{4.11_lm}, which concludes the lemma. 
\end{proof}

\subsection{Proof of the error bound}

In this section, we derive explicit bounds for the adjoint consistency in Lemma \ref{lemma:adjoint_consistency}, the primal consistency in Lemma \ref{l4.4}, and two results on product of higher-order variations of the Lagrange multiplier and the displacement (Lemma \ref{lemma:normal_complement} and \ref{l5.9}). Finally, we use these lemmas to prove the error estimates of Theorem \ref{thm:error.estimate}.  

\begin{lemma}[Adjoint Consistency]\label{lemma:adjoint_consistency}
Recalling the definition \eqref{eq:adjoint.consist} of the adjoint consistency error $\bos{\cW}_h$, it holds that for all $ \bos{\chi}\in\bos{\Sigma} \cap H^s(\cM)^{3\times 3}$ for some $s\in(\frac{1}{2},\frac{3}{2})$,
\begin{align*}
\bos{\cW_h(\bos{\chi})}\lesssim h^{s}\left(\lvert  \bos{\chi}\rvert_{H^s(\cM)}+\lvert \gamma_\bo{n}^+\bos{\chi} \rvert_{H^{s-\frac{1}{2}}(\Gamma)}\right).
\end{align*}
\end{lemma}

\begin{proof}
Let $\uvec{v}_h\in\ddgs$, $\bc\in\bos{\Sigma}$ and $\bos{\chi}\in H^s(\cM)^{3\times 3}$, where $s\in(\frac{1}{2},\frac{3}{2})$. Since $\bos{\chi}$ is symmetric, we have $\bc\colon \epsilon_h(\uvec{v}_h)=\bc\colon \dgr{h}(\uvec{v}_h)$ and thus, by \eqref{eq:adjoint.consist},
\begin{align}\label{eq:adjoint.est2}
  w_h(\bos{\chi},\uvec{v}_h)&=-\sum_{K\in\cM}\int_K\bos{\Pi}^1_K\bc\colon \dgr{K}(\uvec{v}_h)+\int_\Gamma \Pi^0_{\cF_\Gamma}(\gamma^+_\bo{n}\bc)\cdot\sjump{\uvec{v}_h}_h-\sum_{K\in\cM}\int_K \bo{v}_K\cdot \DIV\bc\notag\\
  \overset{\eqref{eq:G1K},\eqref{eq:discrete.jump}}&=\sum_{K\in\cM}\int_K \bo{v}_K\cdot \DIV(\bos{\Pi}^1_K\bc)-\sum_{K\in\cM}\sum_{\sigma\in\cF_K}\int_{\sigma}\up{\Ksig}\uvec{v}_h\cdot\bos{\Pi}^1_K\bc\bo{n}_{\Ksig}\notag\\
    &\qquad+\sum_{\sigma\in{\cF_\Gamma}}\int_{\sigma=K|L} \Pi^0_{\sigma}(\bc\bo{n}_{\Ksig})\cdot\big(\up{\Ksig}{\uvec{v}_h}-\up{\Lsig}{\uvec{v}_h}\big)-\sum_{K\in\cM}\int_K \bo{v}_K\cdot \DIV\bc\notag\\
\overset{\text{IBP}}&=-\sum_{K\in\cM}\sum_{\sigma\in\cF_K}\int_\sigma\bo{v}_K\cdot(\bos{\chi}|_K-\bos{\Pi}^1_K\bos{\chi})\bo{n}_{\Ksig}-\sum_{K\in\cM}\sum_{\sigma\in\cF_K}\int_{\sigma}\up{\Ksig}\uvec{v}_h\cdot\bos{\Pi}^1_K\bc\bo{n}_{\Ksig}\notag\\
    &\qquad+\sum_{K\in\cM}\sum_{\sigma\in{\cF_K\cap\cF_\Gamma}}\int_{\sigma}\up{\Ksig}{\uvec{v}_h}\cdot {\Pi}^0_{\sigma}(\bc\bo{n}_\Ksig),
\end{align}
where the introduction of $\bos{\Pi}^1_K$ in the first line is justified by $\dgr{K}(\uvec{v}_h)\in\P^1(K)^{3\times 3}$ and we recall that, for all $\sigma\in \cF_\Gamma$, we have denoted by $K$ and $L$ the cells that contain $\sigma$ and are, respectively, on the positive and negative side of $\Gamma$. 
Since $\up{\Ksig}\uvec{v}_h=\up{\Lsig}\uvec{v}_h$ for all $\sigma=K|L\in \cF^\inte\setminus\cF_\Gamma$, $\up{\Ksig}\uvec{v}_h=0$ for all $\sigma\in\cF_K\cap\cF^\ext$, and the normal traces of $\bc$ are continuous across $\Gamma$, we have
\begin{equation}\label{eq:adjoint.est3}
\begin{aligned}\sum_{K\in\cM}\sum_{\sigma\in\cF_K}\int_{\sigma}\up{\Ksig}{\uvec{v}_h}\cdot \bc\bo{n}_{\Ksig}={}&\sum_{K\in\cM}\sum_{\sigma\in\cF_K\cap\cF_\Gamma}\int_{\sigma}\up{\Ksig}{\uvec{v}_h}\cdot \bc \bo{n}_{\Ksig}\\
&+\underbrace{\sum_{K\in\cM}\sum_{\sigma\in\cF_K\setminus\cF_\Gamma}\int_{\sigma}\up{\Ksig}{\uvec{v}_h}\cdot \bc \bo{n}_{\Ksig}}_{=0}.
\end{aligned}
\end{equation}
Introducing $\pm \sum_{K\in\cM}\sum_{\sigma\in\cF_K}\int_{\sigma}\up{\Ksig}{\uvec{v}_h}\cdot \bc\bo{n}_{\Ksig}$ in \eqref{eq:adjoint.est2} and invoking \eqref{eq:adjoint.est3}, we obtain
\begin{align}\label{eq:adjoint.est4}
w_h(\bos{\chi},\uvec{v}_h)={}&-\sum_{K\in\cM}\sum_{\sigma\in\cF_K}\int_\sigma \bo{v}_K\cdot \big(\bc|_K-\bos{\Pi}^1_K\bc\big)\bo{n}_{\Ksig}+\sum_{K\in\cM}\sum_{\sigma\in\cF_K}\int_{\sigma}\up{\Ksig}{\uvec{v}_h}\cdot \big(\bc|_K-\bos{\Pi}^1_K\bc\big)\bo{n}_{\Ksig}\notag\\
&+\sum_{K\in\cM}\sum_{\sigma\in{\cF_K\cap\cF_\Gamma}}\int_{\sigma}\up{\Ksig}{\uvec{v}_h}\cdot\big({\Pi}^0_{\sigma}(\bc\bo{n}_{\Ksig})-\bc\bo{n}_{\Ksig}\big)\notag\\
=& \underbrace{\sum_{K\in\cM}\sum_{\sigma\in\cF_K}\int_\sigma \big(\up{\Ksig}\uvec{v}_h-\bo{v}_K\big)\cdot \big(\bc|_K-\bos{\Pi}^1_K\bc\big)\bo{n}_{\Ksig}}_{I_1}\notag\\
&+\underbrace{\sum_{K\in\cM}\sum_{\sigma\in{\cF_K\cap\cF_\Gamma}}\int_{\sigma}\big(\up{\Ksig}{\uvec{v}_h}-\Pi^0_\sigma\up{\Ksig}\uvec{v}_h\big)\cdot\big( {\Pi}^0_{\sigma}(\bc\bo{n}_{\Ksig})-\bc\bo{n}_{\Ksig}\big)}_{I_2}.
\end{align}
Note that the introduction of $\Pi^0_\sigma\up{\Ksig}\uvec{v}_h$ in $I_2$ is justified by the definition of the orthogonal projector, which gives $\int_{\sigma}\Pi^0_\sigma\up{\Ksig}\uvec{v}_h\cdot\big( {\Pi}^0_{\sigma}(\bc\bo{n}_{\Ksig})-\bc\bo{n}_{\Ksig}\big)=0$. Using the Cauchy-Schwarz inequality, the fractional multiplicative trace inequality \cite[Remark $2.6$]{erndongbiharmonic},  and the approximation properties of $\bos{\Pi}^1_K$ and ${\Pi}^0_\sigma$, obtained by interpolating the estimates of \cite[Theorem 1.45]{hho-book} for integral Sobolev exponents and the mesh regularity assumption, we obtain
\begin{align*}
I_1&\lesssim  \sum_{K\in\cM}\bigg(\sum_{\sigma\in\cF_K}\lVert \bo{v}_K-\up{\Ksig}\uvec{v}_h\rVert_{L^2(\sigma)}\\
&\qquad\quad\qquad\times \big(h^{-\frac{1}{2}}_K\lVert \bc-\bos{\Pi}^1_K\bc\rVert_{L^2(K)}
+h^{\min(s-\frac{1}{2},\frac{1}{2})}_K\lvert \bc-\bos{\Pi}^1_K\bc\rvert_{H^{\min(s,1)}(K)}\big)\bigg)\\
&\lesssim\sum_{K\in\cM}\sum_{\sigma\in\cF_K}h^{-\frac{1}{2}}_K\lVert \bo{v}_K-\up{\Ksig}\uvec{v}_h\rVert_{L^2(\sigma)}h^s_K \lvert\bc\rvert_{H^s(K)}\\ 
&\lesssim\bigg(\sum_{K\in\cM}h^{2s}_K \lvert\bc\rvert^2_{H^s(K)}\bigg)^{1/2}\bigg(\sum_{K\in\cM}\bigg(\sum_{\sigma\in\cF_K}h^{-\frac{1}{2}}_K\lVert \bo{v}_K-\up{\Ksig}\uvec{v}_h\rVert_{L^2(\sigma)}\bigg)^2\bigg)^{1/2}\\
\overset{\eqref{eq:vk.bound}}&\lesssim h^s\lvert\bc\rvert_{H^s(\cM)}\bigg(\sum_{K\in\cM}\lVert\uvec{v}_h\rVert^2_{1,K}\bigg)^{1/2},
\end{align*}
and
\begin{align*}
I_2&\lesssim \sum_{K\in\cM}\sum_{\sigma\in\cF_K\cap\cF_\Gamma}h_{\sigma}\lVert \nabla_\sigma\up{\Ksig}\uvec{v}_h\rVert_{L^2(\sigma)}h^{s-\frac{1}{2}}_\sigma\lvert\gamma_\bo{n}^+\bc\vert_{H^{s-\frac{1}{2}}(\sigma)}\\
&\lesssim \bigg(\sum_{K\in\cM}\sum_{\sigma\in\cF_K\cap\cF_\Gamma}h_{\sigma}\lVert \nabla_\sigma\up{\Ksig}\uvec{v}_h\rVert_{L^2(\sigma)}^2\bigg)^{1/2}\bigg(\sum_{K\in\cM}\sum_{\sigma\in\cF_K\cap\cF_\Gamma}h^{2s}_\sigma\lvert\gamma_\bo{n}^+\bc\vert_{H^{s-\frac{1}{2}}(\sigma)}^2\bigg)^{1/2}\\
\overset{\eqref{eq:gradPot.bound}}&\lesssim h^s\lvert\gamma_\bo{n}^+\bc\vert_{H^{s-\frac{1}{2}}(\Gamma)}\bigg(\sum_{K\in\cM}\lVert\uvec{v}_h\rVert^2_{1,K}\bigg)^{1/2}.
\end{align*}
Substituting bounds of $I_1$ and $I_2$ in \eqref{eq:adjoint.est4} and using \eqref{eq:adjoint.consist} concludes the lemma.
\end{proof}

\begin{lemma}\label{l4.4} Let $\bo{u}\in H^1_0(\O\setminus\overline{\Gamma})^3\cap H^{\frac{3}{2}+\fr}(\cM)^3$ with $0<\fr\leq 1$. Then, the following estimate holds:
\begin{align}
\bos{\fC_h(\bo{u},\mathbb I_h\bo{u})}&\lesssim h^{\frac{1}{2}+\fr}\lvert\bo{u}\rvert_{H^{\frac{3}{2}+\fr}(\cM)}.\label{eq:primal.est1}
\end{align}
\end{lemma}
\begin{proof} We first note that the regularity assumption ensures that $\bo{u}\in C^0_0(\overline{\O}\setminus\Gamma)^3$. Let $K\in\cM$ and $\bo{q}$ be the $L^2(K)$-orthogonal projection of $\bo{u}$ on $\P^1(K)^3$. By the approximation properties of the polynomial projector \cite[Theorem 1.45]{hho-book} we have
\begin{equation}\label{eq:primal.est2}
\lvert \bo{u}-\bo{q}\rvert_{H^s(K)}\lesssim h^{\frac{3}{2}+\fr-s}\lvert \bo{u}\rvert_{H^{\frac{3}{2}+\fr}(K)},\quad\forall s\in\left[0,\frac{3}{2}+\fr\right].
\end{equation}
Applying the bound \cite[Eq.(5.110)]{hho-book} to $\bo{u}-\bo{q}$ yields
\begin{equation}\label{eq:est.u.q}
\max_{\overline{K}}\lvert \bo{u-q}\rvert\lesssim h^{\fr}_K\lvert \bo{u}\rvert_{H^{\frac{3}{2}+\fr}(K)}.
\end{equation}
Now, let us gather the degrees of freedom on $K$ of $\mI{u}-\mI{q}$.
\begin{equation}\label{eq:primal.est3}
\begin{aligned}
|(\mI{u}-\mI{q})_{\cK s}| &= |(\bo{u-q})(\bx_s)| \leq \max_{\overline{K}} |\bo{u-q}|,  \quad\forall s \in \cV_K,\\
|(\mI{u}-\mI{q})_{\cK e}| &= |\Pi^0_e((\bo{u-q})|_K)| \leq \max_{\overline{K}} |\bo{u-q}|, \quad \forall e \in \cE_K,\\
|(\mI{u}-\mI{q})_{\cK \sigma}| &= |\Pi^0_\sigma((\bo{u-q})|_K)| \leq \max_{\overline{K}} |\bo{u-q}|, \quad\forall \sigma \in \cF_K,\\
|(\mI{u}-\mI{q})_{K}| &= |\Pi^0_K(\bo{u-q}) + (\widetilde{\bo{u}}_K-\widetilde{\bo{q}}_K)| 
\leq \max_{\overline{K}} |\bo{u-q}| + |\widetilde{\bo{u}}_K|,
\end{aligned}
\end{equation}
where, in the last relation, we have used $\widetilde{\bo{q}}_K=0$ in virtue of the argument in Remark \ref{rem:def.interpolator}.
Let $f\in\P^1(K)$ be such that $\int_K f=0$ and $\widetilde{\bo{u}}_K=\nabla f.$ Since $\Pi^0_Kf=0$, by \eqref{eq:tildevK},
\begin{align*}
\norm{L^2(K)}{\widetilde{\bo{u}}_K}^2\le{}&\sum_{\sigma\in\cF_K}\norm{L^2(\sigma)}{\bo{u}|_K-\up{\Ksig}\mathbb I_h\bo{u}}\norm{L^2(\sigma)}{f}\\
\le{}&\sum_{\sigma\in\cF_K}\norm{L^2(\sigma)}{\bo{u}|_K-\up{\Ksig}\mathbb I_h\bo{u}}\norm{L^2(\sigma)}{f-\Pi^0_K f}\\
\le{}&\sum_{\sigma\in\cF_K}h_\sigma\norm{L^\infty(\sigma)}{\bo{u}|_K-\up{\Ksig}\mathbb I_h\bo{u}}h_K^{1/2}\norm{L^2(K)}{\nabla f}
\end{align*}
where we have used the trace approximation properties of $\Pi^0_K$. Simplify by $\norm{L^2(K)}{\nabla f}=\norm{L^2(K)}{\widetilde{\bo{u}}_K}$ and use \eqref{eq:Upsilon.sigma.interpolate} to introduce $\bo{q}-\up{\Ksig}\mathbb I_h\bo{q}=0$:
\begin{align*}
\norm{L^2(K)}{\widetilde{\bo{u}}_K}\le{}&h_K^{3/2}\sum_{\sigma\in\cF_K}\norm{L^\infty(\sigma)}{(\bo{u}-\bo{q})|_K-\up{\Ksig}\mI{(\bo{u}-\bo{q})}}.
\end{align*}
Then, we invoke \cite[(A.10)]{Di-Pietro.Droniou.ea:24} to get $\norm{L^\infty(\sigma)}{\up{\Ksig}\mI{(\bo{u}-\bo{q})}}\lesssim \max_{\overline{\sigma}}\big|\bo{u}|_K-\bo{q}\big|\le \max_{\overline{K}}|\bo{u}-\bo{q}|$,
which allows us to conclude
\begin{equation}\label{eq:primal.est4}
\lvert\widetilde{\bo{u}}_K\rvert\lesssim \max\limits_{\overline{K}}\lvert\bo{u-q}\rvert.
\end{equation}
Using Lemma \ref{l3.4} and \eqref{eq:primal.est3}–\eqref{eq:primal.est4}, it follows that
\begin{align*}
\left(\lVert \dgr{K}\mI{(u-q)}\rVert^2_{L^2(K)}+\mathbb S_K(\mI{u}-\mI{q},\mI{u}-\mI{q})\right)^{1/2}&=\lVert \mI{u}-\mI{q}\rVert_{1,K}\\
&\lesssim h^{1/2}_K \max\limits_{\overline{K}}\lvert \bo{u-q}\rvert\overset{\eqref{eq:est.u.q}}\lesssim h^{\frac{1}{2}+\fr}_K\lvert \bo{u}\rvert_{H^{\frac{3}{2}+\fr}(K)}.
\end{align*}
Using \eqref{eq:Grad.interpolate} and \eqref{eq:stab.consistency} in the above inequality, we get
\begin{align*}
\left(\lVert \dgr{K}\mI{u}-\nabla \bo{q}\rVert^2_{L^2(K)}+\mathbb S_K(\mI{u},\mI{u})\right)^{1/2}\lesssim h^{\frac{1}{2}+\fr}_K\lvert \bo{u}\rvert_{H^{\frac{3}{2}+\fr}(K)}.
\end{align*}
Recalling the definition \eqref{eq:primal.consist} of $\bos{\fC}_h$, the estimate \eqref{eq:primal.est1} follows by introducing $\pm\nabla\bo{u}$ in the left-hand side above, using a triangle inequality, the approximation property \eqref{eq:primal.est2} (with $s=1$), squaring, summing over $K$, using the bound $h_K\leq h$ and taking the square root. 
\end{proof}

\begin{lemma}\label{lemma:normal_complement}
Let $(\bo{u},\blambda)$ be the solution of \eqref{eq:model}, and assume that $\bo{u}\in H^1_0(\O\setminus\overline{\Gamma})^3\cap H^{\frac{3}{2}+\fr}(\cM)^3$ with $0<\fr\leq 1$ and $\sigma\in\cF_\Gamma$.  
Then, the following estimate holds:
\begin{equation}\label{eq:normal_compplement1.compact}
\int_{\sigma}(\lambda_{\bo{n}}-\Pi^0_{\sigma}\lambda_{\bo{n}})\big( \sjump{\bo{u}}_{\bo{n}}-\Pi^0_{\sigma}\sjump{\bo{u}}_{\bo{n}}\big)\lesssim h_\sigma^{1+2\fr}\big(\lvert \lambda_{\bo{n}}\rvert^2_{H^{\fr}(\sigma)}+\lvert\nabla \sjump{\bo{u}}_{\bo{n}}\rvert^2_{H^{\fr}(\sigma)}\big).
\end{equation}
\end{lemma}

\begin{proof}
Let $\Gamma_{C\sigma}\coloneqq\{x\in\sigma\colon\sjump{\bo{u}(x)}_{\bo{n}}=\bo{0}\}$ and $\Gamma_{N\sigma}\coloneqq\{x\in\Gamma\colon\sjump{\bo{u}(x)}_{\bo{n}}<\bo{0}\}$ be the contact and non-contact zones on $\sigma$.
If $|\Gamma_{N\sigma}|>0$ and $|\Gamma_{C\sigma}|>0$ then the arguments leading to \cite[Eq.~(33) and (34)]{MR3357636} show that \eqref{eq:normal_compplement1.compact} hold with $h_\sigma^{1+2\fr}$ replaced by either $\frac{h_\sigma^{2(1+\fr)}}{|\Gamma_{N\sigma}|^{1/2}}$ or $\frac{h_\sigma^{2(1+\fr)}}{|\Gamma_{C\sigma}|^{1/2}}$.
Since $\sigma=\Gamma_{N\sigma}\sqcup \Gamma_{C\sigma}$, one of $|\Gamma_{N\sigma}|$ or $|\Gamma_{C\sigma}|$ is larger than $|\sigma|/2\gtrsim h_\sigma^2$ (by mesh regularity assumption); picking the estimate corresponding to that largest set yields \eqref{eq:normal_compplement1.compact}.

If one of $|\Gamma_{N\sigma}|$ and $|\Gamma_{C\sigma}|$ vanishes, say $|\Gamma_{C\sigma}|=0$ to fix the ideas, then by complementarity condition we have $\lambda_{\bo{n}}=0$ on $\sigma$, so $\Pi^0_{\sigma}\lambda_{\bo{n}}=0$ and thus the left-hand side of \eqref{eq:normal_compplement1.compact} vanishes, making this inequality trivial.
\end{proof}

\begin{lemma}\label{l5.9}
Let $(\bo{u},\blambda)$ be the solution of \eqref{eq:model}, and consider real parameters $\fr$, $s_{\fr}$ and $t_\fr$ as in Theorem \ref{thm:error.estimate}. Then, 
\begin{equation}\label{eq:bound.u.lambda}
\int_{\Gamma}(\blambda-\Pi^0_{{\cF}_\Gamma}\blambda)\cdot\big( \sjump{\bo{u}}-\Pi^0_{\cF_\Gamma}\sjump{\bo{u}}\big)\lesssim  h^{1+2\fr}\big(\lvert \blambda\rvert^2_{H^{s_\fr}(\cF_\Gamma)}+\lvert \sjump{\bo{u}}\rvert^2_{H^{t_\fr}(\cF_\Gamma)}\big).
\end{equation}
\end{lemma}

\begin{proof}
In case of frictionless contact ($g=0$), we have
\begin{align*}
\int_{\sigma}(\blambda-\Pi^0_{{\sigma}}\blambda)\cdot\big( \sjump{\bo{u}}-\Pi^0_{\sigma}\sjump{\bo{u}}\big)=\int_{\sigma}(\lambda_{\bo{n}}-\Pi^0_{\sigma}\lambda_{\bo{n}})\big( \sjump{\bo{u}}_{\bo{n}}-\Pi^0_{\sigma}\sjump{\bo{u}}_{\bo{n}}\big)\quad\forall \sigma\in\cF_\Gamma.
\end{align*}
By Lemma \ref{lemma:normal_complement} and since $t_\fr=1+\fr$ (as $g=0$), we infer
\begin{equation*}
\int_{\sigma}(\blambda-\Pi^0_{\sigma}\blambda)\cdot\big( \sjump{\bo{u}}-\Pi^0_{\sigma}\sjump{\bo{u}}\big)\lesssim h^{1+2\fr}\big(\lvert \blambda\rvert^2_{H^{\fr}(\sigma)}+\lvert \sjump{\bo{u}}\rvert^2_{H^{t_{\fr}}(\sigma)}\big)\quad\forall \sigma\in\cF_\Gamma.
\end{equation*}
Adding over all the faces of $\cF_\Gamma$ yields \eqref{eq:bound.u.lambda}.

In the case with friction, i.e., $g\not=0$, the estimate follows directly from the Cauchy-Schwarz inequality, approximation properties of the projection operator $\Pi^0_{\cF_\Gamma}$ \cite[Theorem 1.45]{hho-book}, and Young's inequality~(note that this simple approach is only possible because of the restriction $0\le\fr\le 1/2$, instead of the larger range $0\le \fr<1$ allowed by the finer estimate obtained above when $g=0$).
\end{proof}

\begin{proof}[\textit{\textbf {Proof of Theorem~\ref{thm:error.estimate}:}}]
From Theorem \ref{thm:abstract.error.estimate}, Lemma \ref{lemma:adjoint_consistency}, Lemma \ref{l4.4} and Lemma \ref{l5.9}, we obtain
\begin{align*}
\lVert\uvec{u}_h-\mI{u}\rVert_{\mathrm{en},h}\lesssim h^{\frac{1}{2}+\fr}\bigg[{}&\frac{1}{\sqrt{\mu_{1}}}\bigg(\lvert\bos{\sigma}(\bo{u})\rvert_{H^{\frac{1}{2}+\fr}(\cM)}+\lvert \blambda\rvert_{H^{\fr}(\Gamma)}\bigg)+\sqrt{\mu_{2}}\lvert\bo{u}\rvert_{H^{\frac{3}{2}+\fr}(\cM)}\\
&+\lvert \blambda\rvert_{H^{s_\fr}(\cF_\Gamma)}+\lvert \sjump{\bo{u}}\rvert_{H^{t_{\fr}}(\cF_\Gamma)}\bigg].
\end{align*}
This completes the bound of \eqref{eq:total.error}. Finally, the  bound \eqref{eq:total.error_lm} is obtained from \eqref{eq:abstract.error.estimate_lam}, Lemmas~\ref{lma_Pi_lambda_est}-\ref{l4.4}, and the approximation property \cite[Theorem 1.45]{hho-book} of $\Pi^0_{\cF_\Gamma}$.
This concludes the theorem.
\end{proof}
\section{Numerical Experiments}\label{sec numerical}
In this section, we assess the numerical convergence of the proposed discretisation for the poromechanical model with frictional contact at matrix-fracture interfaces. The first and second Lam\'e coefficients are related to  Young's modulus $E$ and Poisson's ratio $\nu$ by
$G:= \frac{E}{2(1+\nu)}, \quad L \coloneqq \frac{\nu E}{(1+\nu)(1-2\nu)},$ respectively.

The subsequent sections present a series of test cases. Sections 6.1–6.3 employ the same families of meshes having uniform Cartesian, tetrahedral, and hexahedral cells. Section~\ref{subsec.6.1} analyzes the convergence of the frictionless formulation using a 3D manufactured analytical solution with a single fracture. Section~\ref{subsec.6.2} focuses on a 3D configuration governed by the Tresca friction law, while Section~\ref{subsec.6.3} examines the performance of the method in nearly incompressible regimes for Lam\'e coefficient $L\in\{1, 10^4, 10^6\}$. Finally, Section~\ref{subsec.6.4} turns to a more challenging 3D Discrete Fracture–Matrix (DFM) model without an analytical solution, relying on tetrahedral meshes refined along the fracture network.

In all the following experiments, the fracture face–wise Lagrange multiplier enables the reformulation of the variational inequality~\eqref{eq:scheme.lagrange}, together with the constraint $\boldsymbol{\lambda}_h \in \mathcal{C}_{f,h}$, into a system of nonlinear equations. These equations involve the projection operators $[r]_{+} = \max(0, r)$ and $[\bos{\xi}]_g$=projection of $\bos{\xi} \in \mathbb{R}^3$ onto the ball centered at $0$ with radius $g$, and are written
\begin{align*}
\begin{cases}
\lambda_{h,\bo{n}}\coloneqq  [\lambda_{h,\bo{n}}+\beta_{h,\bo{n}}\sjump{\uvec{u}_h}_{h,\bo{n}}]_+,\\
\blambda_{h,\tau}\coloneqq [\blambda_{h,\tau}+\beta_{h,\tau}\sjump{\uvec{u}_h}_{h,\tau}]_g,
\end{cases}
\end{align*}
where $\beta_{h,\bo{n}}>0,$ and $\beta_{h,\tau}>0$ are face-wise constant quantities defined along $\Gamma$. The resulting nonlinear system is solved with a semi-smooth Newton algorithm. The stopping criteria is set to $10^{-12}$ on the relative residual, and the linear system at each iteration is solved using the Pardiso LU solver.

\subsection{3D manufactured solution for a frictionless contact~($g=0$) mechanical model}\label{subsec.6.1}

Following \cite{droniou2023bubble}, we consider the 3D domain $\Omega = (-1,1)^3$ with a single non-immersed fracture 
$\Gamma = \{0\} \times (-1,1)^2$. 
The friction coefficient $g$ is set to zero, corresponding to frictionless contact, 
and the Lam\'e coefficients are set to be $G = L = 1$. The exact solution is defined by
\begin{align*}
\bo{u}(x,y,z) =
\begin{cases}
\begin{pmatrix}
q(x,y)p(z)\\[4pt]
p(z)\\[4pt]
x^2p(z)
\end{pmatrix} & \text{if } z \ge 0,\\[12pt]
\begin{pmatrix}
h(x)p^{+}(z)\\[4pt]
h(x)\big(p^{+}(z)\big)'\\[4pt]
-\displaystyle\int_{0}^{x}h(\xi)\,d\xi\,\big(p^{+}(z)\big)'
\end{pmatrix} & \text{if } z<0,~x<0,\\[16pt]
\begin{pmatrix}
h(x)p^{-}(z)\\[4pt]
h(x)\big(p^{-}(z)\big)'\\[4pt]
-\displaystyle\int_{0}^{x}h(\xi)\,d\xi\,\big(p^{-}(z)\big)'
\end{pmatrix} & \text{if } z<0,~x\ge 0.
\end{cases}
\end{align*}
with
$q(x,y) = -\sin\left(\frac{\pi x}{2}\right)\cos\left(\frac{\pi y}{2}\right),~p(z) = z^2, ~h(x) = \cos\left(\frac{\pi x}{2}\right),~p^{+}(z) = z^4,\text{ and }p^{-}(z) = 2z^4.$ The function $\bo{u}$ is designed to satisfy the frictionless contact conditions 
at the matrix-fracture interface $\Gamma$. The right-hand side is $\bo{f} = -\DIV\bos{\sigma}(\bo{u})$ and the Dirichlet boundary conditions on $\partial\Omega$ are deduced from $\bo{u}$. Note that the fracture $\Gamma$ is in contact state for $z>0$ ($\sjump{\bo{u}}_{\bo{n}}=0$) and open for $z<0$, with a normal jump $\sjump{\bo{u}}_{\bo{n}} = -\min(z,0)^4$ depending only on $z$. 

This numerical experiment is performed on families of uniform Cartesian, tetrahedral, and hexahedral meshes. Starting from uniform Cartesian meshes, the hexahedral families -- hereafter called Hexa-cut are generated by random perturbations of the vertices (except those on the fracture) and by cutting non-planar faces into two triangles. We measure the $L^2$-norms of $\bo{u}-\up{h}(\uvec{u}_h)$, $\sjump{\bo{u}}-\sjump{\uvec{u}_h}_h$, $\nabla\bo{u}-\nabla\up{h}({\uvec{u}_h})$,\text{ and }$\lambda_{\bo{n}}-\lambda_{h,\bo{n}}$ versus the cubic root of the number of cells. Corresponding to these errors, the convergence rates of the DDR scheme are illustrated in Figure \ref{fig:6.1}. The results show that the convergence on the Lagrange multiplier is of order 1; this is coherent with the estimate \eqref{eq:total.error_lm} which predicts a rate of almost $3/2$ in discrete $H^{-1/2}$-like norm, and would naturally makes us expect a rate of almost $1$ in $L^2$-norm. For the displacement, \eqref{eq:total.error} predicts a rate slightly below $3/2$ in energy norm, which would translate into the same rate for the gradient; this is observed on tetrahedral and hexa-cut meshes, but Cartesian meshes seem to produce a super-convergence with a rate of almost $2$. The approximation of the displacement and its jump seems to be higher than those for the energy norm, albeit not by a full order: all three meshes indicate an order $2$ convergence on these measures.
Moreover, Figure~\ref{fig:6.3}(a) compares the $L^2$-norm of $\nabla\bo{u}-\nabla\up{h}(\uvec{u}_h)$ between the lowest-order method from \cite{jhr} and the current higher order scheme on tetrahedral meshes. The comparison reveals that, to achieve a given degree of accuracy, the scheme presented here requires significantly fewer degrees of freedom than the method in \cite{jhr}, highlighting the improved efficiency of the higher-order formulation.
\begin{figure}[ht]
\centering
\begin{tikzpicture}

\begin{groupplot}[
  group style={
    group size=2 by 1,
    horizontal sep=2.4cm,
    vertical sep=1.8cm,
  },
  width=0.45\textwidth,
  height=0.48\textwidth,
  xlabel={$N_{\text{cell}}^{1/3}$},
  ylabel={$L^2$ Error},
  ymode=log, xmode=log,
  legend style={
    at={(0.05,0.05)},
    anchor=south west,
    font=\small,
    draw=none, fill=none
  },
  every axis plot/.append style={thick, mark size=2pt}
]

\nextgroupplot[
    xmin=2, xmax=32,
    xtick={2,4,8,14,24,32},
xticklabels={2,4,8,14,24,32},
    ymin=1e-4, ymax=5,
    title={(a)}
]
\addplot[red, mark=*] coordinates {(2,3.025945e-01) (4,7.597039e-02) (8,1.722877e-02) (16,4.006409e-03) (32,9.814365e-04)};
\addlegendentry{$\mathbf{u}$}
\addplot[blue, mark=square*] coordinates {(2,9.283358e-01) (4,3.483605e-01) (8,9.565768e-02) (16,2.256112e-02) (32,5.335594e-03)};
\addlegendentry{$[\![\mathbf{u}]\!]$}
\addplot[green!60!black, mark=asterisk] coordinates {(2,4.062852e-01) (4,1.282431e-01) (8,3.403710e-02) (16,8.456262e-03) (32,2.099141e-03)};
\addlegendentry{$\nabla \mathbf{u}$}
\addplot[orange, mark=x] coordinates {(2,7.638315e-01) (4,4.228743e-01) (8,2.165096e-01) (16,1.088756e-01) (32,5.450681e-02)};
\addlegendentry{$\lambda_{\bo{n}}$}
\draw[thick] 
  (axis cs:22,3e-4) -- 
  (axis cs:11,3e-4) --
  (axis cs:11,1.2e-3) -- cycle;

\draw[densely dotted, thick] 
  (axis cs:22,3e-4) --
  (axis cs:11,3e-4) --
  (axis cs:11,6e-4) -- cycle;

\node[font=\scriptsize] at (axis cs:9,6e-4) {$\mathcal{O}(1)$};
\node[font=\scriptsize] at (axis cs:9,1.2e-3) {$\mathcal{O}(2)$};

\nextgroupplot[
    xmin=5.12992784, xmax=33.4362044, ymin=1e-4, ymax=10,
    title={(b)}
]
\addplot[red, mark=*] coordinates {(5.12992784,8.918814e-02) (7.7969745,3.590380e-02) (13.9897885,9.731805e-03) (20.3914548,3.881200e-03) (26.9500678,2.249882e-03) (33.4362044,1.367476e-03)};
\addlegendentry{$\mathbf{u}$}
\addplot[blue, mark=square*] coordinates {(5.12992784,3.544494e-01) (7.7969745,1.323711e-01) (13.9897885,5.174786e-02) (20.3914548,2.154133e-02) (26.9500678,1.283316e-02) (33.4362044,7.684108e-03)};
\addlegendentry{$[\![\mathbf{u}]\!]$}
\addplot[green!60!black, mark=asterisk] coordinates {(5.12992784,1.452727e-01) (7.7969745,7.015395e-02) (13.9897885,2.399698e-02) (20.3914548,1.105478e-02) (26.9500678,6.371187e-03) (33.4362044,4.038540e-03)};
\addlegendentry{$\nabla \mathbf{u}$}
\addplot[orange, mark=x] coordinates {(5.12992784,4.726740e-01) (7.7969745,2.797944e-01) (13.9897885,1.509177e-01) (20.3914548,1.008614e-01) (26.9500678,7.656991e-02) (33.4362044,6.218081e-02)};
\addlegendentry{$\lambda_{\bo{n}}$}
\draw[thick] 
  (axis cs:30,3e-4) -- 
  (axis cs:15,3e-4) --
  (axis cs:15,1.2e-3) -- cycle;

\draw[densely dotted, thick] 
  (axis cs:30,3e-4) --
  (axis cs:15,3e-4) --
  (axis cs:15,6e-4) -- cycle;

\node[font=\scriptsize] at (axis cs:13,6e-4) {$\mathcal{O}(1)$};
\node[font=\scriptsize] at (axis cs:13,1.2e-3) {$\mathcal{O}(2)$};

\end{groupplot}

\node at ($(group c1r1.south)!0.4!(group c2r1.south) - (0,6.0cm)$) {
  \begin{tikzpicture}
  \begin{axis}[
    width=0.45\textwidth,
    height=0.48\textwidth,
    xlabel={$N_{\text{cell}}^{1/3}$},
    ylabel={$L^2$ Error},
    ymode=log, xmode=log,
    every axis plot/.append style={thick, mark size=2pt},
    xmin=2, xmax=32,
    ymin=1e-4, ymax=5,
    title={(c)},
    legend style={
    at={(0.05,0.05)},
    anchor=south west,
    font=\small,
    draw=none, fill=none
  }
  ]
\addplot[red, mark=*] coordinates {(2,3.014315e-01) (4,8.080911e-02) (8,1.838843e-02) (16,4.279674e-03) (32,1.041480e-03)};
\addlegendentry{$\mathbf{u}$}
\addplot[blue, mark=square*] coordinates {(2,9.371215e-01) (4,3.037581e-01) (8,8.355406e-02) (16,2.163726e-02) (32,5.449991e-03)};
\addlegendentry{$[\![\mathbf{u}]\!]$}
\addplot[green!60!black, mark=asterisk] coordinates {(2,4.064416e-01) (4,1.449232e-01) (8,4.267688e-02) (16,1.156811e-02) (32,3.016451e-03)};
\addlegendentry{$\nabla \mathbf{u}$}
\addplot[orange, mark=x] coordinates {(2,7.618957e-01) (4,4.356406e-01) (8,2.214046e-01) (16,1.112583e-01) (32,5.631664e-02)};
\addlegendentry{$\lambda_{\bo{n}}$}

\draw[thick] 
  (axis cs:22,3e-4) -- 
  (axis cs:11,3e-4) --
  (axis cs:11,1.2e-3) -- cycle;

\draw[densely dotted, thick] 
  (axis cs:22,3e-4) --
  (axis cs:11,3e-4) --
  (axis cs:11,6e-4) -- cycle;

\node[font=\scriptsize] at (axis cs:9,6e-4) {$\mathcal{O}(1)$};
\node[font=\scriptsize] at (axis cs:9,1.2e-3) {$\mathcal{O}(2)$};
  \end{axis}
  \end{tikzpicture}
};

\end{tikzpicture}
\caption{(Test case from Section \ref{subsec.6.1}). Relative $L^2$-norm of errors $\bo{u}-\up{h}(\uvec{u}_h)$, $\sjump{\bo{u}}-\sjump{\uvec{u}_h}_h$, $\nabla\bo{u}-\nabla\up{h}({\uvec{u}_h})$,\text{ and }$\lambda_{\bo{n}}-\lambda_{h,\bo{n}}$ versus the cubic root of the number of cells, for (a) Cartesian, (b) tetrahedral, and (c) Hexa-cut mesh families.}
\label{fig:6.1}
\end{figure}
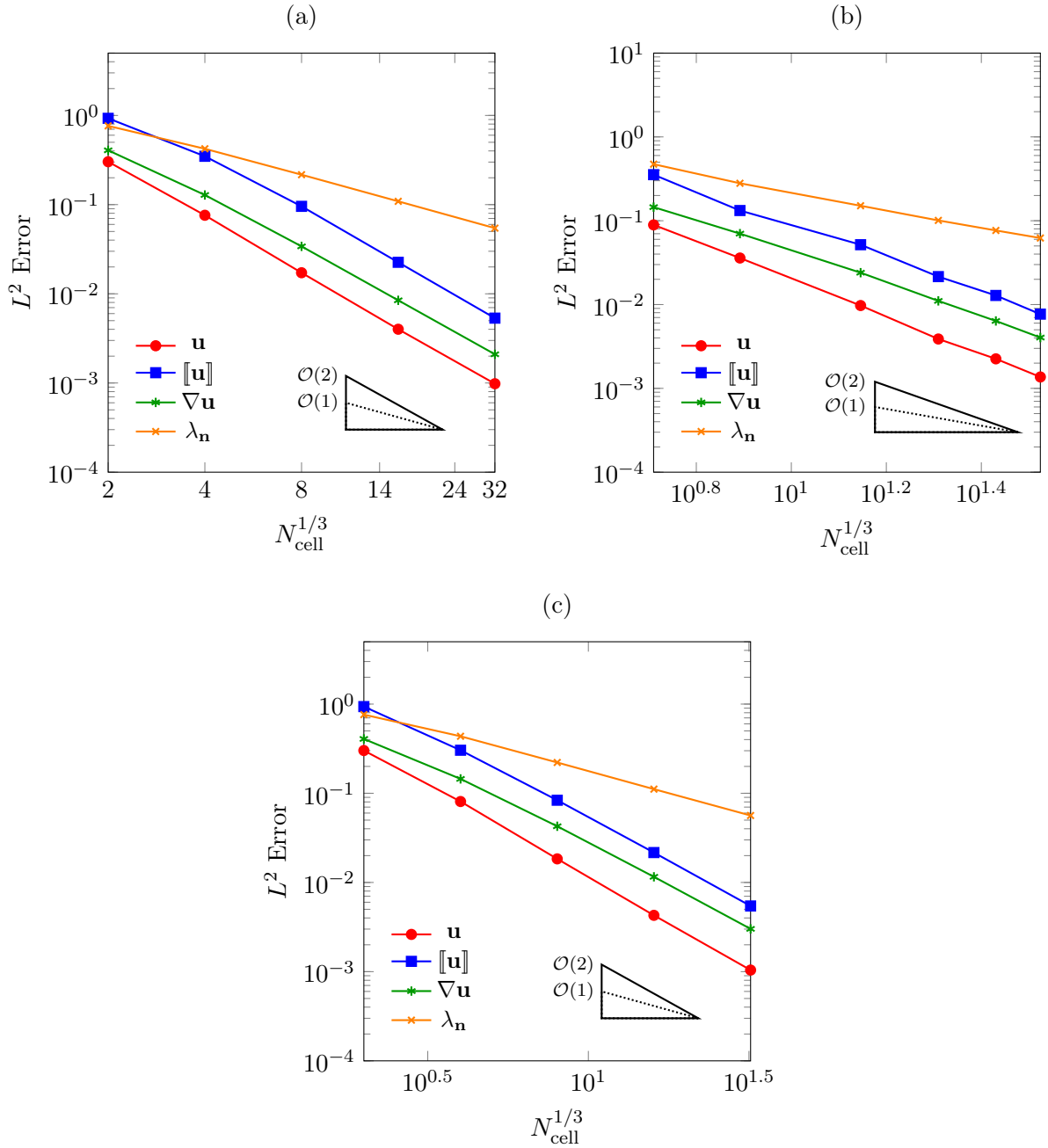

\subsection{3D manufactured solution for the Tresca friction model}\label{subsec.6.2}

We consider the same domain, fracture and Lamé coefficients as in Section \ref{subsec.6.1}, but the Tresca threshold is now set at $g=1$. Following \cite{jhr}, the exact solution is given by
\begin{align*}
\bo{u}(x,y,z) =
\begin{cases}
\begin{pmatrix}
h(x,y)P(z) - g y \\
P(z) \\
x^2 P(z)
\end{pmatrix} & \text{if } z \ge 0, \\[1.2em]
\begin{pmatrix}
h(x,y)Q(z) - g y \\
2Q(z) \\
x^2 Q(z)
\end{pmatrix} & \text{if } z < 0,~ x < 0, \\[1.2em]
\begin{pmatrix}
h(x,y)Q(z) - g y \\
Q(z) \\
x^2 Q(z)
\end{pmatrix} & \text{if } z < 0,~ x \ge 0,
\end{cases}
\end{align*}
with $h(x,y) = -\sin(x)\cos(y)$, $P(z) = z^2$, and $Q(z) = z^2/4$, 
and satisfies the Tresca frictional-contact conditions at the matrix–fracture interface $\Gamma$.The right-hand side $\mathbf{f} = -\operatorname{div}\boldsymbol{\sigma}(\mathbf{u})$ and the Dirichlet boundary conditions on $\partial\Omega$ are obtained from $\bo{u}$. The fracture $\Gamma$ is in slippy-contact for $z<0$ and sticky-contact for $z>0$. The convergence of the scheme is investigated on the same families of meshes as before, with the results given in Figure \ref{fig:6.2}. According to theoretical predictions, the displacement gradient $\nabla\bo{u}$ and the Lagrange multiplier $\lambda_{\bo{n}}$ should converge with order $1$ (see Theorem \ref{thm:error.estimate} and remember that $g\not= 0$ here). The numerical results show that, on all three mesh families, the approximation of the gradient is much better, reaching orders close to 2. 
The Lagrange multiplier variable converges with order 1, which is slightly better than the expected rate of 0.5 (\eqref{eq:total.error_lm} predicts a rate of 1 in discrete $H^{-1/2}$-norm, which would suggest a rate of 0.5 in $L^2$-norm). Given that the improved convergence occurs for a range of unstructured meshes, we believe that the reason for the discrepancy between the theoretical results and the observed rates of convergence is to be found in limitations of the analysis when $g\not=0$, and is not due to a specific behaviour of the scheme on these particular test cases.
  We note that, even for quadratic finite elements, we are not aware of an analysis for mixed formulations that proves an order $>1$ convergence in energy norm for the Tresca model with friction (the study in \cite{wohlmuthquadraticcontact} is carried out with $g=0$). On the other hand, such a result is proved for the Finite Element and Hybrid High Order discretisations based on Nitsche's formulation \cite{Chouly2014,Chouly2020}.

  For the displacement itself and its jump, we notice a rate of about 2 -- which is somewhat expected since Theorem \ref{thm:error.estimate} predicts an order 1 for the gradient -- and even larger than 2 on Cartesian meshes (this superconvergence is probably due to the symmetries of this particular mesh). Figure~\ref{fig:6.3}(b) demonstrates a similar efficiency gain for the higher order method over the lowest order case, consistent with the results from the test case in Section~\ref{subsec.6.1}.

\begin{figure}[ht]
\centering
\begin{tikzpicture}

\begin{groupplot}[
  group style={
    group size=2 by 1,
    horizontal sep=2.4cm,
    vertical sep=1.8cm,
  },
  width=0.45\textwidth,
  height=0.48\textwidth,
  xlabel={$N_{\text{cell}}^{1/3}$},
  ylabel={$L^2$ Error},
  ymode=log, xmode=log,
  legend style={
    at={(0.05,0.05)},
    anchor=south west,
    font=\small,
    draw=none, fill=none
  },
  every axis plot/.append style={thick, mark size=2pt}
]

\nextgroupplot[
    xmin=2, xmax=32,
    xtick={2,4,8,14,24,32},
xticklabels={2,4,8,14,24,32},
    ymin=1e-5, ymax=5,
    title={(a)}
]
\addplot[red, mark=*] coordinates {(2,8.705171e-02) (4,2.308235e-02) (8,4.884993e-03) (16,1.153743e-03) (32,2.864649e-04)};
\addlegendentry{$\mathbf{u}$}
\addplot[blue, mark=square*] coordinates {(2,1.995022e-01) (4,6.909518e-02) (8,1.360289e-02) (16,1.990550e-03) (32,2.626671e-04)};
\addlegendentry{$[\![\mathbf{u}]\!]$}
\addplot[green!60!black, mark=asterisk] coordinates {(2,1.870013e-01) (4,5.610256e-02) (8,1.313300e-02) (16,3.176445e-03) (32,7.874169e-04)};
\addlegendentry{$\nabla \mathbf{u}$}
\addplot[orange, mark=x] coordinates {(2,7.113770e-01) (4,3.739001e-01) (8,1.903817e-01) (16,9.561773e-02) (32,4.786010e-02)};
\addlegendentry{$\lambda_{\bo{n}}$}
\draw[thick] 
  (axis cs:22,3e-5) -- 
  (axis cs:11,3e-5) --
  (axis cs:11,1.2e-4) -- cycle;

\draw[densely dotted, thick] 
  (axis cs:22,3e-5) --
  (axis cs:11,3e-5) --
  (axis cs:11,6e-5) -- cycle;

\node[font=\scriptsize] at (axis cs:9,6e-5) {$\mathcal{O}(1)$};
\node[font=\scriptsize] at (axis cs:9,1.2e-4) {$\mathcal{O}(2)$};

\nextgroupplot[
    xmin=5.12992784, xmax=33.4362044, ymin=1e-5, ymax=10,
    title={(b)}
]
\addplot[red, mark=*] coordinates {(5.12992784,2.974890e-02) (7.7969745,1.716010e-02) (10.7606709,7.701544e-03) (17.2466912,2.632077e-03) (23.4134435, 1.378392e-03 ) (33.4362044,5.402174e-04)};
\addlegendentry{$\mathbf{u}$}
\addplot[blue, mark=square*] coordinates {(5.12992784,2.830399e-02) (7.7969745,1.536020e-02) (10.7606709,4.563126e-03) (17.2466912,2.208606e-03) (23.4134435, 8.232142e-04) (33.4362044,4.062463e-04)};
\addlegendentry{$[\![\mathbf{u}]\!]$}
\addplot[green!60!black, mark=asterisk] coordinates {(5.12992784,7.495378e-02) (7.7969745,4.458098e-02) (10.7606709,2.343235e-02) (17.2466912,8.633359e-03) (23.4134435, 5.183450e-03) (33.4362044,2.296086e-03)};
\addlegendentry{$\nabla \mathbf{u}$}
\addplot[orange, mark=x] coordinates {(5.12992784,4.076211e-01) (7.7969745,2.435796e-01) (10.7606709,1.709078e-01) (17.2466912,1.053065e-01) (23.4134435, 7.696846e-02) (33.4362044,5.489760e-02)};
\addlegendentry{$\lambda_{\bo{n}}$}
\draw[thick] 
  (axis cs:30,3e-5) -- 
  (axis cs:15,3e-5) --
  (axis cs:15,1.2e-4) -- cycle;

\draw[densely dotted, thick] 
  (axis cs:30,3e-5) --
  (axis cs:15,3e-5) --
  (axis cs:15,6e-5) -- cycle;

\node[font=\scriptsize] at (axis cs:13,6e-5) {$\mathcal{O}(1)$};
\node[font=\scriptsize] at (axis cs:13,1.2e-4) {$\mathcal{O}(2)$};

\end{groupplot}

\node at ($(group c1r1.south)!0.4!(group c2r1.south) - (0,6.0cm)$) {
  \begin{tikzpicture}
  \begin{axis}[
    width=0.45\textwidth,
    height=0.48\textwidth,
    xlabel={$N_{\text{cell}}^{1/3}$},
    ylabel={$L^2$ Error},
    ymode=log, xmode=log,
    every axis plot/.append style={thick, mark size=2pt},
    xmin=2, xmax=32,
    ymin=1e-5, ymax=5,
    title={(c)},
    legend style={
    at={(0.05,0.05)},
    anchor=south west,
    font=\small,
    draw=none, fill=none
  }
  ]
\addplot[red, mark=*] coordinates {(2,8.681604e-02) (4,2.277699e-02) (8,5.318409e-03) (16,1.241152e-03) (32,3.049070e-04)};
\addlegendentry{$\mathbf{u}$}
\addplot[blue, mark=square*] coordinates {(2,2.016279e-01) (4,7.496247e-02) (8,1.673511e-02) (16,3.441953e-03) (32,9.381447e-04)};
\addlegendentry{$[\![\mathbf{u}]\!]$}
\addplot[green!60!black, mark=asterisk] coordinates {(2,1.878127e-01) (4,6.480960e-02) (8,1.991956e-02) (16,5.485359e-03) (32,1.635492e-03)};
\addlegendentry{$\nabla \mathbf{u}$}
\addplot[orange, mark=x] coordinates {(2,7.096245e-01) (4,3.914234e-01) (8,1.937437e-01) (16,9.752187e-02) (32,4.950864e-02)};
\addlegendentry{$\lambda_{\bo{n}}$}

\draw[thick] 
  (axis cs:22,3e-5) -- 
  (axis cs:11,3e-5) --
  (axis cs:11,1.2e-4) -- cycle;

\draw[densely dotted, thick] 
  (axis cs:22,3e-5) --
  (axis cs:11,3e-5) --
  (axis cs:11,6e-5) -- cycle;

\node[font=\scriptsize] at (axis cs:9,6e-5) {$\mathcal{O}(1)$};
\node[font=\scriptsize] at (axis cs:9,1.2e-4) {$\mathcal{O}(2)$};
  \end{axis}
  \end{tikzpicture}
};

\end{tikzpicture}
\caption{(Test case from Section \ref{subsec.6.2}). Relative $L^2$-norm errors of $\bo{u}-\up{h}(\uvec{u}_h)$, $\sjump{\bo{u}}-\sjump{\uvec{u}_h}_h$, $\nabla\bo{u}-\nabla\up{h}({\uvec{u}_h})$,\text{ and }$\lambda_{\bo{n}}-\lambda_{h,\bo{n}}$ versus the cubic root of the number of cells, for (a) Cartesian, (b) tetrahedral, and (c) Hexa-cut mesh families.}
\label{fig:6.2}
\end{figure}
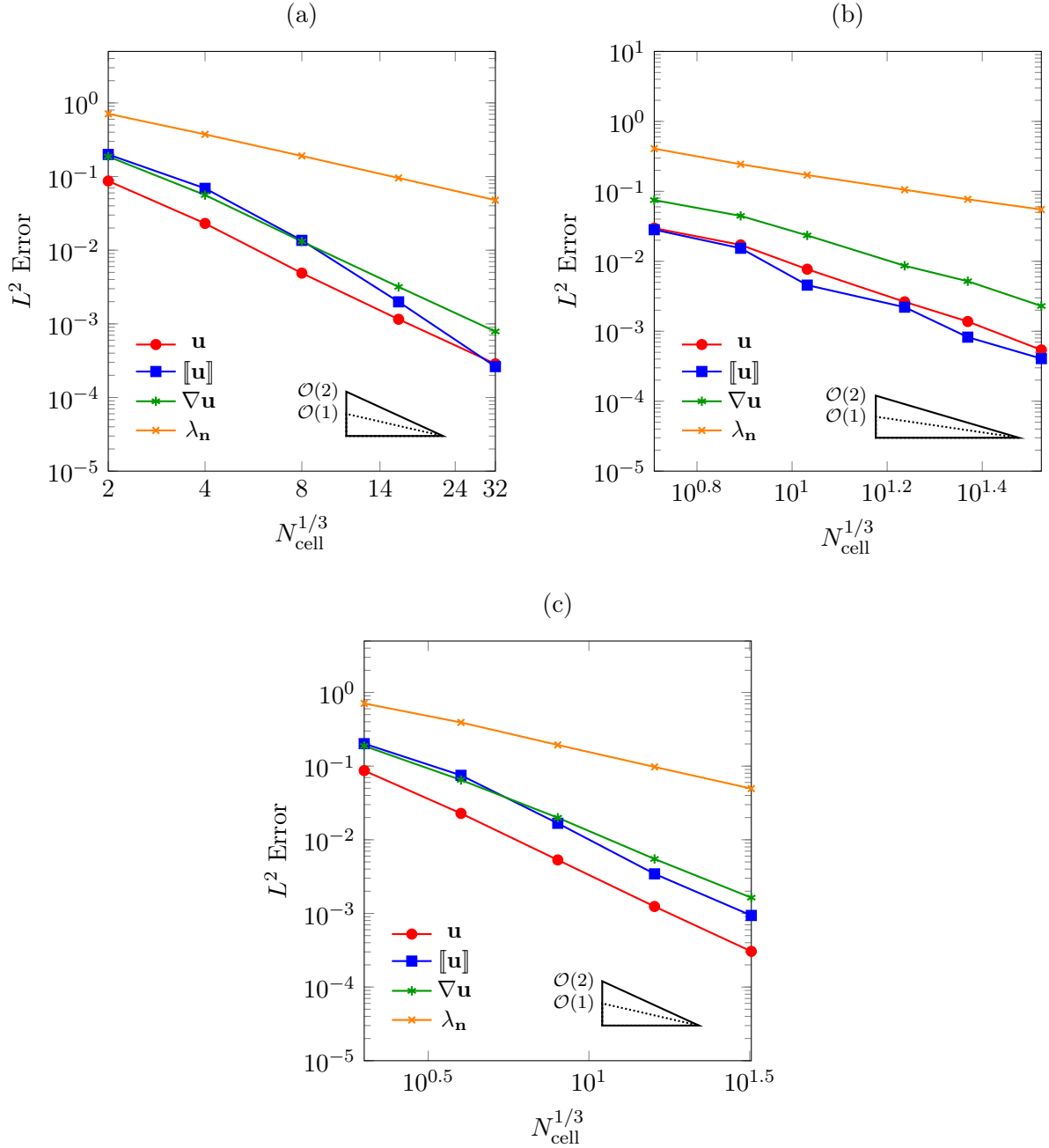

\begin{figure}[ht]
\centering
\begin{tikzpicture}

\begin{groupplot}[
  group style={
    group size=2 by 1,
    horizontal sep=2.4cm,
    vertical sep=1.8cm,
  },
  width=0.45\textwidth,
  height=0.42\textwidth,
  xlabel={Degrees of Freedom},
  ylabel={$L^2$ Error},
  ymode=log, xmode=log,
  legend style={
    at={(0.05,0.05)},
    anchor=south west,
    font=\small,
    draw=none, fill=none
  },
  every axis plot/.append style={thick, mark size=2pt}
]

\nextgroupplot[
    xmin=276, xmax=159654,
    ymin=3e-3, ymax=1.0,
    title={(a)}
]
\addplot[red, mark=*] coordinates {(276,0.61161119644314288) (831, 0.38884262267799291) (3450,0.23653459218442574) (8517,0.16701573411924253) (17169, 0.12910370518713715) (29742,0.10487672307848175)};
\addlegendentry{$\nabla\bo{u}\text{(lowest order)}$}
\addplot[blue, mark=square*] coordinates {(1518,1.452727e-01) (5526,7.015395e-02) (14700,4.007220e-02) (62784,1.514169e-02) (159654,8.268148e-03)};
\addlegendentry{$\nabla\bo{u}$}

\nextgroupplot[
    xmin=276, xmax=159654,
    ymin=3e-3, ymax=0.5,
    title={(b)}
    ]
\addplot[red, mark=*] coordinates {(276, 0.33687371211168049) (831, 0.21908415669739850) (1827,0.15243830615071899) (3450,0.12723724838566694) (5631,0.10025026489527322) (8517,8.6164671435958101E-002) (11964,7.4705570395644377E-002) (17169, 6.6126303860106914E-002) (29742,5.3074672806756526E-002)};
\addlegendentry{$\nabla\bo{u}\text{(lowest order)}$}
\addplot[blue, mark=square*] coordinates {(1518,7.495378e-02) (5526,4.458098e-02) (14700,2.343235e-02) (62784,8.633359e-03) (159654,5.183450e-03)};
\addlegendentry{$\nabla\bo{u}$}

\end{groupplot}
\end{tikzpicture}
\caption{((a) Test case from Section \ref{subsec.6.1} and (b) Test case from Section \ref{subsec.6.2}). Comparison of $L^2$-norm of $\nabla\bo{u}-\nabla\up{h}(\uvec{u}_h)$ of the lowest-order method \cite{jhr} with the current higher-order scheme for a tetrahedral mesh.  }
\label{fig:6.3}
\end{figure}
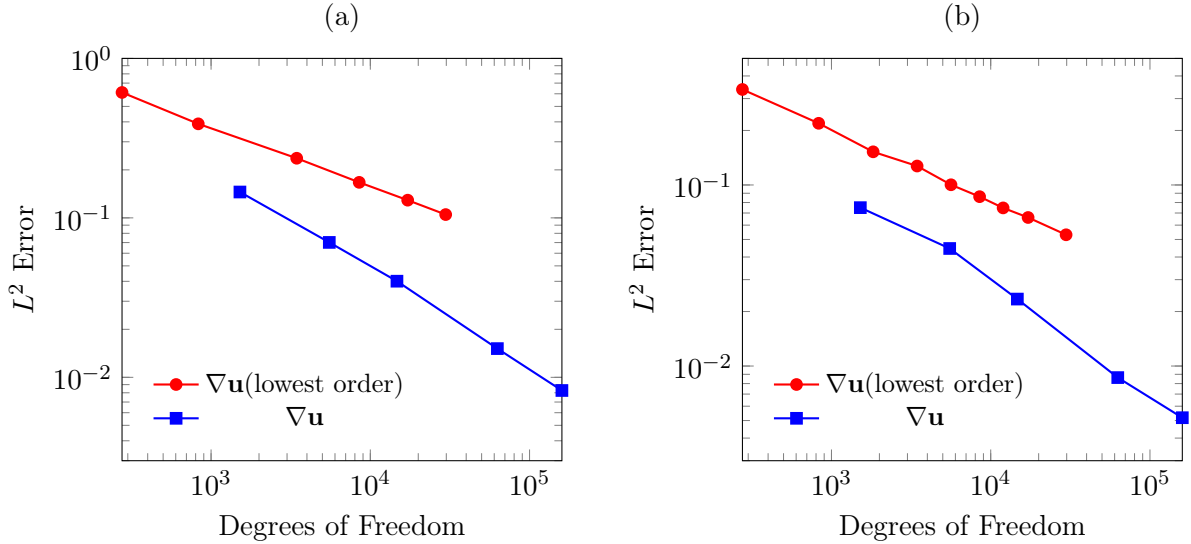

\subsection{3D manufactured solution to test incompressibility for the Tresca friction model}\label{subsec.6.3}
The geometry and mesh are the same as in the previous test cases, but the Lamé coefficient are now $G = 1$, $L\in\{1,10^4,10^6\}$, and the Tresca threshold $g$ is set to $1/L$. The exact solution is defined by
\begin{align*}
\bo{u}(x,y,z) =
\begin{cases}\begin{pmatrix}
x^3(\cos(y)+\sin(z)) \\
-3x^2\sin(y) \\
3x^2\cos(z)
\end{pmatrix}+\frac{1}{L}
\begin{pmatrix}
h(x,y)P(z) - y \\
P(z) \\
x^2 P(z)
\end{pmatrix} & \text{if } z \ge 0, \\[1.2em]
\begin{pmatrix}
x^3(\cos(y)+\sin(z)) \\
-3x^2\sin(y) \\
3x^2\cos(z)
\end{pmatrix}+\frac{1}{L}
\begin{pmatrix}
h(x,y)Q(z) - y \\
2Q(z) \\
x^2 Q(z)
\end{pmatrix} & \text{if } z < 0,~ x < 0, \\[1.2em]
\begin{pmatrix}
x^3(\cos(y)+\sin(z)) \\
-3x^2\sin(y) \\
3x^2\cos(z)
\end{pmatrix}+\frac{1}{L}
\begin{pmatrix}
h(x,y)Q(z) - y \\
Q(z) \\
x^2 Q(z)
\end{pmatrix} & \text{if } z < 0,~ x \ge 0,
\end{cases}
\end{align*}
with $h(x,y) = -\sin(x)\cos(y)$, $P(z) = z^2$, and $Q(z) = z^2/4$, 
is designed to satisfy the Tresca frictional-contact conditions at the matrix–fracture interface $\Gamma$.
The right-hand side $\mathbf{f} = -\operatorname{div}\boldsymbol{\sigma}(\mathbf{u})$ and the Dirichlet boundary conditions on $\partial\Omega$ are deduced from $\bo{u}$. As in section \ref{subsec.6.2}, we have that the fracture $\Gamma$ is in slippy-contact for $z<0$ and sticky-contact for $z>0$. Convergence of the scheme is examined on a family of meshes, with the convergence rates for each value of $L$ is reported in Table \ref{table:6.1}. As before we notice an order 2 convergence on the gradient but, more importantly, we see that the magnitude of the error remains the same no matter how high the second Lamé coefficient becomes. This demonstrate the fact that the scheme is locking-free in the quasi-incompressible limit, which is consistent with our error estimates (see Remark \ref{rem:error.estimates}). We further investigates the conditioning of the method in the nearly incompressible regime. Keeping Young's modulus $E=1$, we vary the Poisson ratio $\nu$ towards $0.5$ on the third tetrahedral mesh (with size $h=$6.75e-01) and the third hexa-cut mesh (with size $h=$5.58e-01). Figure \ref{figure:conditioning} reports the condition number of the global stiffness matrix, and shows that it seems to grow almost linearly with respect to $1/(0.5-\nu)$.

\begin{table}[ht]
\centering
\begin{subtable}{1\textwidth}
\centering
\scriptsize
\begin{tabular}{|c|c|c|c|c|c|c|}
\hline
\multirow{2}{*}{Mesh Size} & \multicolumn{2}{|c|}{$L=1$} & \multicolumn{2}{|c|}{$L=10^4$}& \multicolumn{2}{|c|}{$L=10^6$} \\ \cline{2-7}
& E & Order & E& Order& E& Order \\ \hline
1.73e+00 & 1.73e-01 & -- & 1.94e-01 & -- & 1.94e-01 & -- \\ \hline
8.66e-01 & 4.55e-02 & 1.93 & 5.06e-02 & 1.94 & 5.06e-02 & 1.94 \\ \hline
4.33e-01 & 1.07e-02 & 2.08 & 1.21e-02  & 2.06 & 1.21e-02 & 2.06 \\ \hline
2.16e-01 & 2.63e-03 & 2.04 & 2.99e-03 & 2.02 & 2.99e-03 & 2.02 \\ \hline
1.08e-01 & 6.53e-04 & 2.01 & 7.44e-04 & 2.01 & 7.44e-04 & 2.01 \\ \hline
\end{tabular}
\caption{Cartesian Mesh}
\end{subtable}

\vspace{1em}

\begin{subtable}{\textwidth}
\centering
\scriptsize
\begin{tabular}{|c|c|c|c|c|c|c|}
\hline
\multirow{2}{*}{Mesh Size} & \multicolumn{2}{|c|}{$L=1$} & \multicolumn{2}{|c|}{$L=10^4$}& \multicolumn{2}{|c|}{$L=10^6$} \\ \cline{2-7}
                             & E & Order & E& Order& E& Order \\ \hline
1.275e+00 & 6.42e-02 & -- & 8.39e-02 & -- & 8.40e-02 & -- \\ \hline
1.05e+00 & 2.89e-02 & 4.14 & 3.45e-02 & 4.61 & 3.45e-02 & 4.61 \\ \hline
6.75e-01 & 1.61e-02 & 1.31 & 2.00e-02  & 1.23 & 2.00e-02 & 1.23 \\ \hline
4.12e-01 & 5.48e-03 & 2.20 & 6.36e-03 & 2.33 & 6.36e-03  & 2.33 \\ \hline
2.95e-01  & 2.97e-03 & 1.84 & 3.29e-03 & 1.98 & 3.29e-03 & 1.98 \\ \hline
\end{tabular}
\caption{Tetrahedral Mesh}
\end{subtable}

\vspace{1em}

\begin{subtable}{\textwidth}
\centering
\scriptsize
\begin{tabular}{|c|c|c|c|c|c|c|}
\hline
\multirow{2}{*}{Mesh Size} & \multicolumn{2}{|c|}{$L=1$} & \multicolumn{2}{|c|}{$L=10^4$}& \multicolumn{2}{|c|}{$L=10^6$} \\ \cline{2-7}
                             & E & Order & E& Order& E& Order \\ \hline
1.82e+00 & 2.08e-01 & -- & 2.27e-01 & -- & 2.27e-01 & -- \\ \hline
1.01e+00 & 6.48e-02 & 1.99 & 7.16e-02 & 1.97 & 7.16e-02 & 1.97 \\ \hline
5.58e-01 & 1.65e-02 & 2.30 & 1.86e-02  & 2.26 & 1.86e-02 & 2.26 \\ \hline
2.86e-01 & 4.34e-03 & 2.01 & 4.86e-03 & 2.02 & 4.86e-03 & 2.02 \\ \hline
1.45e-01 & 1.14e-03 & 1.98 & 1.24e-03 & 2.02 & 1.24e-03 & 2.02 \\ \hline
\end{tabular}
\caption{Hexa-cut Mesh}
\end{subtable}
\caption{Order of convergence of $\text{E}=\lVert \nabla \bo{u}-\nabla\up{h}(\uvec{\bo{u}}_h) \rVert_{L^2(\cT_h)}$ with different values of $L$ for different mesh types. Test case of Section \ref{subsec.6.3}.}\label{table:6.1}
\end{table}


\begin{figure}[ht]
\centering
  \begin{tikzpicture}
  \begin{axis}[
    width=0.45\textwidth,
    height=0.48\textwidth,
    xlabel={$\nu$},
    ylabel={Condition Number},
    xlabel={$0.5 - \nu$},
    xmode=log,
    ymode=log,
    x dir=reverse,
    xmin=1e-5, xmax=1e-1,
    every axis plot/.append style={thick, mark size=2pt},
    ymin=400, ymax=1e+13,
    legend style={
    at={(0.05,0.05)},
    anchor=south west,
    font=\small,
    draw=none, fill=none
  }
  ]
\addplot[red, mark=*] coordinates {(0.1,1.24222e+06) (0.01,1.26918e+07) (0.0001,7.36797e+10) (0.00001,1.21285e+12)};
\addlegendentry{Tetrahedral}
\addplot[blue, mark=square*] coordinates {(0.1,84451) (0.01,7.87693e+05) (0.0001,4.98356e+09) (0.00001,1.80707e+11)};
\addlegendentry{Hexa-cut}
  \end{axis}
  \end{tikzpicture}
\caption{Condition number vs $\nu$. Test case of Section \ref{subsec.6.3}.}\label{figure:conditioning}
\end{figure}

\subsection{3D Discrete Fracture Matrix~(DFM) model with intersecting fractures}\label{subsec.6.4}

We consider the three-dimensional domain $\Omega=(0,1)^3$ that contains the fracture network depicted in Figure \ref{fig:mesh_6k}. The meshes are produced using \texttt{GMesh}, which provides tetrahedral meshes compliant with the fracture network \cite{gmsh}. The material is assumed to be isotropic and homogeneous with Young's modulus $E=4\times 10^9$ and Poisson's ratio $\nu=0.2$, yielding a shear modulus $L=\frac{5}{3}\times 10^9$. The Tresca model governs the friction model with a constant friction coefficient $g=4\times 10^6$. Dirichlet boundary conditions are prescribed on the top and bottom surfaces: $\bo{u}=\bo{0}$ at $z=0$ and $\bo{u}=[0.0015,-0.0015,-0.002]$ at $z=1$. Homogeneous Neumann conditions are imposed on the lateral boundaries. In the numerical experiments, the normal and tangential penalty parameters are set to $\beta_{h,\bo{n}}=10^6=\beta_{h,\tau}.$ 
In practice, the nonlinear convergence is only weakly sensitive to the choice of the parameters $\beta_{h,\bo{n}}$ and $\beta_{h,\tau}$, which are both taken as a small fraction, typically $10^{-3}$, of the first Lam\'e coefficient ${E \over {1 + \nu}}$. A high degree of robustness of the nonlinear convergence is observed for values ranging from about $10^{-3}{E \over {1 + \nu}}$ to ${E \over {1 + \nu}}$. The natural scaling with respect to the mesh would be $1/h_\sigma$, where $h_\sigma$ denotes the diameter of the face $\sigma$, but we have noted this scaling does not appear to be necessary in the simulations. For a numerical investigation of the sensitivity of the nonlinear convergence to the choice of these parameters, we refer to \cite{BDMP:21}.

In this study, we focus on the qualitative behavior of the solution. Figure \ref{fig:contact_normal}(a) displays the computed contact state, while Figure \ref{fig:contact_normal}(b) displays the normal component of the displacement jump across the fracture network. The superior performance of the higher-order scheme is demonstrated in Figure \ref{fig:normal_fracture2-4} by the mean of the normal component of the discrete Lagrange multiplier $\lambda_{h, \bo{n}}$ on fractures F2--F4~(see Figure \ref{fig:contact_normal}(b)). As we refine the mesh, we expect the approximate mean to converge to the exact mean; both schemes identify, at convergence, the same value for F2, F3, and F4, but the high-order scheme already gives this value on the coarsest meshes, and with fewer degrees of freedom than the low-order scheme. A comparison of the Newton iteration counts required to reach a tolerance of $10^{-12}$
 for the lowest-order method \cite{jhr}, and the present scheme is reported in Table~\ref{table:unknown_solution}. Furthermore, Figure~\ref{fig:traction_normal} demonstrates that the discrete normal traction produced by the current scheme is considerably smoother than that obtained with the lowest-order approach \cite{jhr}.
This establishes the higher-order scheme as both more accurate and computationally efficient. This advantage extends to fracture F1, where the higher-order method yields a more accurate mean normal displacement jump (Figure \ref{fig:normal_fracture1}(a)) and a superior resolution of the contact transition zone (Figure \ref{fig:contact_normal}(a)), as evidenced by the more precise mean $\lambda_{h,\bo{n}}$ profile in Figure \ref{fig:normal_fracture1}(b).

\begin{table}
\centering
\scriptsize
\begin{tabular}{|c|c|c|c|c|c|c|}
\hline
 & Mesh &1 &2 &3 &4 &5 \\ \hline
 \multirow{2}{*}{Lowest-order method \cite{jhr}} & Nb Newton iterations &12 & 12&12 & 12 &13 \\ \cline{2-7} 
 & Degrees of Freedom&4092 & 20592& 52665&85566 &127047\\ \hline
  \multirow{2}{*}{Present method} &Nb Newton iterations &5 &7 & 7 & 7 & 7 \\ \cline{2-7} 
 &Degrees of Freedom & 32463&205515 & 609507&1087011 &1703745\\ \hline
\end{tabular}
\caption{Test case from Section \ref{subsec.6.4}: Comparison between the lowest-order method \cite{jhr} and the current higher-order scheme on the number of Newton iterations required to converge within the desired tolerance of $10^{-12}$.}\label{table:unknown_solution}
\end{table}
\begin{figure}[ht]
\centering
\includegraphics[width=0.5\textwidth]{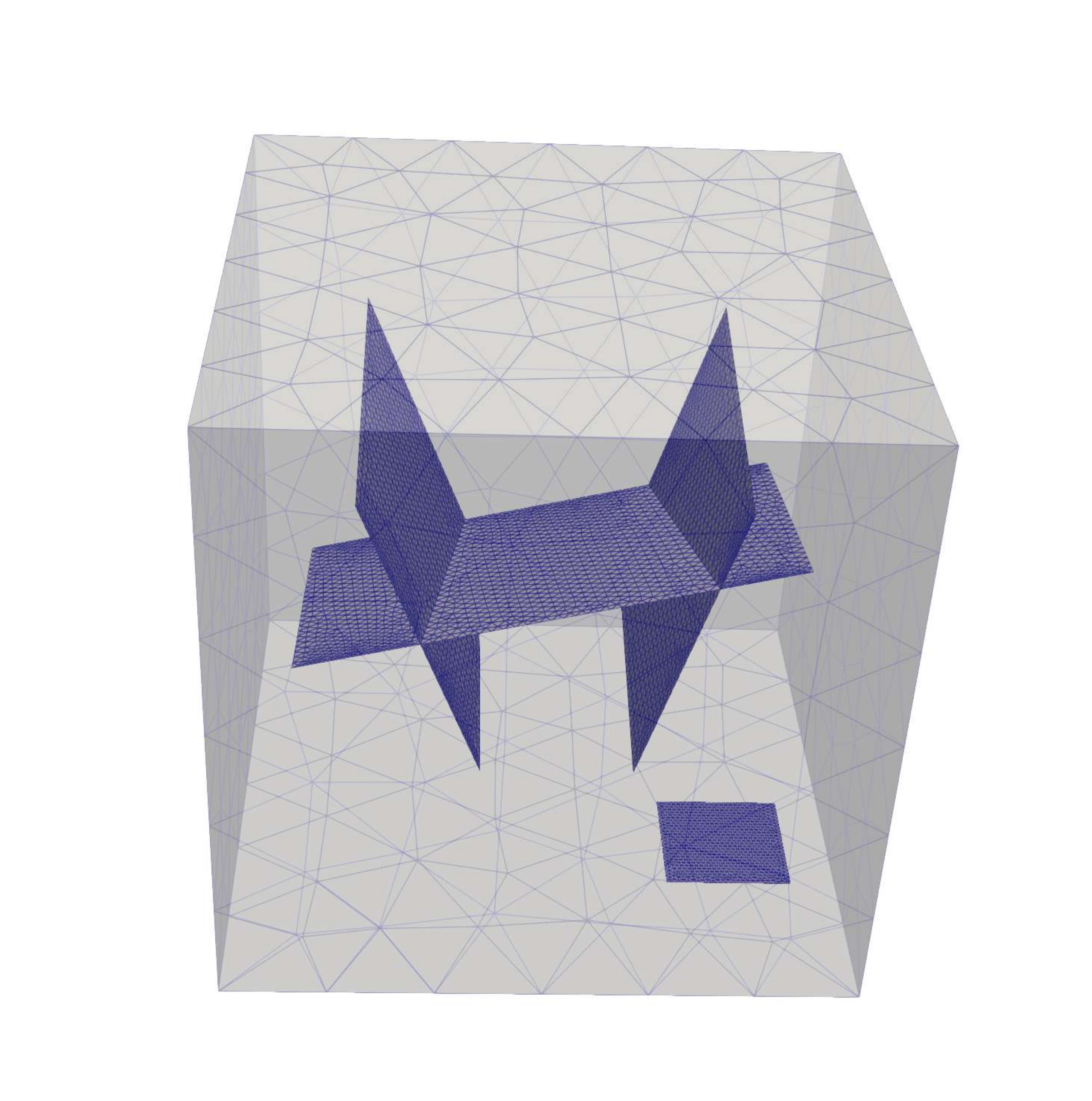}
\caption{Fracture network for the test case in Section~\ref{subsec.6.4}}
\label{fig:mesh_6k}  
\end{figure}
\begin{figure}[ht]
\begin{subfigure}{0.45\textwidth}
  \includegraphics[width=\textwidth]{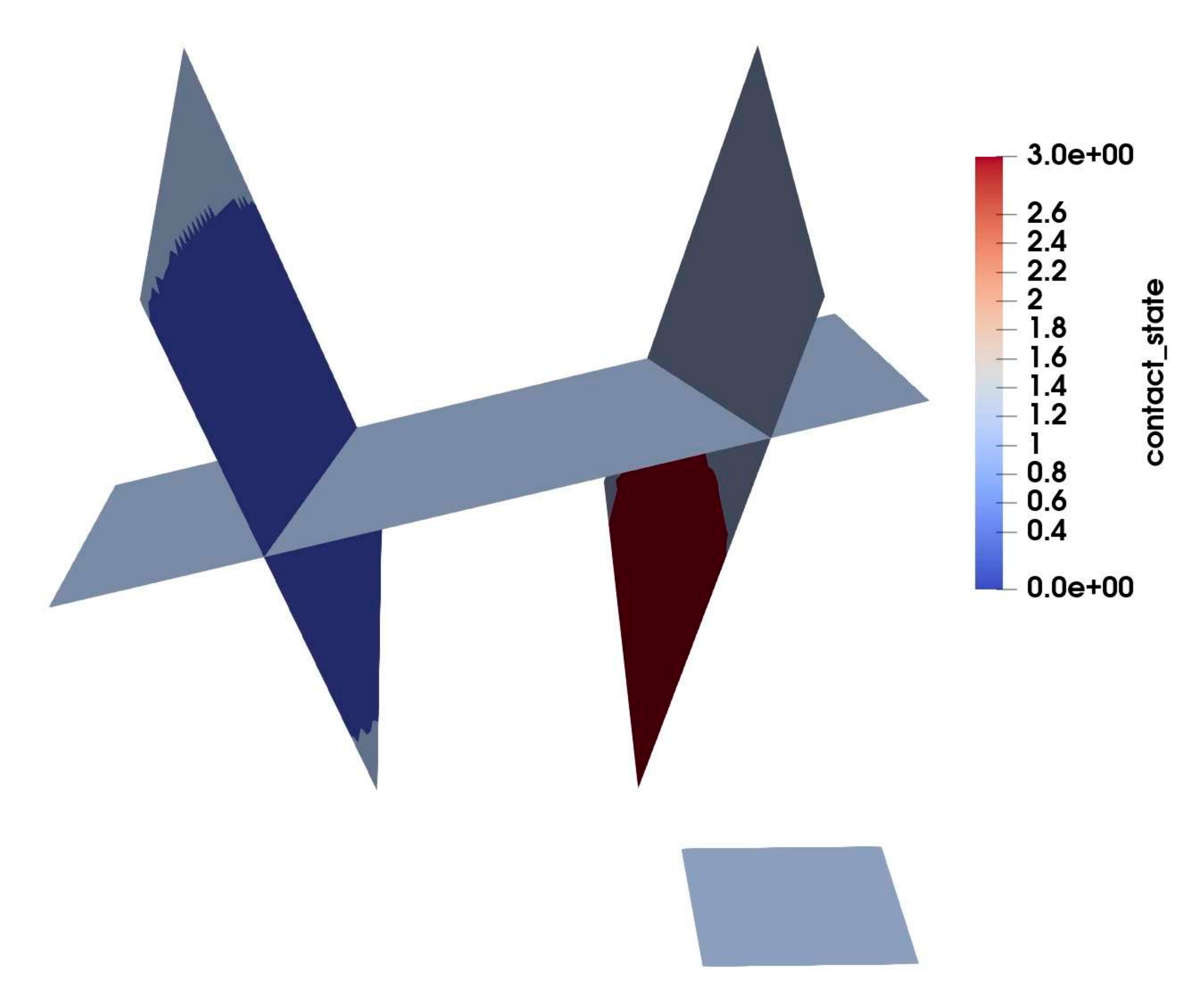}
  \caption{}
\end{subfigure}
\begin{subfigure}{0.45\textwidth}
  \includegraphics[width=\textwidth]{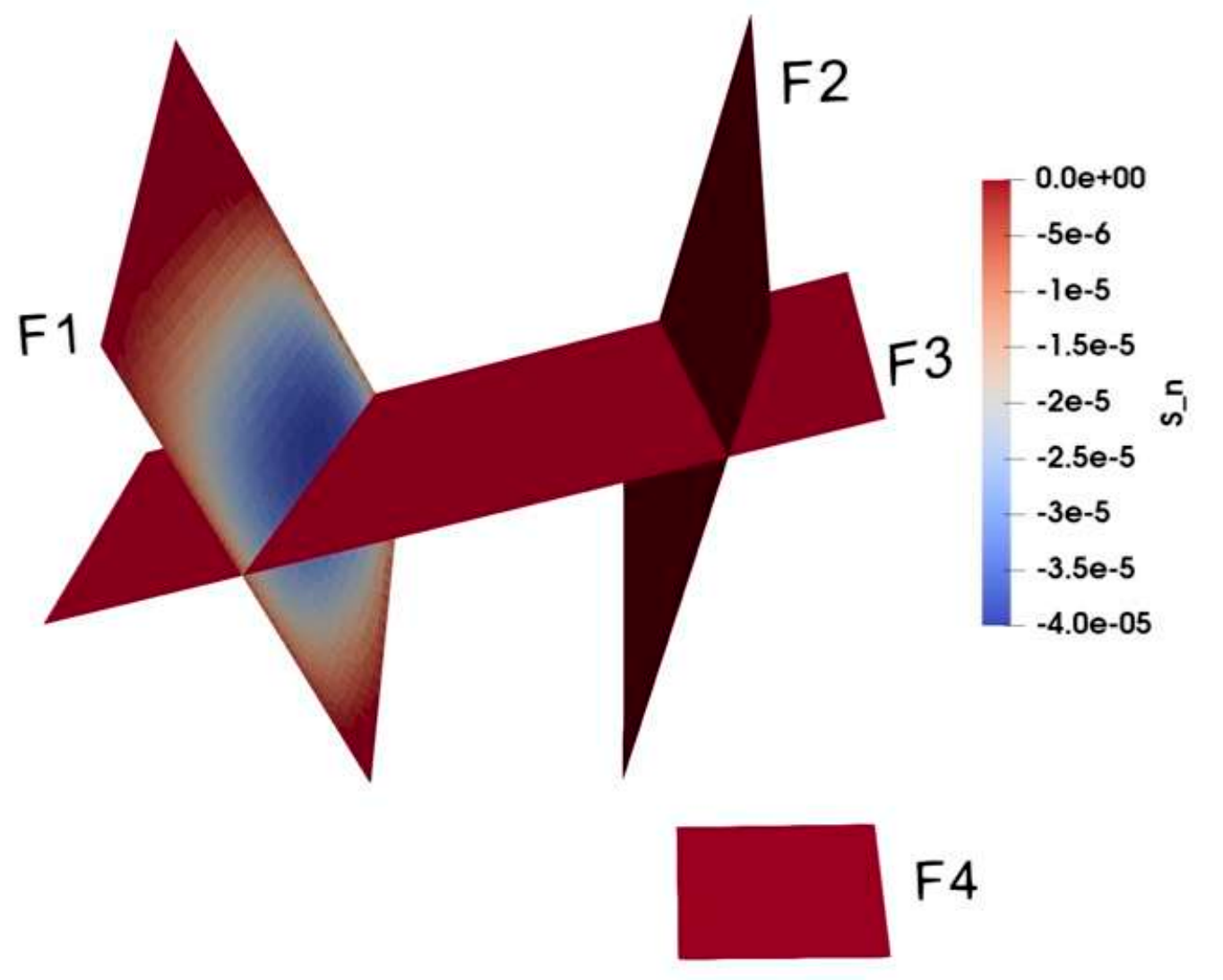}
  \caption{}
\end{subfigure}
\caption{Results for the test case in Section~\ref{subsec.6.4}: (a) contact state classification, where the values indicate: 0 if $\left|\blambda_{h,\tau}\right| < g$ and $\sjump{\uvec{u}_h}_{h,\bo{n}} <0$; 1 if $\left|\blambda_{h,\tau}\right| < g$ and $\sjump{\uvec{u}_h}_{h,\bo{n}} =0$; 2 if $\left|\blambda_{h,\tau}\right| = g$ and $\sjump{\uvec{u}_h}_{h,\bo{n}} <0$; and 3 if $\left|\blambda_{h,\tau}\right| =g$ and $\sjump{\uvec{u}_h}_{h,\bo{n}} =0$; and (b) normal displacement jump obtained using the DDR discretisation with 123k cells and 8.8k fracture faces.}
\label{fig:contact_normal} 
\end{figure}

\begin{figure}[ht]
\begin{subfigure}{0.35\textwidth}
  \includegraphics[width=\textwidth]{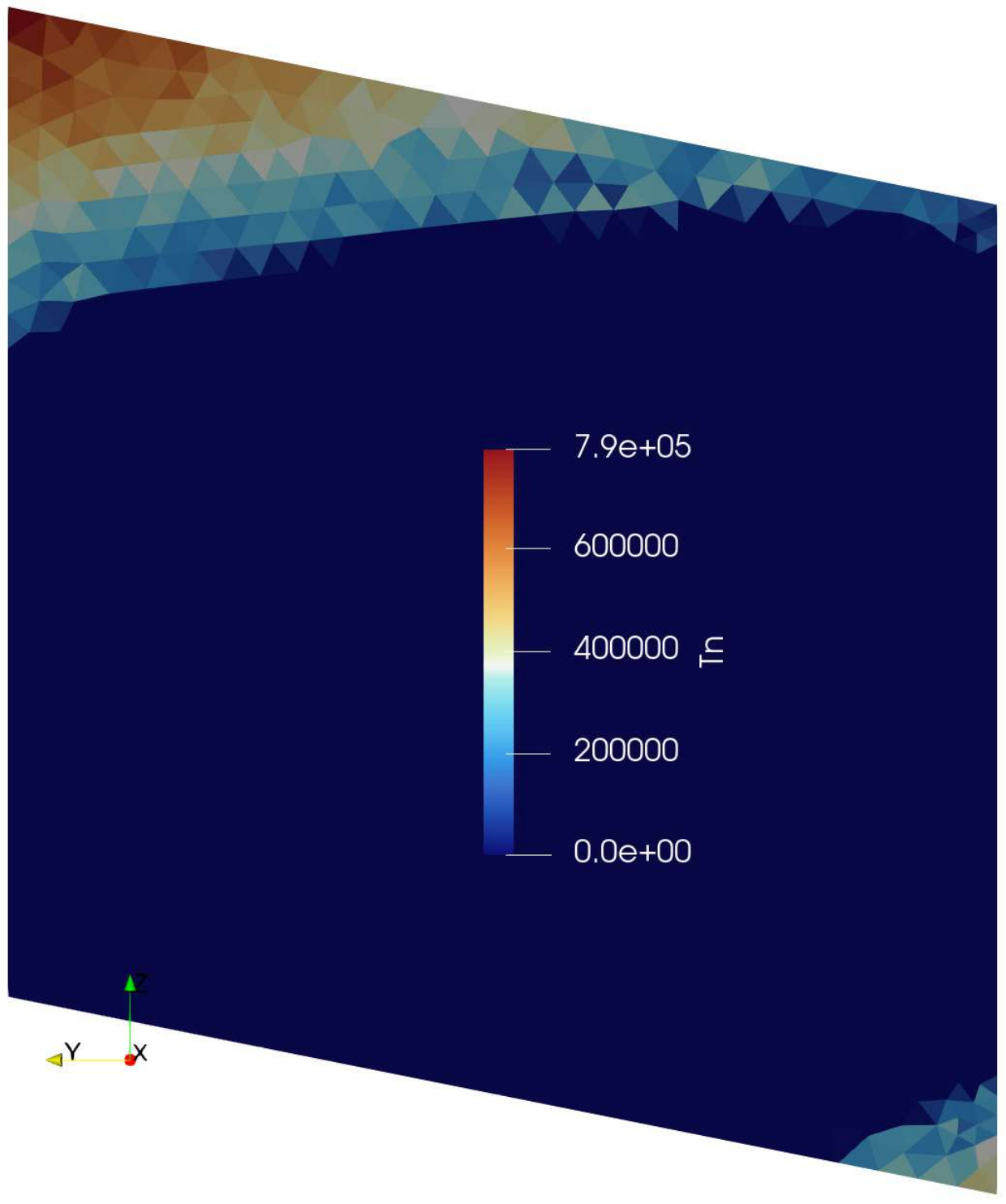}
  \caption{}
\end{subfigure}
\begin{subfigure}{0.35\textwidth}
  \includegraphics[width=\textwidth]{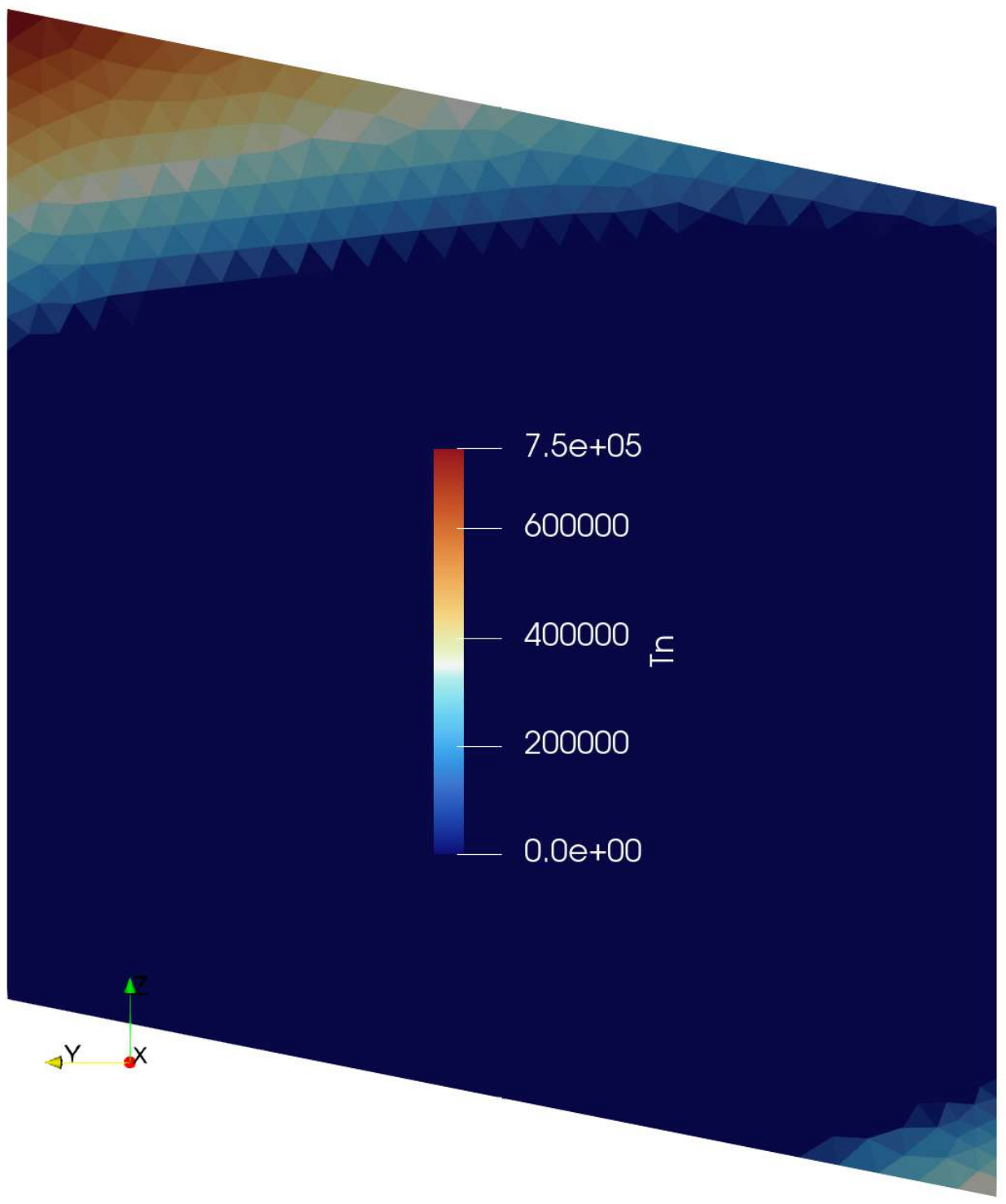}
  \caption{}
\end{subfigure}
\caption{Test case from Section \ref{subsec.6.4}: Plot of normal component of discrete Lagrange multiplier:  (a) lowest-order method \cite{jhr} and (b) the current higher-order scheme for fracture F1 on a tetrahedral mesh.}
\label{fig:traction_normal} 
\end{figure}

\begin{figure}[!ht]
\centering
\begin{tikzpicture}

\begin{groupplot}[
  group style={
    group size=2 by 1,
    horizontal sep=2.4cm,
    vertical sep=1.8cm,
  },
  width=0.45\textwidth,
  height=0.42\textwidth,
  xlabel={Degrees of Freedom},
  xmode=log,
  legend style={
    at={(0.05,0.05)},
    anchor=south west,
    font=\small,
    draw=none, fill=none
  },
  every axis plot/.append style={thick, mark size=2pt}
]

\nextgroupplot[
    xmin=4092, xmax=1015547,
    ymin=-3e-5, ymax=-1e-7,
    title={(a)}
]
\addplot[red, mark=*] coordinates { (4092,-9.8685746435799798E-006) (8031,-9.7280384833157034E-006) (13203,-9.7863271735873665E-006) (20592, -1.1954258652024626E-005) (29367,-1.2365061896298013E-005) (39483,-1.3111056573384876E-005) (52665,-1.3495769209905402E-005) (67329,-1.3588452566626169E-005) (85566,-1.4732352025060066E-005) (104328,-1.4906622970850813E-005) (127047,-1.5017332417664489E-005) (193113, -1.5791665243168232E-005) (261102, -1.5882868954957554E-005) (305838,-1.5949244517750590E-005) (357306, -1.5965315260543715E-005)};
\addlegendentry{$\sjump{\uvec{u}_h}_{h,\bo{n}}\text{(lowest order)}$}
\addplot[blue, mark=square*] coordinates { (32463,-1.754821483105e-05) (69777,-1.714964055313e-05) (122997,-1.709109330011e-05) (205515, -1.714390838323e-05) (306981,-1.714674251832e-05) (430479,-1.716270541900e-05) (609507,-1.715401390793e-05) (815547,-1.718849592655e-05)};
\addlegendentry{$\sjump{\uvec{u}_h}_{h,\bo{n}}$}

\nextgroupplot[
    xmin=4092, xmax=1015547,
    ymin=1e4, ymax=3e5,
    ymode=log,
    title={(b)}
]
\addplot[red, mark=*] coordinates { (4092,51270.052776872348) (8031,76425.545231454817) (13203,83089.382355386770) (20592, 56144.125942610292) (29367,65509.388304796135) (39483,58638.680016498540) (52665,60706.289757865430) (67329,55164.320884697256) (85566,52117.249809083951) (104328,52422.415322112429) (127047,50741.098554447104) (193113, 45816.304585659542) (261102, 45204.491707195106 ) (305838,44781.641980256863) (357306, 44614.305725036422 )};
\addlegendentry{$\lambda_{h,\bo{n}}\text{ (lowest order)}$}
\addplot[blue, mark=square*] coordinates {(32463,2.732659988102e+04) (69777,3.360658140641e+04) (122997,4.239898011932e+04) (205515, 4.287896011842e+04) (306981,4.331842177616e+04) (430479,4.354265923547e+04) (609507,4.372961658076e+04) (815547,4.363108041182e+04)};
\addlegendentry{$\lambda_{h,\bo{n}}$}
\end{groupplot}
\end{tikzpicture}
\caption{Test case from Section \ref{subsec.6.4}: Comparison between the lowest-order method \cite{jhr} and the current higher-order scheme for fracture F1 on a tetrahedral mesh: (a) Mean value of normal displacement jump and (b) Mean value of normal component of the discrete Lagrange multiplier.}
\label{fig:normal_fracture1}
\end{figure}
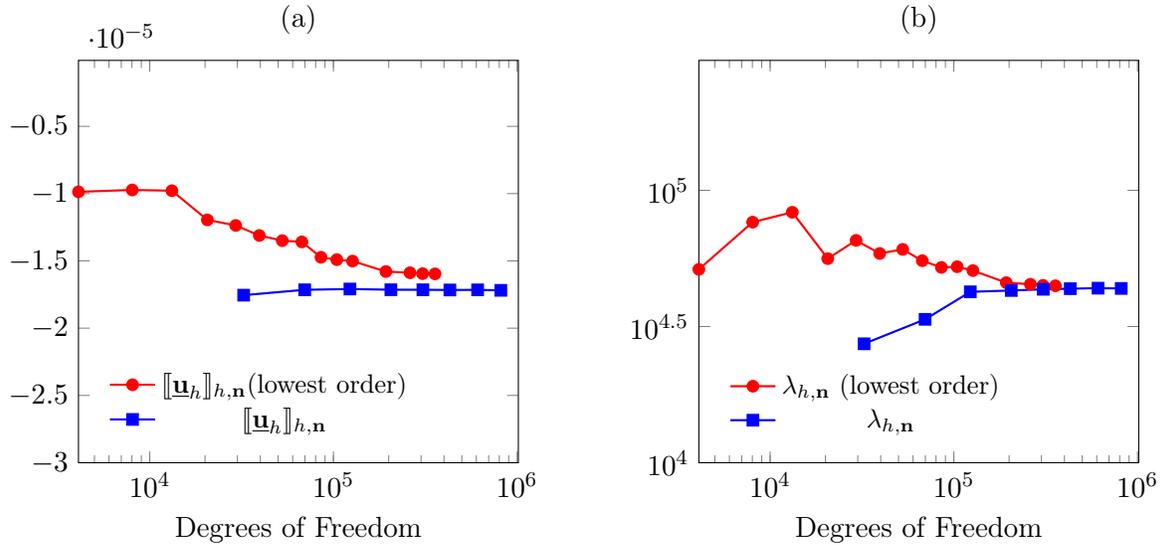

\begin{figure}[ht]
\centering
\begin{tikzpicture}

\begin{groupplot}[
  group style={
    group size=2 by 1,
    horizontal sep=2.4cm,
    vertical sep=1.8cm,
  },
  width=0.45\textwidth,
  height=0.42\textwidth,
  xlabel={Degrees of Freedom},
  ymode=log, xmode=log,
  legend style={
    at={(0.05,0.05)},
    anchor=south west,
    font=\small,
    draw=none, fill=none
  },
  every axis plot/.append style={thick, mark size=2pt}
]

\nextgroupplot[
     xmin=4092, xmax=1015547,
    ymin=0.9e6, ymax=3e6,
    title={(a)}
]
\addplot[red, mark=*] coordinates { (4092,2171096.5779527184) (8031,2034099.4534486097 ) (13203,1970424.3211318683) (20592, 1978156.7430872868) (29367,1936641.9566928442) (39483,1931600.9880419080) (52665,1928107.1550206488) (67329,1927117.9137429339) (85566,1917114.3138301524) (104328,1914940.9023684901) (127047,1914163.4588728605) (193113, 1892210.9880421869) (261102, 1888521.4076513657) (305838,1886806.5991800867) (357306, 1885950.9071255969)};
\addlegendentry{$\lambda_{h,\bo{n}}\text{ (lowest order)}$}
\addplot[blue, mark=square*] coordinates {(32463,1.860224591625e+06) (69777,1.872268898320e+06) (122997,1.878143312935e+06) (205515, 1.878383170878e+06) (306981,1.878893150699e+06) (430479,1.879398898075e+06) (609507,1.879584119375e+06) (815547,1.879480345469e+06)};
\addlegendentry{$\lambda_{h,\bo{n}}$}

\nextgroupplot[
     xmin=4092, xmax=1015547,
    ymin=8.5e6, ymax=1e7,
    title={(b)}
]
\addplot[red, mark=*] coordinates { (4092,9389516.9930326529) (8031,9252165.9330312461) (13203,9191500.8203133401) (20592, 9154377.5559180975) (29367,9134343.1049237568 ) (39483,9113588.3015194181) (52665,9103196.0795311760) (67329,9093077.8503232263 ) (85566,9084778.8346578833) (104328,9073950.9259710368) (127047,9068451.7660144512) (193113, 9036031.4670693167) (261102, 9027948.2669025287) (305838,9026454.8340098467) (357306, 9024399.1835238803)};
\addlegendentry{$\lambda_{h,\bo{n}}\text{ (lowest order)}$}
\addplot[blue, mark=square*] coordinates {(32463,9.014873810310e+06) (69777,9.015958304712e+06) (122997,9.018153770937e+06) (205515, 9.016521787060e+06) (306981,9.016155788724e+06) (430479,9.015658425619e+06) (609507,9.015588977087e+06) (815547,9.015177327698e+06)};
\addlegendentry{$\lambda_{h,\bo{n}}$}
\end{groupplot}

\node at ($(group c1r1.south)!0.4!(group c2r1.south) - (0,6.0cm)$) {
  \begin{tikzpicture}
  \begin{axis}[
    width=0.45\textwidth,
    height=0.42\textwidth,
    xlabel={Degrees of Freedom},
    ymode=log, xmode=log,
    every axis plot/.append style={thick, mark size=2pt},
     xmin=4092, xmax=1015547,
    ymin=1e7, ymax=1.3e7,
    title={(c)},
    legend style={
    at={(0.05,0.05)},
    anchor=south west,
    font=\small,
    draw=none, fill=none
  }
  ]
\addplot[red, mark=*] coordinates { (4092,10574230.497082464) (8031,11039547.250089059 ) (13203,11277804.660379954) (20592, 11347644.738935452) (29367,11387123.383637752) (39483,11437455.692985879 ) (52665,11435854.463999324 ) (67329,11444511.313130355  ) (85566,11471921.485802082) (104328,11474962.722457238) (127047,11488131.678192582) (193113, 11518473.528538164) (261102, 11529103.990502428) (305838,11534050.908261105) (357306, 11536009.347686866 )};
\addlegendentry{$\lambda_{h,\bo{n}}\text{ (lowest order)}$}
\addplot[blue, mark=square*] coordinates {(32463,1.151488543327e+07) (69777,1.148898924947e+07) (122997,1.153300526883e+07) (205515, 1.154100315148e+07) (306981,1.154501609029e+07) (430479,1.154953586803e+07) (609507,1.155069582982e+07) (815547,1.155000971570e+07)};
\addlegendentry{$\lambda_{h,\bo{n}}$}
  \end{axis}
  \end{tikzpicture}
};

\end{tikzpicture}
\caption{Test case from Section \ref{subsec.6.4}: Mean value of normal component of the discrete Lagrange multiplier for fractures F2, F3, and F4 are shown in (a), (b), and (c), respectively. Results compare the lowest-order method \cite{jhr} with the current higher-order scheme on a tetrahedral mesh.}
\label{fig:normal_fracture2-4}
\end{figure}



\FloatBarrier
\section*{Acknowledgments}

We acknowledge the funding of the European Union via the ERC Synergy, NEMESIS, project number 101115663.
Views and opinions expressed are, however, those of the authors only and do not necessarily reflect those of the European Union or the European Research Council Executive Agency.
Neither the European Union nor the granting authority can be held responsible for them.


\bibliographystyle{plain}
\bibliography{refs.bib}
\end{document}